\documentclass[11pt,leqno]{amsart}
\usepackage{amsmath,amssymb, amsfonts,latexsym,graphicx,amssymb,url, color}
\usepackage[pagebackref=true]{hyperref}
\usepackage{txfonts} 

\usepackage[hmargin=2.5 cm,vmargin=2.5 cm]{geometry}

\usepackage{color} 
\renewcommand{\div}{\mbox{div}\,}
\renewcommand{\div}{\mbox{div}\,}

\newcommand{\R}{{\mathbb R}} %%reals
\newcommand{\N}{{\mathbb N}}
\newcommand{\A}{\mathbb{A}}

\newcommand{\B}{\mathbb{B}}

\newcommand\norm[1]{\left\| #1\right\|}

\newcommand{\M}{{\mathcal M}}

\newcommand{\wei}[1]{\langle #1 \rangle}

\newcommand{\hmu}{{\hat{\mu}}}
\newcommand{\F}{{\mathbf F}}

\everymath{\displaystyle}
\newtheorem{theorem}{Theorem}[section]
\newtheorem{definition}[theorem]{Definition}
\newtheorem{remark}[theorem]{Remark}
\newtheorem{lemma}[theorem]{Lemma}
\newtheorem{proposition}[theorem]{Proposition}

\numberwithin{equation}{section}

\newcommand{\beq}{\begin{equation}}
\newcommand{\eeq}{\end{equation}}

\definecolor{darkred}{rgb}{.70,.12,.20}

\definecolor{darkgreen}{rgb}{.20,.52,.14}

\title[Weighted gradient estimates, degenerate Elliptic Equations] {Weighted $W^{1,p}$- estimates for weak solutions of degenerate elliptic equations with coefficients degenerate in one variable}
\author{Tadele Mengesha and Tuoc Phan}
\address{Department of Mathematics, University of Tennessee, Knoxville, 227 Ayres Hall, 1403 Circle Drive, Knoxville, TN 37996, U.S.A.}
\email{mengesha@math.utk.edu, phan@math.utk.edu}
\begin{document}
\begin{abstract} This paper studies the Sobolev regularity of weak solution of degenerate elliptic equations in divergence form $\textup{div}[\A(X) \nabla u] = \textup{div}[\F(X)]$, where $X = (x,y) \in \R^{n} \times \R$ . The coefficient matrix $\A(X)$ is a symmetric, measurable $(n+1) \times (n+1)$ matrix, and it could be degenerate or singular in the one dimensional $y$-variable as a weight function in the Muckenhoupt class $A_2$ of weights. Our results give weighted Sobolev regularity estimates of Calder\'{o}n-Zygmund type for weak solutions of this class of singular, degenerate equations. As an application of these estimates, we establish global Sobolev regularity estimates for solutions of the spectral fractional elliptic equation with measurable coefficients.
This result can be considered as the Sobolev counterpart of the recently established Schauder regularity theory of fractional elliptic equations. 
\end{abstract}

\maketitle

\section{Introduction}
This paper investigates Sobolev regularity theory for weak solutions of 
linear elliptic equations with measurable and degenerate coefficients 
\begin{equation} \label{introc-eqn}
\textup{\div}[\A(X) \nabla u(X)] = \textup{div}[\F(X)], \quad X = (x,y) \in \Omega \times (0,2),
\end{equation}
over some bounded domain $\Omega \subset \R^n$,  $n \in \N$, and with suitable boundary conditions. In  the equation \eqref{introc-eqn},  $\A$ is a given symmetric and  measurable $(n+1) \times (n+1)$ matrix, which can be either singular or degenerate in the $y$-variable in $\R^{n+1}$.  Essentially, we assume that the smallest and the largest eigenvalues of $\A$ behave proportionally as a weight in the $A_{2}$-Muckenhoupt class, which will be defined shortly. Precisely, we assume that there exist $\Lambda>0$ and a weight function $\mu : \R \rightarrow [0,\infty)$ in $A_{2}(\mathbb{R})$ such that
\begin{equation} \label{variable-degenerate}
\Lambda^{-1} |\xi|^2 \mu(y) \leq \wei{\A(X)\xi, \xi} \leq \Lambda |\xi|^2 \mu(y), \quad \text{for a.e.} \ X = (x,y) \in \Omega \times (0,2), \quad 
\forall \ \xi \in \R^{n+1}.
\end{equation}
The vector field $\F$ is an $(n+1)$-tuple of measurable functions.

This work is a continuation of  \cite{CMP}, where Calder\'{o}n-Zygmund type regularity estimates in weighted spaces are established  for weak solutions of degenerate elliptic equations. In \cite{CMP}, the matrix $\A$ is degenerate/singular in all directions of $X$, while the current work focuses in  the case that $\A$ is only singular or degenerate in the one dimensional $y$-variable of $X$ as in \eqref{variable-degenerate}.  
The motivations for studying equations \eqref{introc-eqn} with partially degenerate/singular coefficient $\A$ as in \eqref{variable-degenerate}  are twofold. The first motivation is 
 to extend the Calder\'{o}n-Zygmund type regularity estimates for uniformly elliptic equation, which holds for the wider class of partial \textup{BMO}/\textup{VMO} coefficients as documented in
  \cite{B-Kim, Byun-P, Dong-Kim-1, Dong-Kim, Kim-Krylov, Kim-Krylov-1, Krylov}, to the broader class of  degenerate equations.  This will be achieved by introducing an appropriate means of measuring mean oscillations that is compatible with the degeneracy of the coefficients. 
  Equations with degenerate coefficients appear in applications, see for example the model in mathematical finance  in \cite{Pop-1, Pop-2}. 
  We note  that equations of type \eqref{introc-eqn} with degenerate/singular coefficients $\A$ have been investigated extensively, see \cite{Fabes-1, Fabes, GW, MRW, Murthy-Stamp, NPS, Str, Stredulinsky, Turesson}, to cite a few, in connection with developing a  Schauder regularity theory. 
  In this work we develop the Sobolev counterpart. 
  %It is therefore fundamentally important to understand and develop the Sobolev counterpart theory for the degenerate equation \eqref{introc-eqn}. 
   The second motivation is to obtain Sobolev regularity estimates for solutions of some  (nonlocal) fractional elliptic equations with measurable coefficients. It turns out that some fractional elliptic operators can be obtained as Dirichlet-to-Neumann maps for degenerate elliptic equations in one more space dimension, see for example \cite{Caffa-Sil, T-Stinga2010, Caffa-Stinga, Cabre, Capella}. As a consequence, we obtain Sobolev estimates for solutions of  fractional elliptic equations  from estimate for solutions of degenerate elliptic equations in one more space dimension. 
This result can be considered as the Sobolev counterpart of the Schauder regularity theory for fractional elliptic equations that has recently been developed in \cite{Caffa-Sil, Caffa-Stinga}. 

We focus on the following two model problems where the degeneracy or singularity of $\A$ appears on the hyperplane $y =0$. The first one is the Neumann boundary value type problem
\begin{equation} \label{main-eqn}
\left\{
\begin{array}{cccl}
\text{div}[\A(X) \nabla u(X)] & = &\text{div}[{\bf F}], & \quad X = (x,y) \in Q_2: = B_2 \times (0,2), \\
\lim_{y\rightarrow 0^+}\wei{\A(X) \nabla u(X) - {\bf F}(X) , e_{n+1}} & = & f(x),  & \quad X= (x, y)\ \in B_2 \times\{0\},
\end{array}
\right.
\end{equation}
and the second is the mixed Dirichlet-Neumann boundary value problem over the half cylinder with base the half-ball 
\begin{equation} \label{main-eqn-flat-domain}
\left\{
\begin{array}{cccl}
\text{div}[\A(X) \nabla u(X)] & = &\text{div}[{\bf F}], & \quad X = (x,y) \in Q_2^+ :=B_2^+ \times (0,2), \\
u(X) &= & 0, & \quad X = (x,y) \in   T_2 \times (0,2),\\
\lim_{y\rightarrow 0^+}\wei{\A(X) \nabla u(X) - {\bf F}(X) , e_{n+1}} & = & f(x),  & \quad X= (x, y)\ \in B_2^+ \times\{0\}.
\end{array} \right.
\end{equation}
We explain the notation used in \eqref{main-eqn} and \eqref{main-eqn-flat-domain}.  The ball in $\R^n$ with radius $r>0$ and  centered at the origin is denoted by $B_r$. Its upper half ball is denoted by $B_r^+$:
\[
B_r^+ = \left\{ (x', x_n) \in B_r: x_n >0 \right\}, \quad \text{and}  \quad T_r = \left\{ x \in B_r: x= (x',0), \ x' \in \R^{n-1} \right\}.
\]
We also write $Q_r = B_r \times (0,r)$, and $Q_r^+ = B_r^+ \times (0,r)$. Moreover, $e_{n+1} = (0, 0, \cdots, 0, 1)$ is the standard $(n+1)^{th}$ unit coordinate vector in $\mathbb{R}^{n+1}$.  With either $\Omega = B_2$ or $B_2^+$,  $\F : \Omega \times (0,2) \rightarrow \R^{n+1}$ is a given measurable vector field $\F = (F_1, F_2,\cdots, F_n)$, and the data $f: \Omega \rightarrow \R$ is a given measurable function. We also write
\[
\nabla u(X) = (\nabla_x u(X), \partial_y u(X)), \quad \text{div}[{\bf F}] = \sum_{k=1}^n \partial_{x_k} F_k(X) + \partial_y F_{n+1} (X), \quad X = (x,y) \in \Omega \times (0,2).
\] 
We also use the standard notation $L^p(Q_r, \mu)$ (or $L^p(Q_r^+, \mu)$) for the weighted Lebesgue space with weight $\mu$ that consists of all measurable function $f$ defined on $Q_r$ (or $Q_r^+$) such that $|f|^p$ is integrable on $Q_r$ (or $Q_r^+$) with the $\mu(y) dxdy$ measure. Similarly $W^{1,p}(Q_r, \mu)$ or $W^{1,p}(Q_r^+, \mu)$ is the weighted Sobolev space with weight $\mu$, where both the function and its distributional gradient are in the weighted Lebesgue space. For a weight function $\mu : \R \rightarrow [0,\infty)$, we denote $\hmu(X) = \mu(y)$ with $X = (x,y) \in \R^{n+1}$, and $\hmu(E) = \int_E \mu(y) dX$ for any measurable set $E\subset \R^{n+1}$. For a locally integrable function $f$ in $\R^n$, we also denote $\wei{f}_E$ the average of $f$ on the measurable set $E$, 
$\wei{f}_E = \fint_{E} f(x) dx.
$

We can now state our first result on Sobolev regularity estimates of weak solutions of \eqref{main-eqn}. The standard definitions of weak solutions of \eqref{main-eqn}-\eqref{main-eqn-flat-domain} are given in Definition \ref{grounding-weak-solution-Q-R} and Definition \ref{weak-solution-Q-R-plus}. The definition of Muckenhoupt classes of $A_{p}$-weights is recalled in Section \ref{statement-coefficients}. 
\begin{theorem} \label{local-grad-estimate-interior} Let $\Lambda >0, M_0 \geq 1$, and $p \geq 2$. Then, there exists $\delta = \delta (\Lambda, M_0, n, p) >0$ and sufficiently small such that the following statement holds: 
Suppose that $\mu \in \A_2(\R)$ with $[\mu]_{A_2} \leq M_0$, and $\A \in \mathcal{A}(Q_2, \Lambda, M_0, \mu)$ satisfies the degeneracy  condition
\begin{equation} \label{Q-2-ellipticity-condition}
\Lambda^{-1} \mu(y) |\xi|^2 \leq \wei{\A(X)\xi, \xi} \leq \Lambda \mu(y) |\xi|^2, \quad \forall \ \xi \in \R^{n+1}, \quad \text{for a.e} \ X = (x,y) \in Q_2,
\end{equation}
and the following smallness condition on the mean oscillation with weight $\mu$
\begin{equation}\label{PBMO}
\sup_{0< \rho <1} \sup_{X_0 = (x_0, y_0)  \in Q_1} \frac{1}{\hmu(D_\rho(X_0))} \int_{D_\rho(X_0) \cap Q_2} |\A(x,y) - \wei{\A}_{B_\rho(x_0)}(y)|^2 \mu^{-1}(y) dx dy < \delta^2.
\end{equation}
If $\F : Q_2 \rightarrow \R^{n+1}$, $f : B_2 \rightarrow \R$ satisfy $|\F/\mu| \in L^p(Q_2, \mu)$ and $f/\mu \in L^p(Q_2, \mu)$, then for every weak solution $u \in W^{1,2}(Q_2, \mu)$  of \eqref{main-eqn}, it holds that $\nabla u \in L^p(Q_1, \mu)$, and moreover 
\[
\|\nabla u\|_{L^{p}(Q_{1}, \mu)} \leq C \left( \hmu(Q_{1})^{\frac{1}{p} - \frac{1}{2}} \|\nabla u\|_{L^{2}(Q_{2}, \mu)} + \left\|\F/\mu \right\|_{L^{p}(Q_{2}, \mu)} + \left\|f/\mu\right\|_{L^{p}(Q_{2}, \mu)}\right),
\]
for some constant $C$ depending only on $\Lambda, p, n, M_{0}$. 
\end{theorem}

We would like to note that in \eqref{PBMO}, we only measure the oscillation of $\A$ in the $x$-variable.  When $\mu =1$, this type of  bounded mean oscillation is used in \cite{Byun-P, Dong-Kim, Kim-Krylov, Kim-Krylov-1, Krylov} and in this case $\A$ is referred as a variably partially \textup{BMO} coefficient. For general $\mu$, functions with bounded mean oscillation as measured in all $x,y$ directions as in \eqref{PBMO} is referred as functions of bounded mean oscillation with weight \cite{MW1, MW2, Cuerva}. As noted in \cite{MW1, MW2}, the space of functions with bounded mean oscillation with weight is different from usual weighted \textup{BMO} space, and is also different from the well-known John-Nirenberg \textup{BMO} space. 

We postpone the precise definition for the class of coefficients $\mathcal{A}(Q_2, \Lambda, M_0, \mu)$ to Section \ref{statement-coefficients}, but essentially  this class consists of all measurable symmetric matrix-valued functions with the property that  weak solutions of  the corresponding homogeneous ``freezed coefficient equations'' with coefficient $\langle \mathbb{A} \rangle_{B_{r}(x_{0})} (y)$ satisfy a Lipschitz estimates. When $\mu =1$, this class consists of all uniformly elliptic, symmetric, measurable $(n+1) \times (n+1)$ matrices  and as such, the uniformly elliptic coefficient matrices that are in the variably partially  \textup{VMO} space as used in \cite{B-Kim, Byun-P, Dong-Kim, Kim-Krylov, Kim-Krylov-1, Krylov} satisfy all conditions of Theorem \ref{local-grad-estimate-interior}.  Moreover, we will show in Section \ref{Lipchits-est-section} that coefficient matrices of the form
\[
\mathbb{A}(x,y) = |y|^{\alpha} \begin{bmatrix}
\mathbb{B}(x, y)&0\\
0&b(x,y)
\end{bmatrix},\quad \alpha\in(-1, 1)
\]
where $\begin{bmatrix}
\mathbb{B}(x, y)&0\\
0&b(x,y)
\end{bmatrix}$ is uniformly elliptic in $\mathbb{R}^{n+1}$ and has small mean oscillation in the $x$-variable satisfy all conditions of Theorem \ref{local-grad-estimate-interior}. 
There types of coefficients are important as they arise from the so called ``extension  problem" for fractional elliptic equations \cite{Caffa-Sil, T-Stinga2010, Caffa-Stinga, Cabre, Capella} and are of matrices in Grushin   type operators \cite{Grushin}.

We also remark that the smallness  condition on the mean oscillation with respect to the weight $\mu$ for the coefficient $\A$ as defined in \eqref{PBMO} is natural. A similar but distinct  smallness condition on mean oscillations for degenerate equations has already appeared  in our previous work \cite{CMP}, where we have demonstrated the optimality of this smallness condition via a counterexample. We should also point out that in case of uniformly elliptic equations, i.e. $\mu =1$, the counterexample of Meyers in \cite{M} has demonstrated the necessity of requiring a small mean oscillation on coefficients to obtained desired higher integrability of gradients of weak solutions. This smallness condition on $\A$ in \eqref{PBMO} is reduced to the standard smallness condition in the \text{BMO} space for the  uniformly elliptic case that has been used  in  \cite{ B-Kim, Byun-P,  BW2,  Chiarenza,  Dong-Kim, Kim-Krylov,Krylov, KZ1, MP-1, MP, NP} for elliptic equations and in \cite{BCC,B1,BW1, LTT, Kim-Krylov-1, TN} for parabolic equation.

We next state our second result on the $W^{1,p}$-regularity estimate for weak solutions of  the mixed boundary value problem \eqref{main-eqn-flat-domain}. In the statement, we use the notation for the class of degenerate-singular coefficients $\mathcal{B}(Q_2^+,  \Lambda, M_0, \mu)$, which is defined in Definition \ref{class-B}. This class of matrices is defined similarly as the class $\mathcal{A}(Q_2, \Lambda, M_0, \mu)$ that we have already discussed, with the only difference happens on the flat boundary $T_1 \times [0,1]$ part of $\partial Q_2^+$. 
\begin{theorem}\label{local-grad-estimate-half-cylinder} Let $p\geq 2$, $M_0 \geq 1, \Lambda >0$,  then there exists a sufficiently small positive number  $\delta= \delta(\Lambda, M_0, p, n) > 0$  such that the following holds:  Suppose that $\mu \in A_{2}(\R)$ such that $[\mu]_{A_{2}} \leq M_{0}$, and suppose that the matrix $\mathbb{A} \in \mathcal{B}(Q_2^+, \Lambda, M_0, \mu)$ satisfies the  degeneracy  condition
\begin{equation} \label{Q-2-plus-ellipticity-condition}
\Lambda^{-1} \mu(y) |\xi|^2 \leq \wei{\A(X)\xi, \xi} \leq \Lambda \mu(y) |\xi|^2, \quad \forall \ \xi \in \R^{n+1}, \quad \text{for a.e} \ X = (x,y) \in Q_2^+,
\end{equation}
and the following smallness condition on the mean oscillation with weight $\mu$
\begin{equation}\label{PBMO-Q-2-plus}
 \sup_{0< \rho <1} \sup_{X_0 = (x_0, y_0)  \in Q_1^+} \frac{1}{\hmu(D_\rho(X_0))} \int_{D_\rho(X_0) \cap Q_2^+} |\A(x,y) - \wei{\A}_{B_\rho(x_0) \cap B_2^+}(y)|^2 \mu^{-1}(y) dx dy < \delta^2.
\end{equation}
Suppose also that a vector field $\F: Q_2^+ \rightarrow \R^{n+1}$ and a function $f : B_2^+ \rightarrow \R$ satisfy ${\bf F}/\mu, f/\mu \in L^p(Q_2^+, \mu)$. Then if $u\in W^{1, 2}(Q_{2}^+, \mu)$ is a weak solution of \eqref{main-eqn-flat-domain}, 
it holds that $\nabla u\in L^{p}(Q^{+}_{1}, \mu) $, and moreover
\[
\|\nabla u\|_{L^{p}(Q^{+}_{1}, \mu)} \leq C \left( \hmu(Q^{+}_{1})^{\frac{1}{p} - \frac{1}{2}} \|\nabla u\|_{L^{2}(Q^{+}_{2}, \mu)} + \left\|\F/\mu\right\|_{L^{p}(Q^{+}_{2}, \mu)} + \left\|f/\mu\right\|_{L^{p}(Q^{+}_{2}, \mu)}\right),
\]
for some constant $C$ depending only on $\Lambda, p, n, M_{0}$. 
\end{theorem}
Again in addition to uniformly elliptic matrices (i.e. $\mu = 1$) with small mean  oscillation in the $x$-variable, we will show in Section \ref{Lipchits-est-section} that coefficient matrices of the form
\[
\mathbb{A}(x,y) = |y|^{\alpha} \begin{bmatrix}
\mathbb{B}(x)&0\\
0&b(x)
\end{bmatrix},\quad \alpha\in(-1, 1)
\]
where $\begin{bmatrix}
\mathbb{B}(x)&0\\
0&b(x)
\end{bmatrix}$ is uniformly elliptic in $\mathbb{R}^{n+1}$ and has small mean oscillation in the $x$-variable satisfy all conditions of Theorem \ref{local-grad-estimate-half-cylinder}. 

An interesting feature of Theorem \ref{local-grad-estimate-half-cylinder} is that it establishes (a local) regularity estimate up to the boundary of domains with right angle corner, in this case the half cylinder  $Q_1^+$. This is a Lipschitz domain, but it is not a type with small Lipschitz constant domains nor is it a Reifenberg flat domain as considered in many papers, \cite{Byun-P, B1, BW1, BW2, LTT, MP, MP-1} to cite a few.  Therefore, even in the uniformly elliptic case, i.e. $\mu =1$, Theorem \ref{local-grad-estimate-half-cylinder} appears to be new. Also, as we already mentioned, Theorem \ref{local-grad-estimate-interior} and Theorem \ref{local-grad-estimate-half-cylinder} yield global Sobolev regularity estimate for weak solutions of some fractional elliptic equations with measurable coefficients. Due to the significance of this result, and its independent interest, we describe its details in a separate section, Section \ref{Fractional-elliptic-Sec}. 

Finally, we should point out that  we use a perturbation approach introduced in \cite{CP} to prove Theorem \ref{local-grad-estimate-interior} and Theorem \ref{local-grad-estimate-half-cylinder}. Due to the structure of the domain over which the equation is posed, four types of approximation estimates are required:  interior approximation estimates,  approximation up to the base of the domain cylinder, approximation estimates up to the flat side the domain cylinder and  approximation up to the corner of the domain cylinder. To overcome the difficulty arising from the degeneracy and singularity of the coefficients, in each approximation estimates, we use a two step approximation procedure. Gehring's type self-improved regularity estimates in weighted spaces are established in each of the four types of the approximations. The reverse H\"{o}lder's inequality, and doubling property of Muckenhoupt weights established by R. Coifman and C. Fefferman in \cite{Coif-Feffer} are also used appropriately in the approximation and density estimates. To implement perturbation method \cite{CP}, we also establish some results on uniformly Lipchitz estimates for the freezed coefficient equations in Section \ref{Lipchits-est-section}. These results seem to be new, and also of independent interest.

We now briefly outline the organization of the paper. Results on Sobolev regularity theory for spectral fractional elliptic equations are described in the next section, Section \ref{Fractional-elliptic-Sec}. The proof of these results will be given in Section \ref{proof-extension-problem-sec}. 
Section \ref{statement-coefficients} will define notations and reviews results on weights and weighted inequalities. Section \ref{weak-solution-coefficients} defines weak solutions for a variety of boundary conditions and also provides relevant class of coefficients. Section \ref{Appro-Section} provides approximation estimates, an important intermediate step required for the proof of Theorem \ref{local-grad-estimate-interior} and Theorem \ref{local-grad-estimate-half-cylinder}. The proofs of the main theorems is given in Section \ref{density-est-sec} and Section \ref{half-cyl-proof}. 
Lipchtitz regularity estimates are proved in Section \ref{Lipchits-est-section}. 

%================
\section{Global Sobolev regularity estimate for solutions of spectral fractional elliptic equations}  \label{Fractional-elliptic-Sec}
Let $\Omega \subset \R^n$ be open, bounded domain with $C^1$-boundary $\partial \Omega$. We study the following problem with a special class of coefficients motivated by the realization of fractional elliptic operators as Dirichlet-to-Neumann maps of degenerate elliptic equations: 
\begin{equation} \label{extension-problem}
\left\{
\begin{array}{cccl}
\textup{div}[\A(X) \nabla u] & =& \textup{div}[\F], & \quad \text{in} \quad \Omega_2: = \Omega \times (0,2),  \\
u  & = & 0, & \quad \text{on} \quad \partial \Omega \times (0,2), \\
\lim_{y \rightarrow 0^+} \wei{\A(x,y) \nabla u(x,y) - \F(x,y), e_{n+1}} & = & f(x), & \quad \text{for a. e.} \quad x \in \Omega, 
\end{array} \right.
\end{equation}
where our coefficient $\A(X)$ is defined as
\begin{equation} \label{extension-coefficients}
\A(X) = \mu(y)\left(\begin{matrix} \B(x) & 0 \\ 0 & 1 \end{matrix} \right), \quad  \text{for a.e.} \ X = (x,y)  \in \Omega \times (0,2),
\end{equation}
where $\B : \Omega  \rightarrow \R^{n\times n}$ is a symmetric, measurable matrix, and  $\mu(y) = |y|^\alpha$ with $\alpha \in (-1,1)$. We assume that there is $\Lambda >0$ such that
\begin{equation} \label{B-ellitpticity}
\Lambda^{-1} |\xi|^2 \leq \wei{\B(x) \xi, \xi} \leq \Lambda |\xi|^2, \quad \forall \ \xi \in \R^n, \quad \text{for a.e.} \quad x \in \Omega.
\end{equation}
A direct consequence of Theorem \ref{local-grad-estimate-interior} and Theorem \ref{local-grad-estimate-half-cylinder} is the following 
$W^{1,p}$-regularity estimates for weak solutions of \eqref{extension-problem}. 
 \begin{theorem} \label{extension-theorem} Let $p \geq 2$, $\Lambda >0,  \alpha \in (-1,1)$ be fixed, and let $\mu(y) = |y|^{\alpha}$ for $y \in \R$. There exists $\delta = \delta (\Lambda, \alpha, p, n) >0$ sufficiently small such that the following statement holds. Assume that  \eqref{extension-coefficients}-\eqref{B-ellitpticity} hold, $\partial \Omega \in C^1$, and for some fixed $r_{0} > 0$
\begin{equation} \label{B-BMO-mu}
[\B]_{\textup{BMO}(\Omega)}: =\sup_{ 0 < \rho < r_0} \sup_{x_0 \in \overline{\Omega}} \frac{1}{|B_\rho(x_0)|}\int_{B_\rho(x_0) \cap  \Omega} |\B(x) - \wei{\B}_{B_\rho(x_0) \cap \Omega} |^2  dx \leq \delta.
\end{equation}
 Then for every vector field $\F: \Omega_2 \rightarrow \R^{n+1}$, and function $f :\Omega \rightarrow \R$ such that $|\F/\mu|$ and $f/\mu$ are in $L^p(\Omega_2, \mu)$, if $u \in W^{1,2}(\Omega_2, \mu)$ is a weak solution of \eqref{extension-problem}, it holds that
\[
\|\nabla u\|_{L^{p}(\Omega_{1}, \mu)} \leq C \left( \hmu(B \times (0,2))^{\frac{1}{p} - \frac{1}{2}} \|\nabla u\|_{L^{2}(\Omega_{2}, \mu)} + \left\|\F/\mu\right\|_{L^{p}(\Omega_{2}, \mu)} + \left\|f/\mu\right\|_{L^{p}(\Omega_{2}, \mu)}\right),
\]
for some constant $C = C(\Lambda, \alpha, p, \Omega, n)>0$ and some ball $B \subset \R^n$ sufficiently large such that $\Omega \subset B$.
\end{theorem}
%=========
\noindent
We remark that if $\B$ is in $\textup{VMO}(\Omega)$, the space of functions with vanishing mean oscillation, then \eqref{B-BMO-mu} always holds. 
The proof of Theorem \ref{extension-theorem} is given in Section \ref{proof-extension-problem-sec}.
%=========

As a corollary of Theorem \ref{extension-theorem}, we obtain an important result on Calder\'{o}n-Zygmund regularity estimates for solution of fractional elliptic equations with measurable coefficients.  Given $0<s<1$, we are interested in
the boundary value problem 
\begin{equation}\label{FEE}
\left\{\begin{array}{cccl}
L^{s}u &= & f & \quad \text{in $\Omega$}\\
u &= & 0 & \quad \text{on $\partial \Omega$}
\end{array}\right.
\end{equation}
where the operator $L = \textrm{div}(\mathbb{B}(x)\nabla \cdot)$ and $\mathbb{B}(x)$ is a symmetric, elliptic and measurable coefficient matrix. 
The elliptic operator $L^{s}$, called the spectral fractional operator, and equation \eqref{FEE} are understood 
 via the spectral decomposition of the operator $L$ as follows.
 Let $\{\phi_{k}\}\subset W^{1, 2}_{0}(\Omega)$ be an orthonormal basis of $L^{2}(\Omega)$  consisting of eigenfunctions  of $L$ with eigenvalues $\lambda_{k}$: $0 < \lambda_{0} < \lambda_{1}\leq \lambda_{2}\leq \cdots \to \infty$.  It is well known that the first eigenvalue is simple. 
For $0 < s < 1$, define the space of functions 
 \[
\mathfrak{H}^{s}(\Omega) = \left\{u= \sum_{k=1}^{\infty} u_{k}\phi_{k} \in L^{2}(\Omega): \quad \sum_{k=1}^{\infty}\lambda_{k}^{s}u_{k}^{2} < \infty \right\}.
\]
Then $\mathfrak{H}^{s}(\Omega)$ is a Hilbert space with inner product 
$
\langle u, v\rangle_{\mathfrak{H}^{s}} = \sum_{k=0}^{\infty} \lambda^{s}_{k} u_{k} v_{k},$ where  $u = \sum_{k=1}^{\infty} u_{k}\phi_{k}$, and $v= \sum_{k=1}^{\infty} v_{k}\phi_{k}.
$
As shown in  \cite{Cabre, Capella, Caffa-Stinga} the space $\mathfrak{H}^{s}(\Omega)$ coincides with $H^{s}_{0}(\Omega)$ for $s\neq 1/2$ and with $H^{1/2}_{00}(\Omega)$ when  $s=1/2$. 
Here  $H^{1/2}_{00}$ is the Magenes-Lions space and  $H^{s}_{0}(\Omega)$ are completion of $C_{c}^{\infty}(\Omega)$ under the norm 
\[
\|u\|_{H^{s}(\Omega)}^{2} = \|u\|_{L^{2}}^{2} + \int_{\Omega} \int_{\Omega} \frac{(u(x) - u(z))^{2}}{|x-z|^{n + 2s}} dx dx,\quad 0 < s< 1.
\]
The dual space $\mathfrak{H}^{-s}(\Omega)$ is  identified with 
$\left\{f = \sum_{k=0}^{\infty}f_{k} \phi_{k}:  \sum_{k=0}^{\infty}\lambda_{k}^{-s}f_{k}^{2}  < \infty\right\}$
  with the duality pairing $
\langle f, u \rangle_{\mathfrak{H}^{-s}\to\mathfrak{H}^{s} } = \sum_{k=0}^{\infty} f_{k}u_{k}.
$
We can now define the Dirichlet spectral fractional elliptic operator $L^{s}: \mathfrak{H}^{s}(\Omega) \to \mathfrak{H}^{-s}(\Omega)$  
\[
\langle L^{s}u, v \rangle _{\mathfrak{H}^{-s}\to\mathfrak{H}^{s} }:= \langle u, v \rangle_{\mathfrak{H}^{s}}, \quad \forall u, v\in \mathfrak{H}^{s}(\Omega).
\]
With the above identification, we may write $L^{s}u = \sum_{k=0}^{\infty} \lambda_{k}^{s} u_{k}\phi_{k}$, in $\mathfrak{H}^{-s}(\Omega)$. 
It is now clear that given $f = \sum_{k=0}^{\infty}f_{k}\phi_{k}$ in $\mathfrak{H}^{-s }(\Omega)$ the equation \eqref{FEE}, understood in the above sense, 
has a unique solution $u = \sum_{k=0}^{\infty} u_{k}\phi_{k}\in \mathfrak{H}^{s}(\Omega)$ where $u_{k} = \lambda_{k}^{-s}f_{k}$.  Our main result for fractional elliptic equations of the above type is the following. 
\begin{theorem} \label{fractional-theorem} Let $s \in (0,1)$ be fixed, and let $p \geq 2$ such that $(1-2s)(1-p) +1 >0$. Then there exists a constant $\delta > 0$ such that if $\mathbb{B}$ satisfies \eqref{B-BMO-mu}, then for a given $f\in L^{p}(\Omega)$, the unique solution to the boundary value problem \eqref{FEE} is in $ W^{\alpha, p}(\Omega)$, for $\alpha = s  + \frac{(p-2)(1-s)}{p}$. Moreover, there exists a constant $C = C(\Omega, \Lambda, s, p) > 0$ such that 
\[
\|u\|_{W^{\alpha, p}(\Omega)}  \leq C \|f\|_{L^{p}(\Omega)}.
\] 
\end{theorem}
Some remarks are in order.  
Under the assumption of the theorem, not only $u$ has a higher integrability but also has an improved  smoothness. Indeed, the smoothness parameter $\alpha = s  + \frac{(p-2)(1-s)}{p} > s$, demonstrating a "self-improvement" property  that has been observed by the so called divergence form fractional elliptic equations \cite{Kuusi-M-S,Schikorra}.  The divergence form fractional elliptic operator and the spectral fractional operator are in general different. As it has been shown in  \cite{Caffa-Stinga}, for the spectral fractional operator, there exists a kernel $K_{s}(x, y)$ and a nonnegative function $g_{s}(x)$ such that 
\[
\langle L^{s}u, v \rangle _{\mathfrak{H}^{-s}\to\mathfrak{H}^{s}} = \int_{\Omega}\int_{\Omega} K_{s}(x, z)(u(z)-u(x))(v(z)-v(x))dxdz + \int_{\Omega} g_{s}(x) u(x)v(x)dx,  
\]
where for a constant $C$ that only depends on $n$ and $s$,  
\[
0\leq K_{s}(x, z) \leq \frac{C_{n,s}}{|x-z|^{n + 2s}}, \,\,\text{for all $x\neq z$.}
\]
It is also interesting to note that Theorem \ref{fractional-theorem} is applicable for the case $s < 1/2$ with $2\leq p < 2 + \frac{2s}{1-2s}$, while the results  \cite[Theorem 1.2, Theorem 1.3]{Caffa-Stinga} require $s > 1/2$ for the H\"{o}lder's regularity estimate of $\nabla u$ to hold. %Moreover, as a corollary  of Theorem \ref{fractional-theorem} and  a mere application of fractional Sobolev embedding, we have H\"older regularity of solutions of fractional elliptic equations with coefficients that can possibly be discontinuous. 
%\begin{corollary} \label{fractional-theorem} Let $s \in (0,1)$ be fixed, and let $p \geq 2$ such that $(1-2s)(1-p) +1 >0$, and $p > n + 2(1-s)$. Then there exists a constant $\delta > 0$ such that if $\mathbb{B}$ satisfies \eqref{B-BMO-mu}, then the unique solution to  \eqref{FEE}  $u$ corresponding to a given $f\in L^{p}(\Omega)$ is in $C^{0, \beta}(\overline{\Omega})$, where 
%$\beta = 1 - \frac{n+2(1-s)}{p}$ with the estimate 
%\[
%\|u\|_{C^{0, \beta}(\overline{\Omega})}  \leq C \|f\|_{L^{p}(\Omega)}.
%\] 
%\end{corollary}

%==================
%===============
\section{Notations, Weights, Weighted Sobolev spaces, Weighted inequalities} 
 \label{statement-coefficients}
%========
\subsection{Balls and cylinders} 
For every $r>0$, and $x_0 = (x_0', x^0_n) \in \R^n$, we denote $B_r(x_0)$ the ball in  $\R^n$ centered at  $x_0$ and has radius $r$. Moreover, the upper half balls in $\R^n$ and its flat boundary part is denoted by 
\[
B_r^+ (x_0) = \{x = (x', x_n) \in B_r(x_0): x_n > x_n^0 \}, \quad T_r (x_0) = \{ x = (x', x_n) \in B_r (x_0) : x_n =x_n^0 \}.
\]
When $x_0 =0$, the origin of $\R^n$, we always write
\[
B_r = B_r(0), \quad B_r^+ = B_r^+(0), \quad T_r = T_r(0).
\] 
Two types of cylinders in $\mathbb{R}^{n+1}$ are needed and then denoted differently in the paper.  For $r>0$, $\Gamma_r$ denotes the open interval $(0, r)$, and   for  $y \in \R$, we write $\Gamma_{r}(y)  = (y, y +r)$.  A cylinder  with bottom center point $X= (x, y)$, with $x = (x', x_n) \in \R^n$, and base a ball $B_r(x)$ and half ball $B_{r}^{+}(x_{0})$ are  denoted by, respectively, 
\[
Q_{r}(X) = B_r(x) \times \Gamma_r(y), \quad  \text{and }\, Q_{r}^+(X) = B_r^+(x) \times \Gamma_r(y),
\]
When $X = (0,0)$, we use the notation 
$
 Q_{r} = Q_r(0, 0), \quad T_r = T_r(0), \quad \text{and} \quad Q_{r}^+ = Q_{r}^+(0, 0).
$
We will denote the interval $(y-r, y+ r)$ by $I_{r} (y)$, and if $y =0$, we write $I_{r}  = I_{r} (0)$. Cylinders with middle center point $(x, y)$ with base a ball $B_{r}(x)\times \{y-r\}$ and half ball $B_{r}^{+}(x)\times \{y-r\}$ are denoted by, respectively, 
\[
D_{r}(x, y) = B_{r}(x) \times I_{r} (y) \quad \text{and}\quad  D_{r}^{+}(x, y) = B_{r}^{+}(x) \times I_{r} (y). 
\]
In a similar fashion as above, we denote 
$
D_{r} = D_{r}(0, 0),\quad D_{r}^{+} = D_{r}^{+}(0, 0). 
$
%==========
\subsection{Muckenhoupt weights}
For $1 \leq p < \infty$, a non-negative, locally integrable function $\mu :\R^n \rightarrow [0, \infty)$ is said to be in the class $A_p(\R^n)$ of Muckenhoupt weights if 
\[
\begin{split}
[\mu]_{A_p} & :=  \sup_{\textup{balls} \ B \subset \R^n} \left(\fint_{B} \mu (x) dx \right) \left(\fint_{B} \mu (x)^{\frac{1}{1-p}} dx \right)^{p-1} < \infty, \quad \textup{if} \quad p > 1, \\
[\mu]_{A_1} &: =  \sup_{\textup{balls} \ B \subset \R^n} \left(\fint_{B} \mu (x) dx \right)  \norm{\mu^{-1}}_{L^\infty(B)}  < \infty \quad \textup{if} \quad p  =1.
\end{split}
\]
%=======
It is well known that the class of $A_p$-weights satisfies the reverse H\"{o}lder's inequality and the doubling properties, \cite{Coif-Feffer}. In particular, a measure with an $A_p$-weight density is, in some sense, comparable with the Lebesgue measure. 
\begin{lemma}[\cite{Coif-Feffer}] \label{doubling} For $1 < p < \infty$, the following statements hold true
\begin{itemize}
\item[\textup{(i)}] If $\mu \in A_{p}(\R^n)$,  then for every ball $B \subset \R^n$ and every measurable set $E\subset B$, 
$
\mu(B) \leq [\mu]_{A_{p}} \left(\frac{|B|}{|E|}\right)^{p} \mu(E).
$
\item[\textup{(ii)}] If $\mu \in A_p(\R^n)$ with $[\mu]_{A_p} \leq M$ for some given $M \geq 1$, then there is $C = C(M, n)$ and $\beta = \beta(M, n) >0$ such that
$
\mu(E) \leq C \left(\frac{|E|}{|B|} \right)^\beta \mu(B),
$
 for every ball $B \subset \R^n$ and every measurable set $E\subset B$.
\end{itemize}
\end{lemma} \noindent
%=======
Our next lemma is also a direct consequence of the reverse H\"{o}lder's inequality for the Muckenhoupt weights.  The proof of this lemma can be found in \cite[Lemma 3.4]{CMP}
\begin{lemma} \label{Reverse-H} Let  $M_{0} > 0$, then, there exists $\beta = \beta(M_0, n) >0$ such that if $\mu\in A_{p}(\R^n)$, with some $1 < p < \infty$, satisfying $[\mu]_{A_{p}} \leq M_{0}$, then for every balls $B \subset \R^n$, if $u \in L^2(B, \mu)$, then  
$u \in L^{1+\beta}(B)$, and
 \[
\left (\fint_{B} |u|^{1+\beta} dx \right)^{\frac{1}{1+\beta}} \leq C(n, M_0) \left( \fint_{B} |u|^2 d\mu  \right)^{1/2}. 
\]
\end{lemma} \noindent
The following remark is used frequently in the paper, and it follows directly from the definition of $A_p$-weights
\begin{remark} For $\mu \in A_p(\R)$, let $\hat{\mu}(X) = \mu(y)$ for all a.e. $X = (x,y) \in \R^n \times \R$. Then, $\hmu$ is in $A_p(\R^{n+1})$. Moreover, $[\mu]_{A_p(\R)} = [\hmu]_{A_p(\R^{n+1})}$.
\end{remark} \noindent
Now, we discuss about the Hardy-Littlewood maximal operator  with the measure $\mu(y) dxdy$, and its boundedness in weighted spaces. For a given locally integrable function $f$ on $\R^{n+1}$ and a weight $\mu$ defined on $\R$, we define the weighted Hardy-Littlewood maximal function as 
\[
\mathcal{M}_{\mu}f(X) = \sup_{\rho > 0}\fint_{D_{\rho}(X)}|f(Z)| d(\hmu(Z)). % =  \sup_{\rho > 0} \frac{1}{\mu(B_{\rho} (x)) }\int_{B_{\rho}(x)}|f| \, \mu(x) dx. 
\]
For functions $f$ that are defined on a bounded domain $E \subset \R^{n+1}$, we define 
\[
\mathcal{M}_{\mu, E}f(X) = \mathcal{M}_{\mu}(f\chi_{E})(X).
\]
Since a weight $ \mu \in A_p(\R)$- Muckenhoupt class is a doubling measure, the Hardy-Littlewood maximal operator with respect to this weight is a bounded operator $L^q(\R^{n+1}, \hmu) \rightarrow L^q(\R^{n+1}, \hmu)$, with $q \in (1, \infty)$. The following lemma is classical, and its proof can be found in \cite[Lemma $7.1.9$ - eqn (7.1.28) ]{Grafakos}.
\begin{lemma} \label{Hardy-Max-p-p} Assume that $\mu \in A_s(\R)$ for some $1 < s < \infty$ with $[\mu]_{A_s}\leq M_0$. Then, the followings hold.
\begin{itemize}
\item[(i)] Strong $(p,p)$: Let $1 < p < \infty$, then there exists a constant $C = C(M_0, n,p)$ such that  
\[ \|\mathcal{M}_{\mu}\|_{L^{p}(\R^{n+1}, \hmu) \to L^{p}(\R^{n+1}, \hmu)} \leq C. \] 
\item[(ii)] Weak $(1,1)$: 
There exists a constant $C=C(M_0, p, n)$ such that for any $\lambda >0$, we have 
\[
\hmu(x\in \mathbb{R}^{n+1}: \mathcal{M}_{\mu} (f) > \lambda) \leq \frac{C}{\lambda} \int_{\mathbb{R}^{n+1}}|f|d\hmu.
\]  
\end{itemize}
\end{lemma}

\subsection{Weighted Sobolev spaces, and weighted trace estimates} \label{weighted-spaces}
Let $\Omega \subset \R^n$ be a bounded domain, $1 < p < \infty$, $\mu : \R \rightarrow [0, \infty)$ be some weight function, for $R >0$ and $\Omega_R = \Omega \times (0,R)$, we denote $L^p(\Omega_R, \mu)$ the weighted Lebesgue space consisting of all measurable function $f : \Omega_R\rightarrow \R$ such that
\[
\norm{f}_{L^p(\Omega_R, \mu)} = \left(\int_{\Omega_R} |f(x,y)|^p \mu(y) dxdy \right)^{1/p} < \infty.
\]
We denote $W^{1,p}(\Omega_R, \mu)$ be the set of all measurable functions $u: \Omega_R \rightarrow \mathbb{R}$ such that $u \in L^p(\Omega_R, \mu)$ and all of the weak derivatives 
$\partial_{x_k}u, \partial_y u$ exist and belong to $L^p(\Omega_R, \mu)$. The norm of $W^{1,p}(\Omega_R, \mu)$ is denoted by
\[
\norm{u}_{W^{1,p}(\Omega_R, \mu)} = \left[\int_{\Omega_R} |u(X)|^p \mu(y) dX + \int_{\Omega_R}|\nabla u|^p \mu(y) dX \right]^{1/p}, \quad dX = dxdy.
\]
Note that from \cite{Fabes},  if $\mu \in A_p$, a class of Muckenhoupt weights, then the space $W^{1,p}(\Omega_R, \mu)$ is indeed the closure of $C^\infty(\overline{\Omega}_R)$ with norm $\norm{\cdot}_{W^{1,p}(\Omega_R, \mu)}$.  We also will use the following function spaces in the sequel. 
\begin{enumerate}
\item $W_{0}^{1, p}(\Omega_R, \mu)$  = the closure in $W^{1,p}(\Omega_R, \mu)$ of $C_{0}^{\infty}(\Omega_R)$, the space of smooth functions compactly supported in $\Omega_R$. 
\item $\overset{*}{W}^{1,p}(\Omega_R, \mu)$= the closure in $W^{1,p}(\Omega_R, \mu)$ of the space
\[
\Big\{ \phi \in C^\infty(\overline{\Omega_R}): \phi =0 \quad \text{on}  \quad \Big(\partial \Omega \times (0, R)\Big) \cup \Big({\Omega\times \{R\}} \Big) \Big \}.
\]
\item $\overset{o}{W}^{1,p}(\Omega_R, \mu)$ = the closure in $W^{1,p}(\Omega_R, \mu)$ of the set $
\Big \{ \phi \in C^\infty(\overline{\Omega_R}): \phi =0 \quad \text{on} \quad \partial \Omega \times (0, R)\Big\}$.
\end{enumerate}
We observe that 
\[
W_{0}^{1, p}(\Omega_R, \mu) \subset \overset{*}{W}^{1,p}(\Omega_R, \mu) \subset 
\overset{o}{W}^{1,p}(\Omega_R, \mu) \subset W^{1,2}(\Omega_R, \mu).
\]
The following results of weighted trace inequalities will be needed in the paper.
\begin{lemma} \label{trace-zero}  Suppose that $M_{0}>0$ and $R>0$. Assume that $\mu \in A_2(\mathbb{R})$ with $[\mu]_{A_2} \leq M_0$. Then, there exists a trace map $T : \overset{o}{W}^{1,2}(\Omega_R, \mu) \rightarrow L^2(\Omega)$ such that
\[
\norm{T[u]}_{L^2(\Omega)} \leq C(n, M_0)\left[ \frac{\textup{diam} (\Omega)^2 + R^2}{\mu((0,R)} \int_{\Omega_R} |\nabla u|^2\mu(y) dX \right]^{1/2},   
\]
for all $u \in \overset{o}{W}^{1,2}(\Omega_R, \mu)  $.
\end{lemma} 
\begin{proof}
The proof is elementary and follows from standard arguments. We present  it here for completeness. We only need to prove that 
\[
\fint_{\Omega} |u(x,0)|^2 dx \leq  \frac{C(n, M_0) \Big[\text{diam}(\Omega)^2 + R^2\Big] }{|\Omega| \mu((0,R))}\int_{\Omega_R} |\nabla u(X)|^2 \mu(y) dX,
\]
for all $u \in C^\infty(\overline{\Omega}_R)$ with $u(\cdot, y) =0$ on $\partial \Omega$ for all $y \in (0,R)$.  For such $u$, observe that
\[
u(x,y) - u(x,0) = \int_0^y \partial_y u(x,s) ds.
\]
Therefore,
\[
\begin{split}
|u(x,0)|^2 & \leq 2\left[ |u(x,y)|^2 +  \left ( \int_0^R |u_y(x,s)| \mu(s)^{1/2} \mu(s)^{-1/2} ds \right )^2\right] \\
& \leq 2\left[ |u(x,y)|^2 +  \left ( \int_0^R |u_y(x,s)|^2 \mu(s) ds \right) \left( \int_0^R \mu(s)^{-1} ds \right )\right].
\end{split}
\]
Then, by integrating this last inequality in $x$, we obtain
\[
\int_{\Omega} |u(x,0)|^2 dx \leq 2\left[ \int_{\Omega} |u(x,y)|^2 dx  +  \left ( \int_{\Omega_R} |u_y(x,s)|^2 \mu(s) dx ds \right) \left( \int_0^R \mu(s)^{-1} ds \right )\right]
\]
Now, observe that for each $y \in (0,R)$, the function $u(\cdot, y) =0$ on $\partial \Omega$. Therefore, we can apply the Poincar\'{e}'s inequality  to the first term in the right hand side of the last inequality to further derive 
\[
\int_{\Omega} |u(x,0)|^2 dx \leq C(n)\left[ \text{diam}(\Omega)^2 \int_{\Omega} |\nabla_x u(x,y)|^2 dx  +  \left ( \int_{\Omega_R} |u_y(x,s)|^2 \mu(s) dx ds \right) \left( \int_0^R \mu(s)^{-1} ds \right )\right]
\]
Then, multiplying this last inequality with $\mu$ and then integrating the result in $y$ on $(0,r)$, 
we obtain
\[
\begin{split}
\mu((0,R) \int_{\Omega} |u(x,0)|^2 dx & \leq C(n)\left[ \text{diam}(\Omega)^2 \int_{\Omega_R} |\nabla_x u(x,y)|^2 \mu(y)dx dy  \right. \\
& \quad +  \left. \left ( R^2\int_{\Omega_R} |u_y(x,y)|^2 \mu(y) dx dy \right) \left( \fint_0^R \mu(s) ds \right) \left( \fint_0^R \mu(s)^{-1} ds \right )\right].
 \\
\end{split}
\]
This last estimate implies
\[
\fint_{\Omega} |u(x,0)|^2 dx \leq C(n, M_0) \frac{\text{diam}(\Omega)^2 + R^2 }{|\Omega| \mu((0,R))}\int_{\Omega_R} |\nabla u(X)|^2 \mu(y) dX.
\]
The proof is then complete.
\end{proof}
The following lemma can be proved similarly. 
\begin{lemma} \label{trace}  Assume that $\mu \in A_2$ with $[\mu]_{A_2} \leq M_0$. Then, there exists a trace map $T : W^{1,2}(\Omega_R, \mu) \rightarrow L^2(\Omega)$ such that
\[
\norm{T[u]}_{L^2(\Omega)} \leq C(n, M_0)\left[ \left( \frac{1}{\mu((0,R))} \int_{\Omega_R} |u(X)|^2 \mu(y) dX \right)^{1/2} + \left( \frac{R^2}{\mu((0,R))} \int_{\Omega_R} |\nabla u|^2\mu(y) dX \right)^{1/2} \right],   
\]
for all $u \in W^{1,2}(\Omega_R, \mu) $.
\end{lemma} 
\section{Definition of weak solutions and relevant classes of coefficients}\label{weak-solution-coefficients}
In this section we define the standard notion of weak solution for degenerate elliptic equations. The approximation procedure we will be implementing in the next section uses various boundary conditions whose corresponding notion of weak solution need to be defined. We will also define two relevant classes of degenerate coefficients. 
\subsection{Definition of weak solutions}
In what follows we will assume that the coefficient matrix $\A(X)$ satisfies the degeneracy condition 
\begin{equation}\label{ELLIPTICITY}
\Lambda^{-1}\mu(y)|\xi|^2 \leq \wei{\A(X)\xi, \xi} \leq \Lambda \mu(y)|\xi|^2, \quad  \forall \ \xi \in \R^{n+1},
\end{equation}
for almost every $X = (x,y)$ in the appropriate domain of $\A(X)$, for a given $\mu \in A_2(\R)$ and  $\Lambda >0$.
%=============
\begin{definition} \label{grounding-weak-solution-Q-R} Let $\Omega \subset \R^n$ be open and bounded, $R>0$, and $\A : \Omega_R \rightarrow \mathbb{\R}^{(n+1) \times (n+1)}$ be a symmetric, measurable matrix satisfying \eqref{ELLIPTICITY}. 
For a given $f \in L^2(B_R)$ and a vector field $\F$ such that $|\F| \in L^2(\Omega_R, \mu^{-1})$,  $u \in W^{1,2}(\Omega_R, \mu)$ is a weak solution of
\begin{equation} \label{grounding-Q-R-eqn-u}
\left\{
\begin{array}{cccl}
\textup{div}[\A(X) \nabla u]  &= & \textup{div}({\bf F}), & \quad \Omega_R, \\
\lim_{y \rightarrow 0^+}\Big[\wei{\A(x,y) \nabla u(x,y) - {\bf F}(x,y), e_{n+1}}\Big]&= & f(x), & \quad  x \in \Omega_R.
\end{array}
\right.
\end{equation}
if
\[
\int_{\Omega_R} \wei{\A(X) \nabla u(X), \nabla \psi(X)} dX  = \int_{\Omega_R} \wei{\F, \nabla \psi(X)} dX + \int_{\Omega} \psi(x,0) f(x) dx, \quad \forall \ \psi \in \overset{*}W^{1,2}(\Omega_R, \mu).
\]
\end{definition}
\noindent
%===============
We observe that since $\mu\in A_2$, in this definition,  by the weighted trace inequalities stated and proved in Lemma \ref{trace-zero} and Lemma \ref{trace}, $\psi(x,0) \in L^2(\Omega)$ in sense of trace, for all $\psi \in \overset{*}W^{1,2}(\Omega_R, \mu)$. From this, and since $f \in L^2(\Omega)$, the last integration in the above definition is well-defined. If $f =0$, the requirement $\mu \in A_2(\R)$ could be replaced by some other classes of weights. Next, we give the definition of weak solutions for equation with mixed boundary condition. 
\begin{definition} \label{weak-solution-Q-R-plus}  For $R>0$, let $\A : Q_R^+ \rightarrow \mathbb{\R}^{(n+1) \times (n+1)}$ be a symmetric, measurable matrix satisfying \eqref{ELLIPTICITY}. 
Given $f \in L^2(B_R^+)$ and a vector field $\F$ such that $|\F| \in L^2(Q_R^+, \mu^{-1})$, $u \in W^{1,2}(Q_R^+, \mu)$ is  a weak solution of
\begin{equation} \label{Q-R-plus-eqn}
\left\{
\begin{array}{cccl}
\textup{div}[\A(X) \nabla u]  &= & \textup{div}({\bf F}), & \quad \text{in} \quad Q_R^+, \\
u & =& 0, & \quad \text{on}  \quad T_R \times (0, R), \\
\lim_{y \rightarrow 0^+}\Big[\wei{\A(x,y) \nabla u(x,y) - {\bf F}(x,y), e_{n+1}}\Big]&= & f(x), & \quad  x \in B_R^+.
\end{array}
\right.
\end{equation}
if $u$ is in the closure of $\Big\{\phi \in C^\infty(\overline{Q}_R^+) : \phi_{| T_R \times (0,R)} =0 \Big\}$ in $W^{1,2}(Q_R^+, \mu)$, and
\[
\int_{\Omega_R} \wei{\A(X) \nabla u(X), \nabla \psi(X)} dX  = \int_{\Omega_R} \wei{\F, \nabla \psi(X)} dX + \int_{\Omega} \psi(x,0) f(x) dx, \quad \forall \ \psi \in \overset{*}W^{1,2}(\Omega_R, \mu).
\]
\end{definition}
\noindent
The following two definitions of weak solutions are also required right away in the next subsection. In these two definitions, it is not necessary to assume $\mu \in A_2$.
\begin{definition} \label{weak-solution-D-R} For each $R>0$,  let $\F \in L^2(D_R, \mu^{-1})$, and $\A : D_R \rightarrow \R^{(n+1)\times (n+1)}$ be a measurable, symmetric  matrix satisfying \eqref{ELLIPTICITY}. 
Then  $u\in W^{1, 2}_{\textup{loc}}(D_R, \mu)$ is  a weak solution of
\[
\textup{div}[\A(X)\nabla u] = \textup{div}[\F], \quad \text{in} \quad D_R
\]
if
\[
\int_{D_{R}} \langle \A(X) \nabla u(X), \nabla \phi (X)\rangle dX =  \int_{D_{R}} \langle {\bf F}(X),\nabla \phi(X)\rangle dX,  \quad \forall \phi \in C_{0}^{\infty}(D_{R}).
\]
\end{definition} \noindent
\begin{definition} \label{weak-solution-half-D-R} For $R>0$, let $\F \in L^2(D_R^+, \mu^{-1})$, 
 and $\A : D_R^+ \rightarrow \R^{(n+1)\times (n+1)}$ be a measurable symmetric matrix satisfying \eqref{ELLIPTICITY}. 
Then $u\in W^{1, 2}_{\textup{loc}}(D_R^+, \mu)$ is  a weak solution of
\[
\left\{
\begin{array}{cccl}
\textup{div}[\A(X)\nabla u] & = & \textup{div}[\F],  & \quad \text{in} \quad D_R^+, \\
u & = & 0, & \quad \text{on} \quad T_{R} \times I_R,
\end{array} \right.
\]
if $u$ is in the closure of $\Big\{\phi \in C^\infty(\overline{D}_R): \phi_{| T_R\times I_R} =0 \Big\}$ in $W^{1,2}(D_R^+, \mu)$ and
\[
\int_{D_{R}^+} \langle \A(X) \nabla u(X), \nabla \phi (X)\rangle dX =  \int_{D_{R}^+} \langle {\bf F}(X),\nabla \phi(X)\rangle dX,  \quad \forall \phi \in C_{0}^{\infty}(D_{R}^+).
\]
\end{definition} \noindent

%========
\subsection{Two classes of relevant degenerate/singular coefficients}

Recall that given a matrix $\A$, we write
$$
\wei{\A}_{B_{r(x_0)}} (y) = \fint_{B_r(x_0)} \A(x,y) dx,
$$
where in above the averaging is only in the $x$-variable. 
We now can state a class of coefficient matrices $\A$ that we use in the Theorem \ref{local-grad-estimate-interior}.
\begin{definition} \label{class-A} Let $\Lambda >0, M_0\geq 1$, and let $\mu \in A_2$ with $[\mu]_{A_2} \leq M_0$. 
A measurable, symmetric matrix $\A : Q_2 \rightarrow \R^{(n+1) \times (n+1)}$ is said to be in $\mathcal{A}(Q_2, \Lambda, M_0, \mu)$
 if \eqref{ELLIPTICITY} holds, and there is a constant $C = C(\Lambda, M_0, n)>0$ such that the following conditions hold. 
\begin{itemize}
\item[\textup{(i)}] For all $r \in (0,1)$, all $x_0 \in \overline{B}_1$, 
 and for every weak solution $v \in W^{1,2}(Q_{3r/2}(x_{0}, 0), \mu)$ of 
\[
\left\{
\begin{array}{cccl}
\textup{div}[\wei{\A}_{B_r(x_0)}(y) \nabla v(X)] & = & 0,  & \quad X= (x,y) \in Q_{3r/2}(x_{0}, 0), \\
\lim_{y \rightarrow 0^+ } \wei{\wei{\A}_{B_r(x_0)}(y) \nabla v(x,y), {\bf e}_{n+1}} & = &0, & \quad x \in   B_{3r/2}(x_{0}), 
\end{array}
\right.
\]
it holds that 
\[
\norm{\nabla v}_{L^\infty(Q_{5r/4}(x_{0}, 0))} \leq C \left\{\frac{1}{\hmu(Q_{3r/2}(x_{0}, 0))} \int_{Q_{3r/2}(x_{0}, 0)} |\nabla v(x,y)|^2 \mu(y) dxdy \right\}^{1/2}.
\]
\item[\textup{(ii)}] For every $X_0 = (x_0, y_0) \in \overline{Q}_1$, any  $r \in (0,1)$ so that $D_{2r}(X_0) \subset Q_2$, 
and for every weak solution $v \in W^{1,2}(D_{3r/2}(X_0), \mu)$ of 
\[
\textup{div}[\wei{\A}_{B_{3r/2}(x_0)}(y) \nabla v(X)]  =  0,   \quad X= (x,y) \in D_{3r/2}(X_0),
\]
it holds
\[
\norm{\nabla v}_{L^\infty(D_{5r/4}(X_0))} \leq C \left\{\frac{1}{\hmu(D_{3r/2}(X_0))} \int_{D_{3r/2}(X_0)} |\nabla v(x,y)|^2 \mu(y) dxdy \right\}^{1/2}.
\]
\end{itemize}
\end{definition}
As we already pointed out in the introduction, we will show in Section \ref{Lipchits-est-section} that coefficient matrices of the form
\begin{equation}\label{Example-class-A}
\mathbb{A}(x,y) = |y|^{\alpha} \begin{bmatrix}
\mathbb{B}(x, y)&0\\
0&b(x,y)
\end{bmatrix},\quad \alpha\in(-1, 1)
\end{equation}
where $\begin{bmatrix}
\mathbb{B}(x, y)&0\\
0&b(x,y)
\end{bmatrix}$ is uniformly elliptic in $\mathbb{R}^{n+1}$, belongs to the class $\mathcal{A}(Q_{2}, \Lambda, M_{0}, \mu)$.  
\noindent
The following class of coefficient matrices $\A$ is used in Theorem \ref{local-grad-estimate-half-cylinder}. 
\begin{definition} \label{class-B} Let $\Lambda >0, M_0 \geq 1$, and let $\mu \in A_2$ with $[\mu]_{A_2} \leq M_0$.  A measurable, symmetric matrix $\A : Q_2^+ \rightarrow \R^{(n+1) \times (n+1)}$ is said to be in $\mathcal{B}(Q_2^+, \Lambda, M_0, \mu)$ if \eqref{ELLIPTICITY} holds, and there is a constant $C = C(\Lambda, M_0, n) > 0$  such that the following conditions hold:
\begin{itemize}
\item[\textup{(i)}] The item \textup{(i)} of Definition \ref{class-A} holds for all $r \in (0,1)$ and for a.e  $x_0 \in B_1^+$ so that $B_{2r}(x_0) \subset B_1^+$. 
\item[\textup{(ii)}] The item \textup{(ii)} of Definition \ref{class-A} holds for a.e. $X_0 = (x_0, y_0) \in \overline{Q}_1^+$ and every $r \in (0,1)$ so that $D_{2r}(X_0) \subset Q_2^+$. 
\item[\textup{(iii)}] For a.e. $X_0 = (x_0, y_0) \in \overline{T}_1 \times (0,1)$,  and $r \in (0,1)$ such that 
$y_0 - 2r >0$, 
and for every weak solution $v \in W^{1,2}(D_{3r/2}^+(X_0), \mu)$ of 
\[
\left\{
\begin{array}{cccl}
\textup{div}[\wei{\A}_{B_{3r/2}^+(x_0)}(y) \nabla v(X)]  &= &  0, &   \quad X= (x,y) \in D_{3r/2}^+(X_0),\\
u & =& 0, & \quad \text{on} \quad T_{3r/2}(x_0) \times I_{3r/2},
\end{array} \right. 
\]
it holds
\[
\norm{\nabla v}_{L^\infty(D_{5r/4}^+(X_0))} \leq C \left\{\frac{1}{\hmu(D_{3r/2}(X_0))} \int_{D_{3r/2}^+(X_0)} |\nabla v(x,y)|^2 \mu(y) dxdy \right\}^{1/2}.
\]

\item[\textup{(iv)}]  For a.e. $X_0 = (x_0, 0) \in \overline{T}_1 \times \{0\}$,  $r \in (0,1)$, 
and for every weak solution $v \in W^{1,2}(Q_{3r/2}^+(x_{0}, 0), \mu)$ of 
\[
\left\{
\begin{array}{cccl}
\textup{div}[\wei{\A}_{B_{3r/2}^+(x_0)}(y) \nabla v(X)] & = & 0,  & \quad X= (x,y) \in Q_{3r/2}^+(x_{0}, 0), \\
 u & = & 0, & \quad \text{on} \quad T_{3r/2} \times (0, 3r/2), \\
\lim_{y \rightarrow 0^+ } \wei{\wei{\A}_{B_{3r/2}^+(x_0)}(y)\nabla v(x,y), {\bf e}_{n+1}} & = &0, & \quad x \in   B_{3r/2}^+(x_{0}), 
\end{array}
\right.
\]
it holds that 
\[
\norm{\nabla v}_{L^\infty(Q_{5r/4}^+(x_{0}, 0))} \leq C \left\{\frac{1}{\hmu(Q_{3r/2}(x_{0}, 0))} \int_{Q_{3r/2}^+(x_{0}, 0)} |\nabla v(x,y)|^2 \mu(y) dxdy \right\}^{1/2}.
\]
\end{itemize}
\end{definition}
\noindent
Again, we will show in Section \ref{Lipchits-est-section} that coefficient matrices of the form
\begin{equation}\label{Example-class-B}
\mathbb{A}(x,y) = |y|^{\alpha} \begin{bmatrix}
\mathbb{B}(x)&0\\
0&b(x)
\end{bmatrix},\quad \alpha\in(-1, 1)
\end{equation}
where $\begin{bmatrix}
\mathbb{B}(x)&0\\
0&b(x)
\end{bmatrix}$ is uniformly elliptic in $\mathbb{R}^{n+1}$,  belong to the class $\mathcal{B}(Q_{2}^{+}, \Lambda, M_{0}, \mu)$

%==================
\section{Four types of approximation estimates} \label{Appro-Section}

This section provides intermediate steps to prove Theorem \ref{local-grad-estimate-interior} and Theorem \ref{local-grad-estimate-half-cylinder}.  Our approach is based on the perturbation technique introduced by Caffarelli-Peral \cite{CP}. 
The approach in the paper is also influenced by the work \cite{BW1, BW2, CMP, LTT, MP, Wang}. Essentially, in each small cylinder, we approximate $\nabla u$ by some uniformly bounded vector field. Due to the structure of our domain, we need four types of approximation estimates. Though the statements of these estimates are very similar, the algebraic, and analysis details are quite different. We therefore represent each of these approximation estimates in one separated subsection. As before we work on the underlying assumption that $\mathbb{A}$  symmetric, measurable matrix $\A$ that satisfies the degeneracy condition 
\begin{equation} \label{ELLIPTICITY-5}
\Lambda^{-1}\mu(y)|\xi|^2 \leq \wei{\A(X)\xi, \xi} \leq \Lambda \mu(y)|\xi|^2,  \quad \forall \ \xi \in \R^{n+1},
\end{equation}
for all $ X = (x,y)$ in the appropriate domain of $\mathbb{A}$, with some $\mu \in A_2(\R)$ and $\Lambda >0$. 

%=======================
\subsection{Approximation estimates up to the base of the  cylinder domain } In this subsection we obtain an approximation estimate for $\nabla u$ over cylinders whose base touches the base of $Q_{2}$. 
The main result of the section now can be stated in the following proposition.
%=======================
\begin{proposition} \label{bottom-approximation-proposition} Let $\Lambda >0, M_0 >0$ be fixed. Then, for every $\epsilon >0$, there exists $\delta = \delta(\epsilon, \Lambda, M_0, n) >0$ with the following property. 
For every $X_0 = (x_0, 0) \in \overline{B_1} \times \{0\}$, $0 < r <1$, $\mu \in A_2(\R)$ with $[\mu]_{A_2} \leq M_0$ and $\mathbb{A}\in \mathcal{A}(Q_{2},\Lambda, M_{0},\mu)$ satisfying \eqref{ELLIPTICITY-5}, if 
$\F \in L^2(B_{2r}(x_0), \mu^{-1}), f \in L^2(B_{2r}(x_0))$, and
\[
\begin{split}
& \frac{1}{\hat{\mu}(Q_{3r/2}(X_0))} \int_{Q_{3r/2}(X_0)} |\A(X) - \langle{\A}\rangle_{B_{3/2}(x_0)}(y)|^2 \mu^{-1}(y) dX  \leq \delta^{2}, \\
& \frac{1}{\hat{\mu}(Q_{2r}(X_0))}\left[ \int_{Q_{2r}(X_0)} \Big|\F(X)/\mu(y)\Big|^2 \mu(y)dX  + \int_{Q_{2r}(X_0)} \Big|f(x)/\mu(y)\Big|^2 \mu(y) dX \right] \leq \delta^{2}.
\end{split}
\]
then  for every weak solution $u \in W^{1,2}(Q_{2r}(X_0), \mu)$ of 
\begin{equation} \label{u-Q-6-bottom}
\left\{
\begin{array}{cccl}
\textup{div}[\A(X) \nabla u]  &= & \textup{div}({\bf F}), & \quad Q_{2r}(X_{0}), \\
\lim_{y \rightarrow 0^+} \wei{\A(x,y) \nabla u(x,y) - {\bf F}(x,y), e_{n+1}}&= & f(x), & \quad  x \in B_{2r}(x_{0}),
\end{array}
\right.\end{equation}
satisfying
\[
\frac{1}{\hat{\mu}(Q_{2r}(X_0))} \int_{Q_{2r}(X_0)} |\nabla u(X)|^2 \mu(y) dX \leq 1,
\]
there is some $v \in W^{1,2}(Q_{3r/2}(X_0), \mu)$  such that
\begin{equation} \label{u-v-approximation-bottom}
\frac{1}{\hat{\mu}(Q_{3r/2}(X_0))} \int_{Q_{3r/2}(X_0)} |\nabla u(X) - \nabla v(X)|^2 \mu(y) dX \leq \epsilon^{2}.
\end{equation}
 Moreover, 
\begin{equation} \label{v-L-infty-bottom}
 \norm{\nabla v}_{L^\infty(Q_{5r/4}(X_0))}  \leq C(n, \Lambda, M_0).
\end{equation}
\end{proposition}
\noindent

The rest of the subsection is devoted to the proof Proposition \ref{bottom-approximation-proposition}.  We also note that the definition for weak solutions of \eqref{u-Q-6-bottom} is given in Definition \ref{grounding-weak-solution-Q-R}.
Without loss of generality, we can assume from now on that that $x_0 =0 \in \R^n$, and $r =1$.  We need several Lemmas in order to prove Proposition \ref{bottom-approximation-proposition}. 
The first  is a Gehring's type self-improving regularity estimates for weak solutions of degenerate equations.
\begin{lemma}  \label{Gerling-interior-bottom}  Let $\Lambda >0, M_0 \geq 1$ be fixed and $\theta \in (0,1)$. Then, there exists $\varrho = \varrho(\Lambda, M_0, n)>0$ sufficiently small such that for any $0 < R<2$, if \eqref{ELLIPTICITY-5} holds with some $\mu \in A_2(\R)$ and $[\mu]_{A_2} \leq M_0$, and if $g \in W^{1,2}(Q_R, \mu)$ is a weak solution of 
\begin{equation} \label{v-Q-r-bottom}
\left\{
\begin{array}{cccl}
\textup{div}[\A(X) \nabla g] & =& 0 & \quad \text{in} \quad Q_R, \\
\lim_{y\rightarrow 0^+}\wei{\A(X) \nabla g, e_{n+1}} & = & 0 &\quad \text{on} \quad B_R,
\end{array}
\right.
\end{equation}
then  there is $C =C(M_0, \Lambda, n, \theta) >0$ such that 
\[
\left(\frac{1}{\hat{\mu}(Q_{\theta R})}\int_{Q_{\theta R}} |\nabla g|^{2+\varrho} \mu(y) dX \right)^{\frac{1}{2+\varrho}} \leq C \left( \frac{1}{\hat{\mu}(Q_R)}\int_{Q_R} |\nabla g|^2 \mu(y) dX\right)^{1/2}.
\]
\end{lemma}
\begin{proof} Without loss of generality, we can assume $R =2$ and $\theta = 1/2$. We claim that there is some $ p \in (1,2)$ such that for every $0 < r <1$ and every $Z \in \overline{Q}_1$, 
\begin{equation} \label{Re-Holder-Gehring}
\left(\frac{1}{\hat{\mu}(D_{r/2} (Z))}\int_{D_{r/2} (Z) \cap Q_2} |\nabla g|^{2}\mu(y) dX \right)^{1/2} \leq  C(\Lambda, M_{0})\left(\frac{1}{\hat{\mu}(D_{r} (Z))} \int_{Q_{r} (Z)\cap Q_2}|\nabla g|^{p} \mu(y) dX \right)^{1/p}. 
\end{equation}
Once this is proved, our Lemma follows from the weighted reverse H\"{o}lder's inequality, see for instance \cite[Theorem 1.5]{Kinnunen}, or \cite[Theorem 2.3.3]{Stredulinsky}. We only give the proof of \eqref{Re-Holder-Gehring} for cylinders centered at $Z =  (z, 0)\in \overline{B_1} \times\{0\}$ since for other types of cylinders, the proof is essentially the same. In this case, note that $Q_r(Z) = D_r(Z) \cap Q_2$ and $Q_{r/2}(Z) = D_{r/2}(Z) \cap Q_2$.  Let $\phi \in C_0^\infty(B_r(z))$ be non-negative such that $\phi = 1$ in $B_{r/2}(z)$, $\phi =0$ on $B_{r}(z)\setminus B_{3r/4}(z)$, and $0\leq \phi \leq 1$. Also, let $\chi \in C^\infty(\mathbb{R})$ such that $\chi (y) =1$ for $y \leq r/2$ and $\chi(y) = 0$ for $y \geq r$. We write $\psi(X) = \phi(x) \chi(y)$. Use $(g-\bar{g}_{Q_{r} (Z)})\psi^2 \in \overset{*}W^{1,2}(Q_{r}, \mu)$ as a test function for the equation of \eqref{v-Q-r-bottom}, we obtain
\[
\int_{Q_{r} (Z)} \wei{\A \nabla g, \nabla g} \psi^2(X) dX = - \int_{Q_{r} (Z)} \wei{\A \nabla v, \nabla (\psi^2)} (g-\bar{g}_r) dX. 
\]
Then, by the ellipticity condition \eqref{ELLIPTICITY-5}, we see that
\[
 \Lambda^{-1} \int_{Q_{r} (Z)} |\nabla g|^2  \psi^2 \mu(y) dX \leq  \Lambda \int_{Q_{r} (Z)} |\nabla v| |\nabla \psi| |\psi||g-\bar{g}_{Q_{r} (Z)}| \mu(y) dX.
\]
Then, it follows from H\"{o}lder's inequality and Young's inequality that
\[
\begin{split}
\Lambda^{-1} \int_{Q_{r} (Z)} |\nabla g|^2  \psi^2 \mu(y)dX & \leq \frac{\Lambda^{-1}}{2}\int_{Q_{r} (Z)} |\nabla g|^2 \psi^2 \mu(y)dX + C(\Lambda) \int_{Q_{r} (Z)} |g -\bar{g}_{Q_{r} (Z)}|^2 |\nabla \psi|^2 \mu(y) dX \\
&  \leq \frac{\Lambda^{-1}}{2}\int_{Q_{r} (Z)} |\nabla v|^2 \psi^2 \mu(y) dX + \frac{C(\Lambda)}{r^2} \int_{Q_{r} (Z)} |g -\bar{g}_{Q_{r} (Z)}|^2  \mu(y) dX.
\end{split}
\]
Hence,
\begin{equation}\label{Poincare-Gehring}
\frac{1}{\hat{\mu}(D_{r} (Z))} \int_{Q_{r} (Z)} |\nabla g|^2  \psi(X)^2 \mu dX \leq  \frac{C(\Lambda)}{r^2 \hat{\mu}(D_{r} (Z))} \int_{Q_{r} (Z)} |g -\bar{g}_{Q_{r} (Z)}|^2  \mu(y) dX.
\end{equation}
Then, by the weighted Sobolev-Poincar\'{e}'s inequality for $A_2$-weights \cite[Theorem 1.3]{Fabes},  we can find some $1< p < 2$ such that
\begin{equation}\label{S-P-Gehring}
\left(\frac{1}{\hat{\mu}(D_{r} (Z))} \int_{Q_{r} (Z)} |g -\bar{g}_{Q_{r} (Z)}|^2  \mu(y)dX\right)^{1/2} \leq  r C(\Lambda, M_0)\left(\frac{1}{\hat{\mu}(D_{r} (Z))} \int_{Q_{r} (Z)} | \nabla g |^{p}  \mu(y) dX\right)^{1/p}.
\end{equation}
Combining  inequalities \eqref{Poincare-Gehring}, \eqref{S-P-Gehring} and using the doubling property of $\mu$ we obtain \eqref{Re-Holder-Gehring}.
\end{proof}
\noindent
%--------------------
We will use the above Gehring's type higher integrability result to set up a two-step approximation procedure. 
\begin{lemma} \label{energy-u-tilde-u} Let $\Lambda, M_0, \A, \mu$ be as Proposition \ref{bottom-approximation-proposition}. If  $u \in W^{1,2}(Q_2, \mu)$ be a weak solution of \eqref{u-Q-6-bottom}, then there is a weak solution $\tilde{u} \in W^{1,2}(Q_{7/4})$ of
\begin{equation} \label{v-1-Q-5-bottom}
\left\{
\begin{array}{cccl}
\textup{div} [\A(X) \nabla \tilde{u}] & =  & 0, & \quad \text{in} \quad Q_{7/4}, \\
 \tilde{u} & = &u, & \quad \text{on} \quad \partial Q_{7/4} \setminus B_{7/4} \times \{0\}\\
\lim_{y\rightarrow 0^+} \wei{\A(x, y) \nabla \tilde{u}(x, y), e_{n+1}} & = & 0, & \quad \text{on \quad $B_{7/4}$},
\end{array} 
\right.
\end{equation}
satisfying 
\begin{equation} \label{u-tilde-u}
\int_{Q_{7/4}} |\nabla u -\nabla \tilde{u}|^2 \mu(y) dX \leq C(\Lambda, M_0, n) \left[ \int_{Q_{7/4}} \Big| \F/\mu \Big|^2 \mu(y) dX +  \int_{Q_{7/4}} \Big|f(x)/\mu(y) \Big|^2 \mu(y)dX \right].
\end{equation}
Moreover, there is $\rho = \rho(\Lambda, M_0,n) >0$ such that
\begin{equation} \label{higher-reg-bottom-tilde-u}
\begin{split}
& \left( \frac{1}{\hat{\mu}(Q_{3/2})} \int_{Q_{3/2}} |\nabla \tilde{u}|^{2+\rho} \right)^{\frac{1}{2+\rho}} \\
& \leq C(n, M_0, \Lambda) \left[ \frac{1}{\hat{\mu}(Q_{7/4})} \left\{ \int_{Q_{7/4}} |\nabla u|^2 \mu(y) dX + \int_{Q_{7/4}} \Big| \F/\mu \Big|^2 \mu(y) dX +  \int_{Q_{7/4}} \Big| f(x)/\mu(y) \Big|^2 \mu(y)dX \right\}  \right]^{1/2}.
\end{split}
\end{equation}
\end{lemma}
\begin{proof} 
We begin by noting first that by weak solution to \eqref{v-1-Q-5-bottom} we mean 
if $\tilde{u} -u \in \overset{*}{W}^{1,2}(Q_{7/4}, \mu)$ and
\[
\int_{Q_{7/4}} \wei{\A\nabla \tilde{u}, \nabla \psi} dX = 0, \quad \forall \psi \in \overset{*}{W}^{1,2}(Q_{7/4}, \mu).
\]
Observe that from the  trace inequaly, Lemma \ref{trace-zero}, and for a given $u$ as in the lemma, we can follow \cite[Theorem 2.2]{Fabes} to prove the existence and uniqueness of weak solution $\tilde{u}$ for the equation \eqref{v-1-Q-5-bottom}. Now, we write $g = u -\tilde{u}$. By the definition, $g \in \overset{*}{W}^{1,2}(Q_{7/4}, \mu)$. Moreover, $g$ is a weak solution of
\begin{equation} \label{g-Q-5-bottom}
\left\{
\begin{array}{cccl}
\text{div} [\A(X) \nabla g] & = & \text{div}[{\bf F}], & \quad \text{in} \quad Q_{7/4}, \\
 g & =& 0, & \quad \text{on} \quad \partial Q_{7/4} \setminus (B_{7/4}\times \{0\}), \\
\lim_{y \rightarrow 0^+}\wei{\A(x, y) \nabla g(x, y) -{\bf F}(x, y), e_{n+1}} & = & f(x), & \quad x \in B_{7/4}.
\end{array}
\right.
\end{equation}
Then, use $g$ as a test function of its equation, we obtain that 
\[
\int_{Q_{7/4}} \wei{\A\nabla g, \nabla g} dX = \int_{Q_{7/4}} \wei{ {\bf F}, \nabla g} dX + \int_{B_{7/4}} g(x,0) f(x) dx.
\]
From the  ellipticity condition \eqref{ELLIPTICITY-5}, it follows that 
\[
 \int_{Q_{7/4}} |\nabla g|^2 \mu(y) dX \leq 
C(\Lambda) \left[ \int_{Q_{7/4}} \Big | \F/\mu \Big| |\nabla g| \mu(y) dX + \int_{B_{7/4}} |g(x,0)| |f(x)| dx \right]. 
\]
Then, applying the H\"{o}lder's inequality and Young's inequality, we see that
\[
\begin{split}
\int_{Q_{7/4}} |\nabla g|^2 \mu(y) dX  & \leq \beta
\left[ \int_{Q_{7/4}} |\nabla g|^2\mu(y) dX + \frac{\mu(\Gamma_{7/4})}{2 (7/4)^{2}}\int_{B_{7/4}} |g(x,0)|^2 dx \right]  \\
&+ \frac{C(\Lambda)}{4\beta} \left[ \int_{Q_{7/4}} \Big| \F/\mu \Big|^2 \mu(y) dX + \frac{2 (7/4)^{2}}{\mu(\Gamma_{7/4})}\int_{B_{7/4}} |f(x)|^2 dx \right],
\end{split}
\]
with any $\beta >0$. Then, it follows from the trace inequality, Lemma \ref{trace-zero}, that
\[
\Lambda \int_{Q_{7/4}} |\nabla g|^2 \mu(y) dX  \leq \beta C(n, M_0)
\int_{Q_{7/4}} |\nabla g|^2\mu(y) dX + \frac{C(\Lambda, M_0)}{4\beta} \left[ \int_{Q_{7/4}} \Big| \F/\mu \Big|^2 \mu(y) dX + \frac{ (7/4)^{2}}{\mu(\Gamma_{7/4})}\int_{B_{7/4}} |f(x)|^2 dx \right].
\]
Then, with $\beta$ sufficiently such that  $\beta  C(n, M_0) \leq 1/2$, 
we obtain
\begin{equation} \label{g-energy-last-step}
\int_{Q_{7/4}} |\nabla g|^2 \mu(y) dX \leq C(\Lambda, n, M_0) \left[ \int_{Q_{7/4}} \Big| \F/\mu \Big|^2 \mu(y) dX +  
(7/4)^{2}\left(\int_{0}^{7/4} \mu (y)dy\right)^{-1} \int_{B_{7/4}} |f(x)|^2 dx \right].
\end{equation}
Observe that  by H\"{o}lder's inequality, 
\[
7/4 = \int_0^{7/4} dy  \leq  \left(\int_0^{7/4} \mu^{-1} (y) dy \right)^{1/2} \left(\int_0^{7/4} \mu(y)dy\right)^{1/2}.
\]
Therefore,
\[
\frac{(7/4)^2}{\mu(\Gamma_{7/4})} \int_{B_{7/4}} |f(x)|^2 dx \leq \int_{Q_{7/4}} \Big |f(x)/\mu(y) \Big|^2 \mu(y) dX.
\]
This together with \eqref{g-energy-last-step} prove the estimate \eqref{u-tilde-u} in the lemma. The estimate \eqref{higher-reg-bottom-tilde-u} follows directly from \eqref{u-tilde-u}, Lemma \ref{Gerling-interior-bottom}, and Lemma \ref{doubling}. 
\end{proof}
%============
\begin{lemma} \label{energy-u-v} Let $\tilde{u}$ and $\varrho$ be  as in Lemma \ref{energy-u-tilde-u}. Then, there is $v \in W^{1,2}(B_{3/2}, \mu)$  a weak solution of
\begin{equation} \label{v-Q-4-bottom}
\left\{
\begin{array}{cccl}
\textup{div}[\langle {\A}\rangle_{B_{3/2}}(y) \nabla v] & = & 0, & \quad \text{in} \quad Q_{3/2}, \\
v & =  &\tilde{u},  & \quad \text{on} \quad \partial Q_{3/2}\setminus B_{3/2} \times \{0\}, \\
\lim_{y\rightarrow 0^+} \wei{\langle{\A}\rangle_{B_{3/2}}(y) \nabla v(x, y), e_{n+1}} & =  &0, & \quad \text{on $B_{3/2}$},
\end{array} \right.
\end{equation}
such that 
\[
\begin{split}
 &\left (\frac{1}{\hat{\mu}(Q_{3/2})} \int_{Q_{3/2}} |\nabla \tilde{u} -\nabla v|^2 \mu(y) dX\right)^{1/2}  \\
& \leq C(\Lambda) \left(\frac{1}{\hat{\mu}(Q_{3/2})} \int_{Q_{3/2}}  |\nabla \tilde{u}|^{2+\varrho} \mu(y)\right)^{\frac{1}{2+\varrho}} \left(\frac{1}{\hat{\mu}(Q_{3/2})} \int_{Q_{3/2}}\Big |\A - \langle{\A}\rangle_{B_{3/2}}(y)\Big|^2\mu(y)^{-1} dX  \right)^{\frac{\varrho}{2(2+\varrho)}}.
 \end{split}
 \]
\end{lemma}
\begin{proof} Observe that by the trace inequality Lemma \ref{trace}, and for a given $u, \tilde{u}$ as in the lemma, we can follow \cite[Theorem 2.2]{Fabes} to prove the existence and uniqueness of weak solution $v$ of \eqref{v-Q-4-bottom}. Now, let $w = v - \tilde{u}$. We note that $w \in \overset{*}{W}^{1,2}(Q_{3/2}, \mu)$. Therefore, by using ${w}$ as a test function for the equation of $v$, and $\tilde{u}$, we obtain the following
\[
\int_{Q_{3/2}} \wei{\langle{\A}\rangle_{B_{3/2}}(y) \nabla w, \nabla w} dX = -\int_{Q_{3/2}} \wei{[\A -\langle{\A}\rangle_{B_{3/2}}(y)]\nabla \tilde{u}, \nabla {w}} dX.
\]
This, the ellipticity condition \eqref{ELLIPTICITY-5}, and H\"{o}lder's inequality imply that
\[
\begin{split}
 \int_{Q_{3/2}} |\nabla w|^2 \mu(y) dX & \leq C(\Lambda) \int_{Q_{3/2}} \Big |\A - \langle{\A}\rangle_{B_{3/2}}(y)\Big| |\nabla \tilde{u}| |\nabla w| dX \\
 &\leq C(\Lambda) \left(\int_{Q_{3/2}} |\nabla w|^2 \mu(y) dX \right)^{1/2} \left( \int_{Q_{3/2}} \Big |\A - \langle{\A}\rangle_{B_{3/2}}(y)\Big|^2 |\nabla \tilde{u}|^2 \mu(y)^{-1} dX\right)^{1/2}.
\end{split}
\]
Hence,
\[
\frac{1}{\hat{\mu}(Q_{3/2})}\int_{Q_{3/2}} |\nabla w|^2 \mu(y) dX  \leq \frac{C(\Lambda)}{\hat{\mu}(Q_{3/2})}\int_{Q_{3/2}} \Big (|\A - \langle{\A}\rangle_{B_{3/2}}(y)| \mu(y)^{-1}\Big)^2 |\nabla \tilde{u}|^2 \mu(y) dX.
\]
Then, applying the  H\"{o}lder's inequality with the exponents $\frac{2+\varrho}{2}$ and $\frac{2+\varrho}{\varrho}$, we obtain
\[
\begin{split}
 &\left (\frac{1}{\hat{\mu}(Q_{3/2})} \int_{Q_{3/2}} |\nabla w|^2 \mu(y) dX\right)^{1/2}  \\
 & \leq C(\Lambda) \left(\frac{1}{\hat{\mu}(Q_{3/2})} \int_{Q_{3/2}}  |\nabla \tilde{u}|^{2+\varrho} \mu(y)dX\right)^{\frac{1}{2+\varrho}} \left(\frac{1}{\hat{\mu}(Q_{3/2})} \int_{Q_{3/2}}\Big |(\A - \langle{\A}\rangle_{B_{3/2}}(y)) \mu(y)^{-1}\Big|^{\frac{2(2+\varrho)}{\varrho}} \mu(y) dX  \right)^{\frac{\varrho}{2(2+\varrho)}}.
 \end{split}
\]
Observe that it follows from the ellipticity condition \eqref{ELLIPTICITY-5} that
\[
|\A(X) - \langle{\A}\rangle_{B_{3/2}}(y)| \mu(y)^{-1} \leq \Lambda^{-1}. 
\]
As a consequence,  the right hand side of the previous inequality can be simplified to  
\[
\begin{split}
 &\left (\frac{1}{\hat{\mu}(Q_{3/2})} \int_{Q_{3/2}} |\nabla w|^2 \mu(y) dX\right)^{1/2}  \\
& \leq C(\Lambda) \left(\frac{1}{\hat{\mu}(Q_{3/2})} \int_{Q_{3/2}}  |\nabla \tilde{u}|^{2+\varrho} \mu(y)\right)^{\frac{1}{2+\varrho}} \left(\frac{1}{\hat{\mu}(Q_{3/2})} \int_{Q_{3/2}}\Big |\A - \langle{\A}\rangle_{B_{3/2}}(y)\Big|^2\mu(y)^{-1} dX  \right)^{\frac{\varrho}{2(2+\varrho)}}, 
 \end{split}
\]
which  completes the proof.
\end{proof}
We now are ready to prove Proposition \ref{bottom-approximation-proposition}.
\begin{proof}[Proof of Proposition \ref{bottom-approximation-proposition}] It follows from estimate \eqref{higher-reg-bottom-tilde-u} in Lemma \ref{energy-u-tilde-u}, the assumptions, and Lemma \ref{doubling} that
\[
\left(\frac{1}{\hat{\mu}(Q_{3/2})} \int_{Q_{3/2}}  |\nabla \tilde{u}|^{2+\varrho} \mu(y)\right)^{\frac{1}{2+\varrho}} \leq C(n, \Lambda, M_0).
\]
Observe also that
\[
\begin{split}
& \frac{1}{\hat{\mu}(Q_{3/2})} \int_{Q_{3/2}} |\nabla u - \nabla v|^2 \mu(y) dX \\
& \leq \frac{1}{\hat{\mu}(Q_{3/2})} \int_{Q_{3/2}} |\nabla u - \nabla \tilde{u}|^2 \mu(y) dX + \frac{1}{\hat{\mu}(Q_{3/2})} \int_{Q_{3/2}} |\nabla \tilde{u} - \nabla v|^2 \mu(y) dX \\
 & \leq \frac{\hat{\mu}(Q_{7/4})}{ \hat{\mu}(Q_{3/2})} \frac{1}{\hat{\mu}(Q_{7/4})}\int_{Q_{7/4}} |\nabla u - \nabla \tilde{u}|^2 \mu(y) dX + \frac{1}{\hat{\mu}(Q_{3/2})} \int_{Q_{3/2}} |\nabla \tilde{u} - \nabla v|^2 \mu(y) dX\\
& \leq \frac{C(n, M_0)}{\hat{\mu}(Q_{7/4})}\int_{Q_{7/4}} |\nabla u - \nabla \tilde{u}|^2 \mu(y) dX + \frac{1}{\hat{\mu}(Q_{3/2})} \int_{Q_{3/2}} |\nabla \tilde{u} - \nabla v|^2 \mu(y) dX.
\end{split}
\]
Hence, it follows from \eqref{u-tilde-u} of Lemma \ref{energy-u-tilde-u}, Lemma \ref{doubling}, and Lemma \ref{energy-u-v} that
\[
\begin{split}
& \frac{1}{\hat{\mu}(Q_{3/2})} \int_{Q_{3/2}} |\nabla u - \nabla v|^2 \mu(y) dX\\
&\leq  C(n, \Lambda, M_0)\left[\frac{1}{\hat{\mu}(D_{7/4})} \int_{D_{7/4}} \Big| \F(X)/\mu(y) \Big|^2 \mu(y) dX +  \frac{1}{\hat{\mu}(D_{7/4})}\int_{D_{7/4}} \Big| f(x)/\mu(y) \Big|^2 dX\right.\\
&\quad\quad\quad\quad\quad\left. + \left(\frac{1}{\hat{\mu}(D_{3/2})} \int_{D_{3/2}}\Big |\A - \langle{\A}\rangle_{B_{3/2}}(y)\Big|^2\mu(y)^{-1} dX  \right)^{\frac{\varrho}{(2+\varrho)}}\right]  \leq C(n, \Lambda, M_0) \delta^{2}.
\end{split}
\]
Hence, if we chose $\delta$ sufficiently small so that $\delta^{2}\, C(n, \Lambda, M_0) < \epsilon$, we obtain
\[
 \frac{1}{\hat{\mu}(Q_{3/2})} \int_{Q_{3/2}} |\nabla u - \nabla v|^2 \mu(y) dX \leq \epsilon.
\]
This proves \eqref{u-v-approximation-bottom}. It remains to prove \eqref{v-L-infty-bottom}. From the last estimate, and the assumptions in the Proposition \ref{bottom-approximation-proposition}, it follows that 
\[
\begin{split}
 \frac{1}{\hat{\mu}(Q_{3/2})} \int_{Q_{3/2}} |\nabla v|^2 \mu(y) dX  & \leq \epsilon +  \frac{1}{\hat{\mu}(Q_{3/2})} \int_{Q_{3/2}} |\nabla u|^2 \mu(y) dX \\
 & \leq \epsilon + \frac{C(n, M_0) }{\hat{\mu}(Q_{7/4})} \int_{Q_{7/4}} |\nabla u|^2 \mu(y) dX  \leq \epsilon + C(n, M_0).
 \end{split}
\]
From this, and the assumption that $\A \in \mathcal{A}(Q_2, \Lambda, M_0, \mu)$, we derive
\[
\norm{\nabla v}_{L^\infty(Q_{5/4})} \leq C(n, \Lambda, M_0) \left[ \frac{1}{\hat{\mu}(Q_{3/2})} \int_{Q_{3/2}} |\nabla v|^2 \mu(y) dX \right]^{1/2} \leq C(n, \Lambda, M_0).
\]
The proof of Proposition \ref{bottom-approximation-proposition} is the complete.
\end{proof}

%====================
\subsection{Interior approximation estimates} 
In this subsection we obtain an approximation estimate for $\nabla u$ over cylinders which are completely contained in $Q_{2}$. 
The main result of the section now can be stated in the following proposition.

\begin{proposition} \label{interior-approximation-proposition} 
Let $\Lambda >0, M_0 >0$ be fixed. Then, for every $\epsilon >0$, there exists $\delta = \delta(\epsilon, \Lambda, M_0, n) >0$ with the following property. 
 Let $\mu \in A_2(\R)$ with $[\mu]_{A_2} \leq M_0$, and let $\A \in \mathcal{A}(Q_{2}, \Lambda, M_{0}, \mu)$ satisfy \eqref{ELLIPTICITY-5}. Then for every $X_0 = (x_0, y_0) \in \overline{B_1} \times [0,1]$, $0 < r <1$, such that $D_{2r}(X_{0})\subset Q_{2}$, if  
\[
\begin{split}
& \frac{1}{\hat{\mu}(D_{3r/2}(X_0))} \int_{D_{3r/2}(X_0)} \Big |\A(X) - \langle{\A}\rangle_{B_{3r/2}(x_0)}(y)\Big|^2 \mu^{-1}(y) dX  \leq \delta^{2}, \\
& \frac{1}{\hat{\mu}(D_{2r}(X_0))} \int_{D_{2r}(X_0)} \Big| \F/\mu\Big|^2 \mu(y) dX   \leq \delta^{2}.
\end{split}
\]
then for every weak solution $u \in W^{1,2}(D_{2r}(X_0), \mu)$  of 
\begin{equation} \label{loc-interior-maineqn}
\textup{div}[\A(X) \nabla u(X)] = \textup{div}[\F], \quad \text{in} \quad D_{2r}(X_0),
\end{equation}
satisfying 
\[
\frac{1}{\hat{\mu}(D_{2r}(X_0))} \int_{D_{2r}(X_0)} |\nabla u|^2 \mu(y) dX \leq 1,
\]
 there is  some $v \in W^{1,2}(D_{3r/2}(X_0), \mu)$  
% of 
%\[
%\textup{div}[\langle{\A}\rangle_{B_{3r/2}(x_0)}(y) \nabla v] = 0, \quad \text{in} \quad D_{3r/2}(X_0),
%\] 
such that
\begin{equation} \label{u-v-approximation-interior}
\frac{1}{\hat{\mu}(D_{3r/2}(X_0))} \int_{D_{3r/2}(X_0)} |\nabla u - \nabla v|^2 \mu(y) dX \leq \epsilon^{2}.
\end{equation}
Moreover, there is a constant $C(n,\Lambda, M_0)>0$ such that
\begin{equation} \label{v-L-infty-interior}
 \norm{\nabla v}_{L^\infty(D_{5r/4}(X_0))}  \leq C(n, \Lambda, M_0).
\end{equation}

\end{proposition}
As in the previous section, the proof relies on self-improving regularity estimate of Gehring's type whose proof can be done in a similar way as Lemma \ref{Gerling-interior-bottom}. 
\begin{lemma}  \label{Gerling-interior} Let $M_0 \geq 1$, and suppose that $\mu \in A_2(\R)$ with $[\mu]_{A_2} \leq M_0$ and \eqref{ELLIPTICITY-5} holds. There exists $\varrho = \varrho(\Lambda, M_0, n)>0$ sufficiently small such that for any $0< R \leq 2$, if $u \in W^{1,2}(D_R, \mu)$ is a weak solution of
\[
\begin{array}{cccl}
\textup{div} [\A(X) \nabla u] & = & 0, & \quad \text{in} \quad D_{R},
\end{array}
\]
then for every $\theta \in (0,1)$, we have $\nabla u \in L^{2 + \varrho}(D_{\theta R}, \mu)$ and  
\[
\left(\frac{1}{\hat{\mu}(D_{\theta R})}\int_{D_{\theta R}} |\nabla u|^{2+\varrho} \mu(y) dX \right)^{\frac{1}{2+\varrho}} \leq C \left( \frac{1}{\hat{\mu}(D_R)}\int_{D_R} |\nabla u|^2 \mu(y) dX\right)^{1/2},
\]
for some $C =C(\Lambda, M_0, n, \theta) >0$. 
\end{lemma}\noindent
\begin{proof}[Proof of Proposition \ref{interior-approximation-proposition}] We can assume without loss of generality that $X_0 = (0,0) \in \R^{n+1}$, and $r =1$. As before, we use a two step approximation procedure. Our first step is to approximate $u$ by the 
weak solution $\tilde{u} \in W^{1,2}(D_{7/4}, \mu)$ of
\begin{equation} \label{interior-step1-approx-eqn}
\left\{
\begin{array}{cccl} 
\text{div}(\A(X) \nabla \tilde{u})  & =  &0, & \quad \text{in} \quad  D_{7/4}, \\
\tilde{u}  & = & u, & \quad \text{on} \quad \partial D_{7/4}.
\end{array} \right.
\end{equation}
Then, we approximate $\tilde{u}$ by a weak solution $v \in W^{1, 2}(D_{3/2}, \mu)$
of the equation
\begin{equation}
\left\{
\begin{array}{cccl}
\text{div}(\langle\A \rangle_{B_{3/2}}(y ) \nabla v) & = & 0,  & \quad \text{on} \quad  D_{3/2}, \\
v & = &\tilde{u}, & \quad \text{on}\quad \partial D_{3/2}.
\end{array} \right.
\end{equation}
We note that the existence and uniqueness of solutions $\tilde{u}$ and $v$ are attained by 
\cite[Theorem 2.2]{Fabes}. Also, by using $\tilde{u} -u$ as a test function for \eqref{interior-step1-approx-eqn} and Lemma \ref{doubling}, we obtain
\begin{equation} \label{tilde-u-gradien-L2-inte}
\frac{1}{\hat{\mu}(D_{7/4})}\int_{D_{7/4}} |\nabla \tilde{u}|^2 \mu(y) dX \leq \frac{C(\Lambda, M_0)}{\hat{\mu}(D_{2})} \int_{D_{2}} |\nabla u|^2 \mu dX \leq C(\Lambda, M_0).
\end{equation}
Moreover, it is simple to check that $u - \tilde{u}\in W^{1,2}_0(\mu)$ 
is a weak solution of
\[
\left\{
\begin{split}
\text{div}(\A(X) \nabla (\tilde{u} - u)) & = \text{div} {\bf F}, &\quad  \text{in} \,\, D_{7/4}, \\
\tilde{u} &= u, & \quad \text{on} \,\, \partial D_{7/4}.
\end{split} \right.  
\] 
Therefore, from the energy estimate and Lemma \ref{doubling}, there is constant $C = C(\Lambda)$ such that 
\begin{equation} \label{interior-energy-u-tilde-u}
\frac{1}{\hmu(D_{7/4})}\int_{D_{7/4}} |\nabla u (X) - \nabla \tilde{u} (X) |^2 \mu(y)d X  \leq  \frac{C}{\hmu(D_{2})} \int_{D_2} \left|\F/\mu \right|^2 \mu (y)d X\leq C \delta^2.
\end{equation}
Similarly, $v-\tilde{u} \in W^{1, 2}_{0}(D_{3/2}, \mu)$ is a weak solution of  
\[
\left\{
\begin{array}{cccl}
\text{div}[\langle\A \rangle_{B_{3/2}}(y ) \nabla (v-\tilde{u})] &= &\text{div}[(\A(X) - \langle\A \rangle_{B_{3/2}}(y)) \nabla \tilde{u}],  &\quad  \text{in}\quad D_{3/2},  \\
v & = & \tilde{u},  & \quad \text{on} \quad \partial D_{3/2}.
\end{array} \right.
\]
Hence, 
\begin{equation} \label{interior-energy-v-tilde-u}
\frac{1}{\hmu(D_{3/2})}\int_{D_{3/2}} |\nabla v - \nabla \tilde{u} |^2 \mu(y) d X  \leq \frac{C}{\hmu(D_{3/2})}\int_{D_{3/2}} \left|(\A(X) - \langle \A \rangle_{B_{3/2}}(y )) \nabla \tilde{u} \right|^2 \mu^{-1}(y) d X. 
\end{equation}
By Lemma \ref{Gerling-interior}, and \eqref{tilde-u-gradien-L2-inte}, we obtain
\[
\left(\frac{1}{\hmu(D_{3/2})}\int_{D_{3/2}} |\nabla \tilde{u}|^{2 + \varrho} \mu (y) d X \right)^{1/(2 + \varrho)} \leq  C\left(\frac{1}{\hmu(D_{7/4})} \int_{D_{7/4}} |\nabla \tilde{u}|^{2} \mu(y) dX\right)^{1/2} \leq  C(\Lambda, M_0).
\]
From this, we now apply H\"{o}lder's inequality for \eqref{interior-energy-v-tilde-u} with the pair $\frac{2 + \varrho}{2}$ and $\frac{2 + \varrho}{\varrho}$ to obtain that 
\[
\begin{split}
\frac{1}{\hmu(D_{3/2})}&\int_{D_{3/2}} |\nabla v - \nabla \tilde{u} |^2 \mu(y) dX  \leq  \frac{C}{\hmu(D_{3/2})}\int_{D_{3/2}} \left|(\A(X) - \langle \A \rangle_{B_{3/2}}(y)) \nabla \tilde{u} \right|^2 \mu^{-1}(y) dX\\
& \leq  \frac{C}{\hmu(D_{3/2})}\int_{D_{3/2}} \left|(\A(X) - \langle \A \rangle_{B_{3/2}}(y))\mu^{-1}|^{2} |\nabla \tilde{u}(X) \right|^2 \mu(y) dX  \\
& \leq C \left(\frac{1}{\hmu(D_{3/2})} \int_{D_{3/2}} |(\A(X) - \langle \A \rangle_{B_{3/2}}(y))\mu^{-1}(y)|^{\frac{2(2 + \varrho)}{\varrho}}\mu(y) dX \right)^{\frac{\varrho}{2 + \varrho}}.
\end{split}
\] 
We now use a fact that follows from the ellipticity \eqref{ELLIPTICITY-5}
\[
|\A(X) - \langle\A\rangle_{B_{3/2}}(y)| \leq 2\Lambda \mu(y),\quad \text{for a.e.\ $X = (x, y) \in D_{2}$},
\]
and the above higher integrability estimate to write the right hand side as 
\[
\begin{split}
\frac{1}{\hmu(D_{3/2})}&\int_{D_{3/2}} |\nabla v - \nabla \tilde{u} |^2 \mu(y)dX\leq C  \left(\frac{1}{\hmu(D_{3/2})}\int_{D_{3/2}} |(\A(X) - \langle \A \rangle_{B_{3/2}}(y))|^{2} \mu^{-1}(y) dX  \right)^{\frac{\varrho}{2 + \varrho}}. 
\end{split}
\]
From this estimate, and \eqref{interior-energy-u-tilde-u}, we conclude that given $\epsilon > 0$, we can choose $\delta > 0$ such that 
\begin{equation} \label{interior-u-v.est}
\begin{split}
& \frac{1}{\hmu(D_{3/2})}\int_{D_{3/2}} |\nabla v  -\nabla u | ^{2} \mu(y) dX \\
&\leq \frac{1}{\hmu(D_{3/2})}\int_{D_{3/2}} |\nabla v  -\nabla \tilde{u} | ^{2} \mu(y) dX + \frac{1}{\hmu(D_{3/2})}\int_{D_{3/2}} |\nabla \tilde{u}  -\nabla u | ^{2} \mu(y) dX\\
 &\leq C (\delta^2 + \delta ^{\frac{2}{2 + \varrho}}) \leq \epsilon^{2}. 
\end{split}
\end{equation}
To prove \eqref{v-L-infty-interior}, we first observe that we can assume $\epsilon <1$. Then, it follows from  \eqref{interior-u-v.est}, our assumption, and Lemma \ref{doubling} that 
\[
\begin{split}
\frac{1}{\hmu(D_{3/2})} \int_{D_{3/2}} |\nabla v(X)|^2 \mu(y) dX & \leq \epsilon^2 + \frac{1}{\hmu(D_{3/2})} \int_{D_{3/2}} |\nabla u (X)|^{2} \mu(y) dX \\
& \leq 1 + \frac{C(M_0)}{\hmu(D_{2})} \int_{D_{2}} |\nabla u(X)|^{2} \mu(y) dX \leq C(M_0).
\end{split}
\]
Then, the estimate \eqref{v-L-infty-interior} follows from the assumption that $\A\in \mathcal{A}(Q_2, \Lambda, M_{0}, \mu)$
and the inequality that 
\[
\|\nabla v\|_{L^{\infty} (D_{5/4})} \leq C(\Lambda, M_0, n) \left(\frac{1}{\hmu(D_{3/2})} \int_{D_{3/2}} |\nabla v|^{2}(X) \mu(y) dX\right)^{1/2} \leq C (\Lambda, M_0,n). 
\]
\end{proof}

\subsection{Approximation estimates up to the flat side of the half-cylinder domain } 
In this subsection we obtain an approximation estimate for $\nabla u$ over half cylinders whose flat side overlaps with the flat side of $Q^{+}_{2}$. 
The main result of the section now can be stated in the following proposition.

\begin{proposition} \label{interior-approximation-proposition-halfCyl} 
Let $\Lambda >0, M_0 >0$ be fixed. Then, for every $\epsilon >0$, there exists $\delta = \delta(\epsilon, \Lambda, M_0, n) >0$ with the following property. 
 Let $\mu \in A_2(\R)$ such that $[\mu]_{A_2} \leq M_0$, and let $\A \in \mathcal{B}(Q_{2}^{+}, \Lambda, M_{0}, \mu)$ satisfy \eqref{ELLIPTICITY-5}. 
Then for every $ X_0 = (x_0, y_0) \in \overline{B}_1^+ \times [0, 1]$, where $x_0 = (x_0', 0) \in \overline{B}^+_1$ , $0 < r <1$ such that $D^{+}_{2r}(X_0)\subset Q_{2}^{+}$, if 
\[
\begin{split}
& \frac{1}{\hat{\mu}(D_{3r/2}(X_0))} \int_{D^{+}_{3r/2}(X_0)} |\A(X) - \langle{\A}\rangle_{B_{3r/2}^+}(y)|^2 \mu^{-1}(y) dX  \leq \delta^{2},\quad \frac{1}{\hat{\mu}(D_{2r}(X_0))} \int_{D^{+}_{2r}(X_0)} \Big| \F/\mu\Big|^2 \mu(y) dX   \leq \delta^{2}, 
\end{split}
\]
and    $u \in W^{1,2}(D^{+}_{2r}(X_0), \mu)$ is a weak solution of
\begin{equation*}
\left\{
\begin{array}{cccl}
\emph{div}[\A(X) \nabla u]  & = & \emph{div}({\bf F}), &\quad \text{ in} \quad  D^{+}_{2r}(X_{0}),\\
u  &= & 0, & \quad \text{on} \quad  T_{2r}(x'_{0})\times I_{2r}(y_{0}),
\end{array} \right.
\end{equation*}
satisfying 
\[
\frac{1}{\hat{\mu}(D_{2r}(X_0))} \int_{D^{+}_{2r}(X_0)} |\nabla u|^2 \mu(y) dX \leq 1,
\]
then there exists $v \in W^{1,2}(D^{+}_{3r/2}(X_0), \mu)$ such that 
\begin{equation} \label{u-v-approximation-interior-halfCyl}
\frac{1}{\hat{\mu}(D_{3r/2}(X_0))} \int_{D^{+}_{3r/2}(X_0)} |\nabla u - \nabla v|^2 \mu(y) dX \leq \epsilon^{2}.
\end{equation}
Moreover, there is a constant $C(\Lambda, M_0,n)>0$ such that
\begin{equation} \label{v-L-infty-interior-halfCyl}
 \norm{\nabla v}_{L^\infty(D^{+}_{5r/4}(X_0))}  \leq C(\Lambda, M_0,n).
\end{equation}

\end{proposition}

We begin by stating a result on self-improving regularity estimates of Gehring's type, whose proof is similar to that of Lemma \ref{Gerling-interior-bottom}. 
\begin{lemma}  \label{Gerling-interior-half} Suppose that \eqref{ELLIPTICITY-5} holds with $[\mu]_{A_2} \leq M_0$. There exists $\varrho = \varrho(\Lambda, M_0,n)>0$ sufficiently small such that for any $R>0$, $\theta \in (0,1)$, if $\tilde{u} \in W^{1,2}(D^{+}_R, \mu)$ is a weak solution of 
\[
\left\{
\begin{array}{cccl}
\textup{div} [\A(X) \nabla \tilde{u}] & = & 0, & \quad \text{in} \quad D^{+}_{R},\\
\tilde{u}&=&0&\quad \text{on }\quad T_{R}\times I_{R},
\end{array} \right.
\]
then  we have $\nabla \tilde{u} \in L^{2 + \varrho}(D^{+}_{\theta R}, \mu)$ and there is $C =C(\Lambda, M_0, n, \theta) >0$ such that
\[
\left(\frac{1}{\hat{\mu}(D^+_{\theta R})}\int_{D^{+}_{\theta R}} |\nabla \tilde{u}|^{2+\varrho} \mu(y) dX \right)^{\frac{1}{2+\varrho}} \leq C \left( \frac{1}{\hat{\mu}(D^{+}_R)}\int_{D^{+}_R} |\nabla \tilde{u}|^2 \mu(y) dX\right)^{1/2}.
\]
 \end{lemma}
\noindent
%====================

\begin{proof}[Proof of Proposition \ref{interior-approximation-proposition-halfCyl}] We again assume without loss of generality that $x_0 =0, y_0 =0$ and $r =1$.
We use a two step approximation process. Let $\tilde{u} \in W^{1, 2}(D^{+}_{7/4}, \mu)$ be a weak solution of  
\begin{equation*}
\left\{\begin{array}{cccl}
\text{div}(\A(X) \nabla \tilde{u}) &= &0, & \quad  \text{in $D^{+}_{7/4}$}, \\
\tilde{u} &= & u, &\quad \text{on $\partial D^{+}_{7/4}$},
\end{array}
\right. 
\end{equation*}
and $v \in W^{1, 2}(D_{3/2}^{+}, \mu)$ be a weak solution of
\begin{equation}\label{siding-v-eqn}
\left\{\begin{array}{cccl}
\text{div}(\langle\A \rangle_{B_{3/2}^+}(y ) \nabla v) &= & 0, & \quad  \text{in $D^{+}_{3/2}$}\\
v&= & \tilde{u}, & \quad \text{on $\partial D^{+}_{3/2}$}. 
\end{array}
\right. 
\end{equation}
We note that the existence and uniqueness of $\tilde{u}, v$ are established in \cite[Theorem 2.2]{Fabes}. By the energy estimates for $\tilde{u}$, and Lemma \ref{doubling}, we see that
\[
\begin{split}
& \frac{1}{\hmu(D_{7/4})}\int_{D^{+}_{7/4}} |\nabla \tilde{u}|^{2}dX \leq \frac{C(\Lambda, M_0}{\mu(D_2)} \int_{D_2^+} |\nabla u|^2 dX \leq C(\Lambda, M_0), \\
& \frac{1}{\hmu(D_{7/4})}\int_{D^{+}_{7/4}} |\nabla u (X) - \nabla \tilde{u} (X) |^2 \mu(y)d X  \leq  \frac{C(\Lambda, M_0)}{\hmu(D_{2})} \int_{D^{+}_2} \left|\frac{{\bf F}}{\mu}\right|^2 \mu (y)d X\leq C \delta^2,
\end{split}
\]
and similarly,
\begin{equation} \label{siding-energy-v-tilde-u}
\frac{1}{\hmu(D_{3/2})}\int_{D^{+}_{3/2}} |\nabla v - \nabla \tilde{u} |^2 \mu(y) d X  \leq  \frac{C}{\hmu(D_{3/2})}\int_{D^{+}_{3/2}} \left|(\A(X) - \langle \A \rangle_{B_{3/2}^+}(y )) \nabla \tilde{u} \right|^2 \mu^{-1}(y) d X. 
\end{equation}
Then, by Lemma \ref{Gerling-interior-half}, 
\[
\left(\frac{1}{\hmu(D_{3/2})}\int_{D^{+}_{3/2}} |\nabla \tilde{u}|^{2 + \varrho} \mu (y) d X \right)^{1/(2 + \varrho)} \leq \left(\frac{C}{\hmu(D_{7/4})} \int_{D^{+}_{7/4}} |\nabla \tilde{u}|^{2} \mu(y) dX\right)^{1/2} \leq  C (\Lambda, M_0, n). 
\]
We now apply H\"{o}lders inequality with the pair $\frac{2 + \varrho}{2}$ and $\frac{2 + \varrho}{\varrho}$ to obtain from \eqref{siding-energy-v-tilde-u} that 
\[
\begin{split}
\frac{1}{\hmu(D_{3/2})}&\int_{D^{+}_{3/2}} |\nabla v - \nabla \tilde{u} |^2 \mu(y) dX  \leq \frac{C}{\hmu(D_{3/2})}\int_{D^{+}_{3/2}} \left|(\A(X) - \langle \A \rangle_{B_{3/2}^+}(y)) \nabla \tilde{u} \right|^2 \mu^{-1}(y) dX\\
& \leq  \frac{C}{\hmu(D_{3/2})}\int_{D^{+}_{3/2}} \left|(\A(X) - \langle \A \rangle_{B_{3/2}^+}(y))\mu^{-1}|^{2} |\nabla \tilde{u}(X) \right|^2 \mu(y) dX  \\
& \leq C \left(\frac{1}{\hmu(D_{3/2})} \int_{D^{+}_{3/2}} |(\A(X) - \langle \A \rangle_{B_{3/2}^+}(y))\mu^{-1}(y)|^{\frac{2(2 + \varrho)}{\varrho}}\mu(y) dX \right)^{\frac{\varrho}{2 + \varrho}}.
\end{split}
\] 
Then, using the fact that 
\[
|\A(X) - \langle\A\rangle_{B_{3/2}^+}(y)| \leq 2\Lambda \mu(y),\quad \text{for each $X = (x, y) \in D_{2}$},
\]
we infer
\[
\begin{split}
\frac{1}{\hmu(D_{3/2})}&\int_{D^{+}_{3/2}} |\nabla v - \nabla \tilde{u} |^2 \mu(y)dX\leq C  \left(\frac{1}{\hmu(D_{3/2})}\int_{D^{+}_{3/2}} |(\A(X) - \langle \A \rangle_{B_{3/2}}(y))|^{2} \mu^{-1}(y) dX  \right)^{\frac{\varrho}{2 + \varrho}}. 
\end{split}
\]
We conclude that given $\epsilon > 0$, we can choose $\delta > 0$ such that 
\[
\begin{split}
\frac{1}{\hmu(D_{3/2})}\int_{D^{+}_{3/2}} |\nabla v  -\nabla u | ^{2} \mu(y) dX &\leq \frac{1}{\hmu(D_{3/2})}\int_{D^{+}_{3/2}} |\nabla v  -\nabla \tilde{u} | ^{2} \mu(y) dX + \frac{1}{\hmu(D_{3/2})}\int_{D^{+}_{3/2}} |\nabla \tilde{u}  -\nabla u | ^{2} \mu(y) dX\\
 &\leq C (\delta^2 + \delta ^{\frac{2}{2 + \varrho}}) \leq \epsilon^{2}. 
\end{split}
\]
Estimates \eqref{v-L-infty-interior-halfCyl} follows from  the assumption that $\mathbb{A}\in \mathcal{B}(Q_{2}^{+}, \Lambda, M_{0}, \mu)$, Lemma \ref{doubling}, 
and the inequality that 
\[
\begin{split}
\|\nabla v\|_{L^{\infty} (D^{+}_{5/4})}  & \leq C \left(\frac{1}{\hmu(D_{3/2})} \int_{D^{+}_{3/2}} |\nabla v|^{2}(X) \mu(y) dX\right)^{1/2}  \leq C \left(\frac{1}{\hmu(D_{2})} \int_{D^{+}_{2}} |\nabla u|^{2}(X) \mu(y) dX\right)^{1/2}  +\epsilon \leq C (\Lambda, M_0, n),
\end{split}
\]
assuming $\epsilon \leq 1$. The proof is complete.
\end{proof}
%===============
\subsection{Approximation estimates up to the corner of the half-cylinder domain} 
In this subsection we obtain an approximation estimate for $\nabla u$ over half cylinders whose flat side and base overlaps with the flat side and base of $Q^{+}_{2}$. 
The main result of the section now can be stated in the following proposition.
\begin{proposition} \label{bottom-approximation-proposition-sittingCy} Let $\Lambda >0, M_0 >0$ be fixed. Then, for every $\epsilon >0$, there exists $\delta = \delta(\epsilon, \Lambda, M_0, n) >0$ with the following property.  Let  $\mu \in A_2(\R)$ so that $[\mu]_{A_2(\R)} \leq M_0$. Let also $\A \in \mathcal{B}(Q_{2}^{+}, \Lambda, M_0, \mu)$ satisfy \eqref{ELLIPTICITY-5}. 
 Then for every $X_0 = (x_0, 0)$, where $x_0 = (x_0', 0) \in T_1$,  and $0 < r <1$ such that $Q_{2r}^{+}(X_0) \subset Q_2^{+}$, 
if \[
\begin{split}
& \frac{1}{\hat{\mu}(Q_{3r/2}(X_0))} \int_{Q^{+}_{3r/2}(X_0)} |\A(X) - \langle{\A}\rangle_{B^{+}_{3r/2}(x_0)}(y)|^2 \mu^{-1}(y) dX  \leq \delta^{2}, \\
&\frac{1}{\hat{\mu}(Q_{2r}(X_0))}\left[ \int_{Q^{+}_{2r}(X_0)} \Big| \F(X)/\mu(y)\Big|^2 \mu(y)dX  + \int_{Q^{+}_{2r}(X_0)} \Big|f(x)/\mu(y)\Big|^2 \mu(y) dX \right] \leq \delta^{2},
\end{split}
\]
and 
 $u \in W^{1,2}(Q^{+}_{2r}(X_0), \mu)$ is a weak solution of 
\[
\left\{
\begin{array}{cccl}
\textup{div}[\A(X) \nabla u]  &= & \textup{div}({\bf F}), & \quad Q^{+}_{2r}(X_{0}), \\
u &=& 0&\quad  T_{2r}(x_{0}) \times \Gamma_{2r}, \\
\lim_{y\rightarrow 0^+}\wei{\A(x,y) \nabla u(x,y) - {\bf F}(x,y), e_{n+1}}&= & f(x), & \quad  x \in B^{+}_{2r}(x_{0}),
\end{array}
\right.
 \]
satisfying
\[
\frac{1}{\hat{\mu}(Q_{2r}(X_0))} \int_{Q^{+}_{2r}(X_0)} |\nabla u(X)|^2 \mu(y) dX \leq 1,
\]
then there exists $v \in W^{1,2}(Q^{+}_{3r/2}(X_0), \mu)$ such that 
\begin{equation} \label{u-v-approximation-bottom-sittingCy}
\frac{1}{\hat{\mu}(Q_{3r/2}(X_0))} \int_{Q^{+}_{3r/2}(X_0)} |\nabla u(X) - \nabla v(X)|^2 \mu(y) dX \leq \epsilon^{2}. 
\end{equation}
Moreover, there exists a constant $C=C(n, \Lambda, M_0)>0$ such that 
\begin{equation} \label{v-L-infty-bottom-sittingCy}
 \norm{\nabla v}_{L^\infty(Q^{+}_{5r/4}(X_0))}  \leq C(n, \Lambda, M_0).
\end{equation}
\end{proposition}

We now state the self-improving regularity estimates of Gehring's type for weak solutions of mixed boundary value problems. Its proof is very similar to that of Lemma \ref{Gerling-interior-bottom}.
\begin{lemma}  \label{Gerling-side-mixed}Suppose that \eqref{ELLIPTICITY-5} holds and $[\mu]_{A_2} \leq M_0$. Then,  there exists $\varrho = \varrho(\Lambda, M_0, n)>0$ sufficiently small such that for every $\theta \in (0,1)$,  such that for any $0 < R <2$, if $\tilde{u} \in W^{1,2}(Q^{+}_R, \mu)$ is a weak solution of
\begin{equation} \label{v-Q-r-bottom-half-R}
\left\{
\begin{array}{cccl}
\textup{div} [\A(X) \nabla \tilde{u}] & = & 0, & \quad \text{in} \quad Q^{+}_{R}, \\
\tilde{u}&=&0,&\quad \text{on} \quad T_{R}\times \Gamma_{R}\\
\wei{\A(x, 0) \nabla \tilde{u}(x, 0) , e_{n+1}} & = & 0, & \quad x \in B^{+}_R, 
\end{array}
\right.
\end{equation}
then there is some $C =C(M_0, n, \theta) >0$ such that
\[
\left(\frac{1}{\hat{\mu}(Q_{\theta R})}\int_{Q^{+}_{\theta R}} |\nabla \tilde{u}|^{2+\varrho} \mu(y) dX \right)^{\frac{1}{2+\varrho}} \leq C \left( \frac{1}{\hat{\mu}(Q_R)}\int_{Q^{+}_R} |\nabla \tilde{u}|^2 \mu(y) dX\right)^{1/2}.
\]
\end{lemma}
%============
\begin{proof}[Proof of Proposition \ref{bottom-approximation-proposition-sittingCy}]
By scaling and translation, without loss of generality, we only need to prove the proposition when $X_0 = 0$ and $r =1$.  In this case, $u$ is given to be a weak solution of 
\begin{equation} \label{u-Q-6-bottom-sittingCy}
\left\{
\begin{array}{cccl}
\text{div}[\A(X) \nabla u]  &= & \text{div}({\bf F}), &  \quad \text{in} \quad Q^{+}_2, \\
u &= & 0,& \quad \text{in} \quad T_{2} \times \Gamma_{2} \\
\lim_{y\rightarrow 0^+}\wei{\A(x,y) \nabla u(x,y) - {\bf F}(x,y), e_{n+1}}&=  &f(x), & \quad  x \in B^{+}_2.
\end{array}
\right.
\end{equation}
Note that the definition of weak solutions of \eqref{u-Q-6-bottom-sittingCy} is given in Definition \ref{weak-solution-Q-R-plus}. 
As before, we use a two step approximation procedure. The first step is to approximate $\nabla u$ it by $\nabla \tilde{u}$, where $\tilde{u}$ is a weak solution of the corresponding homogeneous equation of \eqref{u-Q-6-bottom-sittingCy}.

\noindent 
{\bf Step 1}. We claim that  there is a constant $C = C(n, M_0) >0$ such that if $u \in W^{1,2}(Q^{+}_2, \mu)$ is a weak solution of \eqref{u-Q-6-bottom-sittingCy}, then there is a weak solution  $\tilde{u} \in W^{1,2}(Q_{7/4},\mu)$ of 
\begin{equation} \label{v-1-Q-5-bottom-sittingCy}
\left\{
\begin{array}{cccl}
\textup{div} [\A(X) \nabla \tilde{u}] & =  & 0, & \quad \text{in} \quad Q^{+}_{7/4}, \\
 \tilde{u} & = &u, &  \quad \text{on} \quad \partial Q^{+}_{7/4} \setminus (B^{+}_{7/4}\times \{0\}), \\
 \lim_{y\rightarrow 0^+ }\wei{\A(x, y) \nabla \tilde{u}(x, y), e_{n+1}} & = & 0, &  \quad x \in B^{+}_{7/4}.
\end{array}
\right.
\end{equation}
satisfying 
\begin{equation} \label{u-tilde-u-sittingCy}
\int_{Q^{+}_{7/4}} |\nabla u -\nabla \tilde{u}|^2 \mu(y) dX \leq C(\Lambda, n, M_0) \left[ \int_{Q^{+}_{7/4}} \Big|\F/\mu \Big|^2 \mu(y) dX +  \int_{Q^{+}_{7/4}} \Big|f(x)/\mu(y) \Big|^2 \mu(y)dX \right].
\end{equation}
Moreover, there is $\varrho = \varrho(n, M_0) >0$ such that
\begin{equation} \label{higher-reg-bottom-tilde-u-sittingCy}
\begin{split}
& \left( \frac{1}{\hat{\mu}(Q_{3/2})} \int_{Q^{+}_{3/2}} |\nabla \tilde{u}|^{2+\varrho} \right)^{\frac{1}{2+\varrho}} \\
& \leq C(n, M_0, \Lambda) \left[ \frac{1}{\hat{\mu}(Q_{7/4})} \left\{ \int_{Q^{+}_{7/4}} |\nabla u|^2 \mu(y) dX + \int_{Q^{+}_{7/4}} \Big|\F/\mu \Big|^2 \mu(y) dX +  \int_{Q^{+}_{7/4}} \Big| \frac{f(x)}{\mu(y)} \Big|^2 \mu(y)dX \right\}  \right]^{1/2}.
\end{split}
\end{equation}
%\end{lemma}
To prove the claim we first note that by a weak solution $\tilde{u}\in W^{1,2}(Q_{7/4}, \mu)$ of \eqref{v-1-Q-5-bottom-sittingCy} we mean
 if $\tilde{u} -u \in \overset{*}{W}^{1,2}(Q^{+}_{7/2}, \mu)$ and
\[
\int_{Q^{+}_{7/2}} \wei{\A\nabla \tilde{u}, \nabla \psi} dX = 0, \quad \forall \psi \in \overset{*}{W}^{1,2}(Q^{+}_{7/4}, \mu).
\]
From the above and the trace inequality, Lemmas \ref{trace-zero}, we observe that for  each weak solution $u \in W^{1,2}(Q_2^+, \mu)$ of \eqref{u-Q-6-bottom-sittingCy}, the existence and uniqueness weak solution of \eqref{v-1-Q-5-bottom-sittingCy} can be proved as in \cite[Theorem 2.2]{Fabes}. We now write $g = u -\tilde{u}$. By the definition, $g \in \overset{*}{W}^{1,2}(Q^{+}_{7/4}, \mu)$ and is a weak solution of
\begin{equation} \label{g-Q-5-bottom-sittingCy}
\left\{
\begin{array}{cccl}
\text{div} [\A(X) \nabla g] & = & \text{div}[{\bf F}], & \quad \text{in} \quad Q^{+}_{7/4}, \\
 g & =& 0, & \quad \text{on} \quad \partial Q^{+}_{7/4} \setminus (B^{+}_{7/4}\times \{0\}), \\
 \lim_{y\rightarrow 0^+}\wei{\A(x, y) \nabla g (x, y) -{\bf F}(x, y), e_{n+1}} & = & f(x), & \quad x \in B^{+}_{7/4}.
\end{array}
\right.
\end{equation}
Then, use $g$ as a test function of its equation, we obtain that 
\[
\int_{Q^{+}_{7/4}} \wei{\A\nabla g, \nabla g} dX = \int_{Q^{+}_{7/4}} {\bf F}\cdot\nabla g dX + \int_{B^{+}_{7/4}} g(x,0) f(x) dx.
\]
From the  ellipticity condition \eqref{ELLIPTICITY-5} it follows that 
\[
 \int_{Q^{+}_{7/4}} |\nabla g|^2 \mu(y) dX \leq C(\Lambda) \left[
\int_{Q^{+}_{7/4}} \Big | \frac{{\bf F}}{\mu} \Big| |\nabla g| \mu(y) dX + \int_{B^{+}_{7/4}} |g(x,0)| |f(x)| dx \right]. 
\]
Then, applying the H\"{o}lder's inequality and Young's inequality, we see that with an5 $\beta>0$
\[
\begin{split}
\Lambda \int_{Q^{+}_{7/4}} |\nabla g|^2 \mu(y) dX  & \leq \beta 
\left[ \int_{Q^{+}_{7/4}} |\nabla g|^2\mu(y) dX + \frac{\mu(\Gamma_{7/4})}{2 (7/4)^{2}}\int_{B^{+}_{7/4}} |g(x,0)|^2 dx \right]  \\
&+ \frac{C(\Lambda)}{4\beta} \left[ \int_{Q^{+}_{7/4}} \Big| \frac{{\bf F}}{\mu} \Big|^2 \mu(y) dX + \frac{2 (7/4)^{2}}{\mu(\Gamma_{7/4})}\int_{B^{+}_{7/4}} |f(x)|^2 dx \right]
\end{split}
\]
Then, it follows from the trace inequality, Lemma \ref{trace-zero}, that
\[
\int_{Q^{+}_{7/4}} |\nabla g|^2 \mu(y) dX  \leq \beta C(n, M_0)
\int_{Q^{+}_{7/4}} |\nabla g|^2\mu(y) dX + \frac{C(\Lambda, M_0)}{4\beta} \left[ \int_{Q^{+}_{7/4}} \Big| \frac{{\bf F}}{\mu} \Big|^2 \mu(y) dX + \frac{ (7/4)^{2}}{\mu(\Gamma_{7/4})}\int_{B^{+}_{7/4}} |f(x)|^2 dx \right].
\]
Then, with $\beta$ sufficiently such that  $\beta  C(n, M_0)] \leq 1/2$, 
we obtain
\begin{equation} \label{g-energy-last-step-sittingCy}
\begin{split}
\int_{Q^{+}_{7/4}} |\nabla g|^2 \mu(y) dX &\leq C(\Lambda, n, M_0) \left[ \int_{Q^{+}_{7/4}} \Big| \frac{{\bf F}}{\mu} \Big|^2 \mu(y) dX +  
(7/4)^{2}\left(\int_{0}^{7/4} \mu (y)dy\right)^{-1} \int_{B^{+}_{7/4}} |f(x)|^2 dx \right]\\
&\leq C(\Lambda, n, M_0) \left[ \int_{Q^{+}_{7/4}} \Big| \frac{{\bf F}}{\mu} \Big|^2 \mu(y) dX +  
\int_{Q^{+}_{7/4}} \Big |\frac{f(x)}{\mu(y)} \Big|^2 \mu(y) dX \right].
\end{split}
\end{equation}
This together with \eqref{g-energy-last-step-sittingCy} proves estimate \eqref{u-tilde-u-sittingCy}, while estimate \eqref{higher-reg-bottom-tilde-u-sittingCy} follows directly from \eqref{u-tilde-u-sittingCy} and Lemma \ref{Gerling-side-mixed}. \\
\noindent
{\bf Step 2.}  Let $\tilde{u}$ and $\varrho$ be  as in {\bf Step 1.} Then, we claim that  there exists $v \in W^{1,2}(Q_{3/2}^+, \mu)$  a weak solution 
\begin{equation} \label{v-Q-4-bottom-sittingCy}
\left\{
\begin{array}{cccl}
\textup{div}[\langle {\A}\rangle_{B_{3/2}^+}(y) \nabla v] & = &0, &  \quad \text{in} \quad Q^{+}_{3/2}, \\
v & = &\tilde{u}, &  \quad \text{on} \quad \partial Q^{+}_{3/2} \setminus (B^{+}_{3/2} \times \{0\} ), \\
\lim_{y\rightarrow 0^+} \wei{\langle{\A}\rangle_{B_{3/2}^+}( y) \nabla v(x, y), e_{n+1}} & = & 0, & \quad x \in B^{+}_{3/2},
\end{array}
\right.
\end{equation}
satisfying 
\[
\begin{split}
 &\left (\frac{1}{\hat{\mu}(Q^{+}_{3/2})} \int_{Q^{+}_{3/2}} |\nabla \tilde{u} -\nabla v|^2 \mu(y) dX\right)^{1/2}  \\
& \leq C(\Lambda) \left(\frac{1}{\hat{\mu}(Q_{3/2})} \int_{Q^{+}_{3/2}}  |\nabla \tilde{u}|^{2+\varrho} \mu(y)\right)^{\frac{1}{2+\varrho}} \left(\frac{1}{\hat{\mu}(Q_{3/2})} \int_{Q^{+}_{3/2}}\Big |\A - \langle{\A}\rangle_{B_{3/2}^+}(y)\Big|^2\mu(y)^{-1} dX  \right)^{\frac{\varrho}{2(2+\varrho)}}.
 \end{split}
 \]
To prove this claim, first the existence of $v$ can be deduced in a similar way as in {\bf Step 1}.  
Let $w = v - \tilde{u}$. We note that $w \in \overset{*}{W}^{1,2}(Q^{+}_{3/2}, \mu)$. Therefore, by using ${w}$ as a test function for the equation of $v$, and $\tilde{u}$, we obtain the following
\[
\int_{Q^{+}_{3/2}} \wei{\langle{\A}\rangle_{B_{3/2}^+}(y) \nabla w, \nabla w} dX = -\int_{Q^{+}_{3/2}} \wei{[\A -\langle{\A}\rangle_{B_{3/2}^+}(y)]\nabla \tilde{u}, \nabla {w}} dX.
\]
Together with the ellipticity condition \eqref{ELLIPTICITY-5}, and H\"{o}lder's inequality, this implies that 
\[
\begin{split}
 \int_{Q^{+}_{3/2}} |\nabla w|^2 \mu(y) dX & \leq C(\Lambda) \int_{Q^{+}_{3/2}} \Big |\A - \langle{\A}\rangle_{B_{3/2}^+}(y)\Big| |\nabla \tilde{u}| |\nabla w| dX \\
 &\leq C(\Lambda) \left(\int_{Q^{+}_{3/2}} |\nabla w|^2 \mu(y) dX \right)^{1/2} \left( \int_{Q^{+}_{3/2}} \Big |\A - \langle{\A}\rangle_{B_{3/2}^+}(y)\Big|^2 |\nabla \tilde{u}|^2 \mu(y)^{-1} dX\right)^{1/2}.
\end{split}
\]
Hence,
\[
\frac{1}{\hat{\mu}(Q_{3/2})}\int_{Q^{+}_{3/2}} |\nabla w|^2 \mu(y) dX  \leq \frac{C(\Lambda)}{\hat{\mu}(Q_{3/2})}\int_{Q^{+}_{3/2}} \Big (|\A - \langle{\A}\rangle_{B_{3/2}^+}(y)| \mu(y)^{-1}\Big)^2 |\nabla \tilde{u}|^2 \mu(y) dX.
\]
Then, applying the  H\"{o}lder's inequality with the exponents $\frac{2+\varrho}{2}$ and $\frac{2+\varrho}{\varrho}$, we obtain
\[
\begin{split}
 &\left (\frac{1}{\hat{\mu}(Q_{3/2})} \int_{Q^{+}_{3/2}} |\nabla w|^2 \mu(y) dX\right)^{1/2}  \\
 & \leq C(\Lambda, M_{0}) \left(\frac{1}{\hat{\mu}(Q_{3/2})} \int_{Q^{+}_{3/2}}  |\nabla \tilde{u}|^{2+\varrho} \mu(y)dX\right)^{\frac{1}{2+\varrho}} \left(\frac{1}{\hat{\mu}(Q_{3/2})} \int_{Q^{+}_{3/2}}\Big |(\A - \langle{\A}\rangle_{B_{3/2}^+}(y)) \mu(y)^{-1}\Big|^{\frac{2(2+\varrho)}{\varrho}} \mu(y) dX  \right)^{\frac{\varrho}{2(2+\varrho)}}.
 \end{split}
\]
Again from the ellipticity condition \eqref{ELLIPTICITY-5} that
\[
|\A(X) - \langle{\A}\rangle_{B_{3/2}^+}(y)| \mu(y)^{-1} \leq \Lambda^{-1}. 
\]
As a consequence,  the right hand side of the previous inequality can be simplified to  
\[
\begin{split}
 &\left (\frac{1}{\hat{\mu}(Q_{3/2})} \int_{Q^{+}_{3/2}} |\nabla w|^2 \mu(y) dX\right)^{1/2}  \\
& \leq C(\Lambda, M_{0}) \left(\frac{1}{\hat{\mu}(Q_{3/2})} \int_{Q^{+}_{3/2}}  |\nabla \tilde{u}|^{2+\varrho} \mu(y)\right)^{\frac{1}{2+\varrho}} \left(\frac{1}{\hat{\mu}(Q_{3/2})} \int_{Q^{+}_{3/2}}\Big |\A - \langle{\A}\rangle_{B_{3/2}^+}(y)\Big|^2\mu(y)^{-1} dX  \right)^{\frac{\varrho}{2(2+\varrho)}}, 
 \end{split}
\]
proving the claim.

\noindent 
{\bf Step 3}. This is the final step where we put together the above two steps to prove Proposition \ref{bottom-approximation-proposition-sittingCy}.
To that end, it follows from estimate \eqref{higher-reg-bottom-tilde-u-sittingCy} in {\bf Step 1} and assumptions that
\[
\left(\frac{1}{\hat{\mu}(Q_{3/2})} \int_{Q^{+}_{3/2}}  |\nabla \tilde{u}|^{2+\varrho} \mu(y)\right)^{\frac{1}{2+\varrho}} \leq C(n, \Lambda, M_0).
\]
Observe also that
\[
\begin{split}
 \frac{1}{\hat{\mu}(Q_{3/2})} \int_{Q^{+}_{3/2}} |\nabla u - \nabla v|^2 \mu(y) dX & \leq \frac{1}{\hat{\mu}(Q_{3/2})} \int_{Q^{+}_{3/2}} |\nabla u - \nabla \tilde{u}|^2 \mu(y) dX + \frac{1}{\hat{\mu}(Q_{3/2})} \int_{Q^{+}_{3/2}} |\nabla \tilde{u} - \nabla v|^2 \mu(y) dX \\
 & \leq \frac{\hat{\mu}(Q_{7/4})}{ \hat{\mu}(Q_{3/2})} \frac{1}{\hat{\mu}(Q_{7/4})}\int_{Q^{+}_{7/4}} |\nabla u - \nabla \tilde{u}|^2 \mu(y) dX + \frac{1}{\hat{\mu}(Q_{3/2})} \int_{Q^{+}_{3/2}} |\nabla \tilde{u} - \nabla v|^2 \mu(y) dX\\
& \leq \frac{C(n, M_0)}{\hat{\mu}(Q_{7/4})}\int_{Q^{+}_{7/4}} |\nabla u - \nabla \tilde{u}|^2 \mu(y) dX + \frac{1}{\hat{\mu}(Q_{3/2})} \int_{Q^{+}_{3/2}} |\nabla \tilde{u} - \nabla v|^2 \mu(y) dX.
\end{split}
\]
Hence, it follows from \eqref{u-tilde-u-sittingCy} of {\bf Step 1} and {\bf Step 2} that
\[
\begin{split}
& \frac{1}{\hat{\mu}(Q_{3/2})} \int_{Q^{+}_{3/2}} |\nabla u - \nabla v|^2 \mu(y) dX \\
&\leq  C(n, \Lambda, M_0)\left[\frac{1}{\hat{\mu}(Q_{7/4})} \int_{Q^{+}_{7/4}} \Big| \frac{{\bf F}}{\mu} \Big|^2 \mu(y) dX +  \frac{1}{\hat{\mu}(Q_{7/4})}\int_{Q^{+}_{7/4}} \Big| \frac{f(x)}{\mu(y)} \Big|^2 dX\right.\\
&\quad\quad\quad\quad\quad\left. + \left(\frac{1}{\hat{\mu}(Q_{3/2})} \int_{Q^{+}_{3/2}}\Big |\A - \langle{\A}\rangle_{B_{3/2}^+}(y)\Big|^2\mu(y)^{-1} dX  \right)^{\frac{\rho}{(2+\rho)}}\right]  \leq C(n, \Lambda, M_0) \delta^{2}.
\end{split}
\]
Hence, if we chose $\delta$ sufficiently small so that $\delta C(n, \Lambda, M_0) < \epsilon$, we obtain
\[
 \frac{1}{\hat{\mu}(Q_{3/2})} \int_{Q^{+}_{3/2}} |\nabla u - \nabla v|^2 \mu(y) dX \leq \epsilon.
\]
This proves \eqref{u-v-approximation-bottom-sittingCy}. It remains to prove \eqref{v-L-infty-bottom-sittingCy}. From the last estimate, and the assumptions in the Proposition \ref{bottom-approximation-proposition-sittingCy}, it follows that 
\[
\begin{split}
 \frac{1}{\hat{\mu}(Q_{3/2})} \int_{Q^{+}_{3/2}} |\nabla v|^2 \mu(y) dX  & \leq \epsilon +  \frac{1}{\hat{\mu}(Q_{3/2})} \int_{Q^{+}_{3/2}} |\nabla u|^2 \mu(y) dX \\
 & \leq \epsilon + \frac{C(n, M_0) }{\hat{\mu}(Q_{7/4})} \int_{Q^{+}_{7/4}} |\nabla u|^2 \mu(y) dX  \leq \epsilon + C(n, M_0).
 \end{split}
\]
From this, the equation \eqref{v-Q-4-bottom-sittingCy}, and the fact that $v = 0$ on $T_{3/2} \times (0, 3/2)$, and $\mathbb{A} \in \mathcal{B}(Q_{2}^{+}, \Lambda, M_{0}, \mu)$, it follows from  the Definition \ref{class-B} that
\[
\norm{\nabla v}_{L^\infty(Q^{+}_{5/4})} \leq C(n, \Lambda, M_0) \left[ \frac{1}{\hat{\mu}(Q_{3/2})} \int_{Q^{+}_{3/2}} |\nabla v|^2 \mu(y) dX \right]^{1/2} \leq C(n, \Lambda, M_0).
\]
The proof of Proposition \ref{bottom-approximation-proposition-sittingCy} is the complete.
\end{proof}
%=========================
\section{Local weighted $W^{1,p}$-regularity estimates up to the base of the cylinder domain } \label{density-est-sec}
In this section, we focus on the following equation to prove Theorem \ref{local-grad-estimate-interior}: 
\begin{equation}\label{interior-MAIN}
\left\{
\begin{array}{cccl}
\textup{div}[\mathbb{A}(X) \nabla u] & = & \textup{div}({\bf F}), & \quad \text{in} \quad  Q_{2},\\
\lim_{y\rightarrow 0^+ }\langle \mathbb{A}(x, y) \nabla u (x, y) - {\bf F}(x, y), e_{n + 1} \rangle & = & f(x), &\quad  x \ \in \ B_2,
\end{array}
\right.
\end{equation}
where $\A : Q_2 \rightarrow \R^{(n+1) \times (n+1)}$ is a symmetric and measurable matrix.  Several lemmas are needed for proving Theorem \ref{local-grad-estimate-interior}. Let us obtain some density estimates for the interior first. 
\begin{lemma} \label{cor-varpi}
Suppose that $M_{0}>0$ and $\mu \in A_{2}(\R)$ such that $[\mu]_{A_{2}} \leq M_{0}$. 
There exists a constant $\varpi_1(n, \Lambda, M_0)> 1$ such that for  any $\epsilon > 0 $, there exists a small constant $\delta_1(\epsilon, \Lambda, M_0, n)>0$ with the property that for  every $\delta \in (0, \delta_1], \varpi \geq \varpi_1$, and every $\A \in \mathcal{A}(Q_2, \Lambda, M_0, \mu)$ such that \eqref{Q-2-ellipticity-condition}-\eqref{PBMO} hold,  if $u\in W^{1, 2}(Q_2, \mu)$ is any  weak solution to 
\[
\textup{div}[\mathbb{A} \nabla u] = \textup{div}({\bf F})\quad \text{in} \quad Q_{2},
\]
and if $Z= (z, z_{n+1})\in Q_1$, $r>0$ so that $D_{2r}(Z) \subset Q_2$, 
\begin{equation}\label{max-f-small-forany-r-int-D}
D_{r}(Z)\cap \{X\in Q_1: \M_{\mu, Q_{2}}(|\nabla u|^{2}) \leq 1 \}\cap \{X\in Q_1: \M_{\mu, Q_2}\left(\left|\F/\mu\right|^{2}\right) \leq \delta^{2} \} \neq \emptyset, 
\end{equation}
then 
\[
\hmu\Big (\{X\in Q_1: \M_{\mu, Q_{2}}(|\nabla u|^{2})  > \varpi^{2}\}\cap D_{r}(Z) \Big ) < \epsilon \hmu\Big(D_{r}(Z)\Big). 
\]
\end{lemma}
\begin{proof} For given $\epsilon > 0$, let $\eta > 0$ to be determined, depending only on $\epsilon, \Lambda, M_0, n$.  Let $\delta_1 = \delta(\eta, n, M_0, \Lambda) > 0$ be defined as in Proposition \ref{interior-approximation-proposition}. By the assumption \eqref{max-f-small-forany-r-int-D}, we can find $X_0 =(x_0, y_0) \in D_{r}(Z)\cap Q_1$ such that 
%\end{proof}
\[
\M_{\mu, Q_{2}}(|\nabla u|^{2})(X_0) \leq 1\quad \text{and}\quad \mathcal{M}_{\mu, Q_2}\left(\left|\F/\mu \right|^{2}\chi_{Q_{2}}\right)(X_0) \leq \delta^{2}. 
\]
 From these inequalities, it follows that for any $\tau > 0$, 
\begin{equation} \label{X-zero-D-tau}
\frac{1}{\hmu(D_{\tau}(X_0))}\int_{D_{\tau}(X_0) \cap Q_2}  |\nabla u|^{2} \mu (y) dX \leq 1 ,\quad \text{and\, }\frac{1}{\hmu(D_{\tau}(X_0))}\int_{D_{\tau}(X_0) \cap Q_2}  \left|\F(X)/\mu(y)\right|^{2} \mu (y) dX \leq \delta^{2}.
\end{equation}
Note that $D_{2r}(Z)\subset D_{3r}(X_{0})$ and it follows from Lemma \ref{doubling} that
\[
\frac{\hmu(D_{3r}(X_0))}{\hmu(D_{2r}(Z))} = \frac{|B_{3r}(x_0)| \mu(-3r, 3r) }{|B_{2r}(z)|\mu(-2r,  2r ) } \leq M_{0} \left(\frac{3}{2}\right)^{n + 2}. 
\]
Moreover, since $D_{2r}(Z)\subset D_{3r}(X_{0}) \cap Q_2$, it then follows from \eqref{X-zero-D-tau} that 
\[
\frac{1}{\hmu(D_{2r}(Z))}\int_{D_{2r}(Z)} |\nabla u|^{2} \mu (y) dX \leq  \frac{\hmu(D_{3r}(X_0))}{\hmu(D_{2r}(Z))}  \frac{1}{\hmu(D_{3r}(X_0))}\int_{D_{3r}(X_0) \cap Q_2}  |\nabla u|^{2} \mu (y) dX \leq M_{0} \left(\frac{3}{2}\right)^{n + 2}. 
\]
Similarly, we have the inequality 
\[
\begin{split}
\frac{1}{\hmu(D_{2r}(Z))}\int_{D_{2r}(Z)} \left|\frac{{\bf F}(X)}{\mu(y)}\right|^{2} \mu (y) dX & \leq  \frac{\hmu(D_{3r}(X_0))}{\hmu(D_{2r}(Z))}  \frac{1}{\hmu(D_{3r}(X_0))}\int_{D_{3r}(X_0)  \cap Q_2} \left|\frac{{\bf F}(X)}{\mu(y)}\right|^{2} \mu (y) dX \\
&\leq \delta^{2}M_{0} \left(\frac{3}{2}\right)^{n + 2}. 
\end{split}
\]
From \eqref{PBMO}, and the easy assertion $D_{3r/2}(Z) \subset Q_2$, it follows
\[
\frac{1}{\hmu(D_{3r/2}(Z) )}\int_{D_{3r/2}(Z) }|\A(X) - \langle \A\rangle_{B_{3r/2}(z)}(y)|^{2}\mu^{-1}(y) dX  \leq \delta^{2}. % \quad \wei{\A}_{B_{2r}(z)} \in \mathbb{M}_{Z, r}(\Lambda, M_0, \Lambda).
\]
From these estimates, our assumption that $\A \in \mathcal{A}(Q_2, \Lambda, M_0, \mu)$, and Proposition \ref{interior-approximation-proposition}, there exists a $v \in W^{1,2}(D_{3r/2}(Z), \mu)$ such that  
\[ \|\nabla v\|_{L^{\infty}(D_{5r/4}(Z)} \leq C_{0} = C(n,M_{0}, \Lambda),
\]
and 
\begin{equation} \label{conclusion-lemma-int-1}
\frac{1}{\hmu(D_{3r/2}(Z) )}\int_{D_{3r/2}(Z) } |\nabla u - \nabla v|^{2} \mu (y) dX  < \eta^{2} M_{0}\left(\frac{3}{2}\right)^{n+2}.
\end{equation}
\noindent
Now, let $\varpi_1 =[ \max\{M_{0} 9^{n + 2}, 4C^{2}_{0} \}]^{1/2}$. Then with $\varpi \geq \varpi_1$, we  claim that  
\[
\Big \{
X\in Q_{1}: \M_{\mu, Q_{2}}(|\nabla u|^{2}) > \varpi^{2}
\Big  \}\cap D_{r} (Z) \subset\Big  \{X\in Q_1:  \M_{\mu, D_{3r/2}(Z)}(|\nabla u -\nabla v|^{2})>C_{0}^{2} \Big \}\cap D_{r}(Z). \]
To prove the claim, let $H\in D_{r}(Z)\cap Q_1$, such that
\[ 
\mathcal{M}_{\mu,  D_{3r/2}(Z)} (|\nabla u -\nabla v|^{2})(H)\leq C_{0}^{2}.
\]
For $\rho < r/4$, we see that $D_{\rho}(H) \subset D_{5r/4}(Z) \subset D_{3r/2}(Z) \subset Q_2$, and therefore
\[
\begin{split}
& \frac{1}{\hmu({D_{\rho}(H)})}\int_{D_{\rho}(H) \cap Q_2} |\nabla u |^{2} \mu(y) dX \\
&\leq \frac{2}{\hmu({D_{\rho}(H)})} \left\{\int_{D_{\rho}(H)} |\nabla u - \nabla v|^{2} \mu(y) dX + \int_{D_{\rho}(H)} |\nabla v|^{2} \mu(y) dX\right\}\\
&\leq 2 C_{0}^{2} +  2 \|\nabla v\|_{L^{\infty}(D_{5r/4}(Z))}^{2}  \leq  4 C_{0}^{2} \leq \varpi^2.
\end{split}
\]
If $\rho \geq r/4$, we use the fact that $D_{\rho}(H) \subset D_{9\rho}(X_{0})$ to estimate 
\[
\frac{1}{\hmu({D_{\rho}(H)})}\int_{D_{\rho}(H)\cap Q_2}  |\nabla u |^{2} \mu(y) dX   \leq \frac{\hmu({D_{9\rho}(X_0)})}{\hmu({D_{\rho}(H)})} \frac{1}{\hmu({Q_{9\rho}(X_0)})}\int_{D_{9\rho}(X_0) \cap Q_2} |\nabla u|^{2} \mu(y) dX \leq M_{0} 9^{n + 2}.
\]
Combining the above estimates we obtain 
\[
\frac{1}{\hmu({D_{\rho}(H)})}\int_{D_{\rho}(H)\cap Q_2} |\nabla u|^{2} \mu(y) dX \leq \varpi^2, \quad \forall \ \rho >0.
\]
Therefore, $\M_{\mu, Q_2} (|\nabla u_{\kappa}|^{2}) (H) \leq \varpi^{2}$ as desired. From the claim, it follows that
\[
\begin{split}
&\hmu \Big( \{ X\in Q_{1}: \M_{\mu, Q_{2}}(|\nabla u|^{2}) > \varpi^{2} \} \cap D_{r} (Z) \Big) \\
& \leq \hmu \Big (\big \{X\in Q_{1}: \M_{\mu, D_{3r/2}(Z)}(|\nabla u -\nabla v|^{2})>C_{0}^{2}\big \} \cap D_{r} (Z) \Big)\\
& \leq \frac{C(n, M_{0})}{C_{0}^{2}}\ \hmu(D_{3r/2}(Z)) \frac{1}{\hmu(D_{3r/2}(Z))}\int_{D_{3r/2}(Z)} |\nabla u - \nabla v|^{2} \mu(y)dX \\
&\leq C(M_0, n)\eta^{2}\ \hmu(D_{r}(Z)),
\end{split}
\]
where $C(n, M_{0})$ comes from the weak $1-1$ estimate, Lemma \ref{Hardy-Max-p-p}, and we have used \eqref{conclusion-lemma-int-1}. 
From the last estimate, we observe that if we choose $\eta > 0$ sufficiently small such that $ C\eta^{2} < \epsilon$, Lemma \ref{cor-varpi} follows. 
\end{proof}
%================
\begin{lemma} \label{cor-bdry-varpi}
Suppose that $M_{0}>0$ and $\mu \in A_{2}$ such that $[\mu]_{A_{2}} \leq M_{0}$. 
There exists a constant $\varpi_2(n, \Lambda, M_0)> 1$ such that for  any $\epsilon > 0 $, there exists a small constant $\delta_2(\epsilon, \Lambda, M_0, n)>0$  with the property that for every $\delta \in (0, \delta_2]$,  $\varpi \geq \varpi_2$, and for $\mathbb{A} \in \mathcal{A}(Q_2, \Lambda, M_0, \mu)$ such that \eqref{Q-2-ellipticity-condition}-\eqref{PBMO} hold, if $u\in W^{1, 2}(Q_2, \mu)$ is any  weak solution to 
\[
\left\{\begin{array}{cccl}
\textup{div}[\mathbb{A} \nabla u] & = &\textup{div}({\bf F}), & \quad \text{in} \quad Q_{2},\\
\lim_{y\to 0^{+}}\langle \mathbb{A}(x, y) \nabla u (x, y) - \F, e_{n + 1} \rangle & = & f(x),& \quad x\in B_2
\end{array}
\right.
\]
and if $Z_0= (z, 0)\in B_{1}\times\{0\}$ and some $r>0$ such that $Q_{2r}(Z_0)\subset Q_2$, and 
\begin{equation}\label{max-f-small-forany-r-bdry-1}
D_{r}(Z_0)\cap \{X\in Q_1: \mathcal{M}_{\mu, Q_{2}} (|\nabla u|^{2}) \leq 1 \}\cap \{X\in Q_1: \mathcal{M}_{\mu, Q_2}\left(\left|\F/\mu\right|^{2} \right) +\mathcal{M}_{\mu,Q_2}\left(\left|f/\mu\right|^{2}\right)\leq \delta^{2} \} \neq \emptyset, 
\end{equation}
then 
\[
\hmu(\{X\in Q_1: \mathcal{M}_{\mu, Q_{2}}(|\nabla u|^{2})  > \varpi^{2}\}\cap D_{r}(Z_0) ) < \epsilon \hmu(D_{r}(Z_0)). 
\]
\end{lemma}
\begin{proof} For given $\epsilon > 0$, let $\eta > 0$ to be determined depending only on $\epsilon, \Lambda, M_0, n$.  Choose $\delta_2 = \delta(\eta, n, M_0, \Lambda)$ defined in Proposition \ref{bottom-approximation-proposition}. Now, for $\delta \in (0, \delta_2]$, by using this $\delta$ in the assumption \eqref{max-f-small-forany-r-bdry-1}, we can find $X_0 =(x_0, y_0) \in D_{r}(Z_0)\cap Q_1 = Q_{r}(Z_0)\cap Q_1$ such that 
\[
\mathcal{M}_{\mu, Q_{2}}(|\nabla u|^{2})(X_0) \leq 1\quad \text{and}\quad \mathcal{M}_{\mu, Q_2}\left(\left|\F/\mu\right|^{2} \right)(X_0) +\mathcal{M}_{\mu, Q_2}\left(\left|f/\mu\right| \right)\leq \delta^{2}. 
\]
From these inequalities it follows that  
\begin{equation} \label{max-Z-zero-bottom}
\begin{split} 
& \frac{1}{\hmu(D_{\tau}(X_0))}\int_{D_{\tau}(X_0)\cap Q_2} |\nabla u|^{2} \mu (y) dX \leq 1, \quad \text{and}
\\
& \frac{1}{\hmu(D_{\tau}(X_0))}\int_{D_{\tau}(X_0) \cap Q_2} \left(|\F(X)/\mu(y)|^{2} + |f(x)/\mu(y)|^{2}\right) \mu (y) dX \leq \delta^{2}, \quad \forall \ \tau >0.
\end{split}
\end{equation}
We notice that  $Q_{2r}(Z_0) = D_{2r}(Z_0) \cap Q_2  \subset D_{3r}(X_0) $, and therefore it follows from Lemma \ref{doubling} that
\[
\frac{\hmu(D_{3r}(X_0))}{\hmu(Q_{2r}(Z_0))} = 
 \frac{|B_{3r}(x_0)| \mu(-3r, 3r) }{|B_{2r}(z)|\mu(0,  2r )}  \leq 4M_{0} \left(\frac{3}{2}\right)^{n + 2}. 
\]
Moreover, since $Q_{2r}(Z_0) = D_{2r}(Z_0) \cap Q_2 \subset D_{3r}(X_0) \cap Q_2,  
$, it then follows that 
\[
\frac{1}{\hmu(Q_{2r}(Z_0))}\int_{Q_{2r}(Z_0) } |\nabla u|^{2} \mu (y) dX \leq  \frac{\hmu(D_{3r}(X_0))}{\hmu(Q_{2r}(Z_0))} \frac{1}{\hmu(D_{3r}(X_0))}\int_{D_{3r}(X_0) \cap Q_2}  |\nabla u|^{2} \mu (y) dX \leq 4M_{0} \left(\frac{3}{2}\right)^{n + 2}. 
\]
Similarly,
\[
\begin{split}
& \frac{1}{\hmu(D_{2r}(Z_0))}\int_{D_{2r}(Z_0)\cap Q_2} \left(\left|\frac{{\bf F}(X)}{\mu(y)}\right|^{2} +  \left(\frac{|f|}{\mu(y)}\right)^{2} \right) \mu (y) dX \\ 
&\leq  \frac{\hmu(D_{3r}(X_0))}{\hmu(D_{2r}(Z_0))} \frac{1}{\hmu(D_{3r}(X_0))}\int_{D_{3r}(X_0)\cap Q_2} \left(\left|\frac{{\bf F}(X)}{\mu(y)}\right|^{2} +  \left(\frac{|f|}{\mu(y)}\right)^{2}  \right)\mu (y) dX\\
& \leq 4 \delta^{2}M_{0} \left(\frac{3}{2}\right)^{n + 2}. 
\end{split}
\]
Moreover, from \eqref{PBMO} and Definition \ref{class-A}, we notice that 
\[
\frac{1}{\hmu(Q_{3r/2}(Z_0) )}\int_{Q_{3r/2}(Z_0)}|\A(X) - \langle \A\rangle_{B_{3r/2}(z)}(y)|^{2}\mu^{-1}(y) dX  \leq 4M_0\delta^{2}. % \quad \wei{\A}_{B_{3r/2}} \in \G_r(\Lambda, M_0, \mu).
\]
Therefore,  all of the assumptions of Proposition \ref{bottom-approximation-proposition} are satisfied with $u$ replace by $u/[4M_0 (3/2)^{n+2}]^{1/2}$ and ${\F}$ replaced by $\F/[4M_0 (3/2)^{n+2}]^{1/2}$.  As a consequence there exists a $v \in W^{1,2}(Q_{3r/2}(Z_0), \mu)$ such that  
\begin{equation} \label{conclusion-lemma-int}
\left\{
\begin{split}
& \frac{1}{\hmu(Q_{3r/2}(Z_0) )}\int_{Q_{3r/2}(Z_0) } |\nabla u- \nabla v|^{2} \mu (y) dX  < \eta^{2} [4M_0 (3/2)^{n+2}], \quad \text{and} \\
&\|\nabla v\|_{L^{\infty}(Q_{5r/4}(Z_0))} \leq C_{0} := C(n,M_{0}, \Lambda).
\end{split} \right.
\end{equation}
Now, chose $\varpi_2 = [\max\{M_{0} 9^{n + 2}, 4C^{2}_{0} \}]^{1/2}$. Then, with $\varpi \geq \varpi_2$, we claim that  
\[
\{
X\in Q_{1}: \mathcal{M}_{\mu, Q_{2}}(|\nabla u|^{2}) > \varpi^{2}
\}\cap D_{r} (Z_0) \subset \{X\in Q_1:  \mathcal{M}_{\mu, Q_{3r/2}(Z_0)}(|\nabla u -\nabla v|^{2})>C_{0}^{2}\}\cap D_{r}(Z_0). \]
To prove the claim, let $H\in D_{r}(Z_0)\cap Q_1$, such that
\[ 
\mathcal{M}_{\mu, Q_{3r/2}(Z_0)}(|\nabla u -\nabla v|^{2})(H)\leq C_{0}^{2}.
\]
For $\rho < r/4$, we see that $D_{\rho}(H) \subset D_{5r/4}(Z_0) \subset D_{3r/2}(Z_0)$, and therefore 
\[
\begin{split}
& \frac{1}{\hmu({D_{\rho}(H)})}\int_{D_{\rho}(H)} \chi_{Q_2}(X)|\nabla u |^{2} \mu(y) dX \\ &\leq \frac{2}{\hmu({D_{\rho}(H)})} \left\{\int_{D_{\rho}(H)} \chi_{Q_{3r/2}(Z_0)}|\nabla u - \nabla v|^{2} \mu(y) dX + \int_{D_{\rho}(H)} \chi_{Q_{3r/2}(Z_0)}|\nabla v|^{2} \mu(y) dX\right\}\\
&\leq 2 C_{0}^{2} +  2 \|\nabla v\|_{L^{\infty}(Q_{5r/4}(Z))}^{2}  \leq  4 C_{0}^{2}.
\end{split}
\]
If $\rho \geq r/4$, we use the fact that $D_{\rho}(H) \subset D_{9\rho}(X_{0})$  and \eqref{max-Z-zero-bottom} to estimate 
\[
\frac{1}{\hmu({D_{\rho}(H)})}\int_{D_{\rho}(H)} \chi_{Q_{2}} |\nabla u |^{2} \mu(y) dX   \leq \frac{\hmu({D_{9\rho}(X_0)})}{\hmu({D_{\rho}(H)})} \frac{1}{\hmu({D_{9\rho}(X_0)})}\int_{D_{9\rho}(X_0)}\chi_{Q_{2}} |\nabla u |^{2} \mu(y) dX \leq M_{0} 9^{n + 2}.
\]
Combining the above estimates we obtain that for any $\rho > 0$
\[
\frac{1}{\hmu({D_{\rho}(H)})}\int_{D_{\rho}(H)} \chi_{Q_{2}} |\nabla u_{\kappa}|^{2} \mu(y) dX \leq \varpi^2.
\]
That is, $\mathcal{M}_{\mu, Q_2}(|\nabla u_{\kappa}|^{2}) (H) \leq \varpi^{2}$ as desired, and the claim follows. From this claim, the  the weak $(1,1)$ estimate of the Hardy-Littlewood maximal operator, i.e. Lemma \ref{Hardy-Max-p-p}, and \eqref{conclusion-lemma-int},  we infer that 
\[
\begin{split}
& \hmu \Big ( \{ X\in Q_{1}: \mathcal{M}_{\mu,  Q_{2}} (|\nabla u|^{2}) > \varpi^{2} \} \cap D_{r} (Z)\Big) \\
 & \leq \hmu \Big (\{X\in Q_{1}: \mathcal{M}_{\mu,  Q_{3r/2}(Z)}(|\nabla u_{\kappa} -\nabla v|^{2})>C_{0}^{2}\} \cap D_{r} (Z) \Big)\\
& \leq \frac{C(n, M_{0})}{C_{0}^{2}}\ \hmu(Q_{3r/2}(Z)) \frac{1}{\hmu(Q_{3r/2}(Z))}\int_{Q_{3r/2}(Z)} |\nabla u_{\kappa} - \nabla v|^{2} \mu(y)dX \\
&
\leq C \eta^{2}\ \hmu(D_{r}(Z)),
\end{split}
\]
where $C(n, M_{0})$ is some universal constant. From the last estimate, we observe that if we choose $\eta > 0$ sufficiently small such that $ C \eta^{2} < \epsilon$, Lemma \ref{cor-bdry-varpi} follows. 
\end{proof}

\begin{proposition}\label{contra-interior}
Suppose that $M_{0}>0$ and $\mu \in A_{2}$ such that $[\mu]_{A_{2}} \leq M_{0}$. 
There exists a constant $\varpi = \varpi(n, \Lambda, M_0)> 1$ such that for  any $\epsilon > 0 $, there exists a small constant $\delta= \delta(\epsilon, \Lambda, M_0, n)>0$ satisfying that the following holds: For a matrix $\A \in \mathcal{A}(Q_2, \Lambda, M_0, \mu)$ such that \eqref{Q-2-ellipticity-condition}-\eqref{PBMO} hold, if  $u\in W^{1, 2}(Q_2, \mu)$ is any  weak solution of \eqref{interior-MAIN}, and  for any $Z = (z, z_{n+1}) \in \overline{Q}_1$ and $r\in (0, 1/6)$ such that 
\begin{equation}\label{upper-level-estimate}
\mu( \{x\in Q_1: \mathcal{M}_{\mu,  Q_{2}}(|\nabla u|^{2})  > \varpi^{2}\}\cap D_{r}(Z)) \geq \epsilon \mu(D_{r}(Z)), 
\end{equation}
then 
\begin{equation}\label{max-f-small-contra}
D_{r}(Z)\cap Q_1\subset  \{x\in Q_{1}: \mathcal{M}_{\mu,  Q_{2}} |\nabla u|^{2}) > 1 \}\cup \{x\in Q_1: \mathcal{M}_{\mu, Q_2}\left(\left|\F/\mu\right|^{2} \right) + \mathcal{M}_{\mu, Q_2}\left(\left|f/\mu \right|^{2} \right) >\delta^{2} \}. 
\end{equation}
\end{proposition}
\begin{proof} Let $\varpi = \max\{ \varpi_1, \varpi_2\}$ and
\[ \delta = \min\Big \{\delta_1(\epsilon, \Lambda, M_0, n), \delta_2(\epsilon/(M_{0}3^{n+2}), \Lambda, M_0, n)\Big \},
\]
where $\varpi_1, \delta_1$ are defined in Lemma \ref{cor-varpi}, and $\varpi_2, \delta_2$ are  defined in Lemma \ref{cor-bdry-varpi}.  We claim that Proposition \ref{contra-interior} holds with these choices. Indeed, if $D_{2r}(Z)\subset Q_2$, then \eqref{max-f-small-contra} follows from Lemma \ref{cor-varpi} and our choice of $\delta$.  In case, $D_{2r}(Z)\cap B_2\times\{0\}  \neq \emptyset$, we assume by contradiction  that there is a point $X_{0} \in D_{r}(Z)\cap Q_1$  such that  
\begin{equation} \label{X-zero-max-cond-interior}
\mathcal{M}_{\mu, Q_{2}} ( |\nabla u|^{2})(X_0) \leq 1, \quad 
\text{and} \quad
\mathcal{M}_{\mu, Q_2}\left(\left|\F/\mu\right|^{2} \right)(X_0) + \mathcal{M}_{\mu, Q_2}\left(\left|f/\mu\right|^{2} \right)(X_0) \leq \delta^{2}. 
\end{equation}
Since $D_{2r}(Z)\cap \Big(B_2\times\{0\} \Big) \neq \emptyset$, we see that $Z_0 = (z, 0)\in D_{2r}(Z)\cap \big(B_1\times\{0\}\big)$. Moreover, $z_{n+1} \leq 2r$, and therefore
\[
X_0 \in D_{r}(Z)\cap Q_{1}\subset Q_{3r}(Z_0) \cap Q_1 \subset Q_{6r}(Z_0) \subset Q_2. 
\]
Now $X_0\in Q_{3r}(Z_0) \subset Q_2$ satisfying \eqref{X-zero-max-cond-interior}, $Q_{6r}(Z_0) \subset Q_2$, and all the conditions of Lemma \ref{cor-bdry-varpi}. From our choice of $\delta$, we can applying Lemma \ref{cor-bdry-varpi} to obtain
\[
\begin{split}
\hmu(\{X\in Q_1: \mathcal{M}_{\mu, Q_{2}}(|\nabla u|^{2})  > \varpi^{2}\}\cap D_{r}(Z) )&\leq 
\hmu(\{X\in Q_1: \mathcal{M}^{\mu}(\chi_{Q_{2}}|\nabla u|^{2})  > \varpi^{2}\}\cap D_{3r}(Z_0) )\\
& < \frac{\epsilon}{M_{0}3^{n+2}}  \hmu(D_{3r}(Z_0))\leq \epsilon \hmu(D_{r}(Z)). 
\end{split}
\]
The later contradicts \eqref{upper-level-estimate}, completing the proof. 
\end{proof}
Our next statement, which is the key in obtaining the higher gradient integrability of solution, gives the level set estimate of $\mathcal{M}_{\mu, Q_{2}}(|\nabla u|^{2})$ in terms of that of $\mathcal{M}_{\mu, Q_2}\Big( |\F/\mu|^{2} \Big)$ and $\mathcal{M}_{\mu, Q_2}\Big( |f/\mu|^{2} \Big).$

\begin{lemma} \label{interior-density-est-l}  
Suppose that $\Lambda > 0$, and  $M_{0} >0$ and let $\varpi$ be as in Proposition \ref{contra-interior}. Then, for every  $\epsilon > 0 $, there is $\delta= \delta(\epsilon, \Lambda, M_0, n) >0$ sufficiently small  with the property that  for $\mu \in A_{2}(\mathbb{R})$ such that $[\mu]_{A_{2}} \leq M_{0}$,  $\mathbb{A} \in \mathcal{A}(Q_2, \Lambda, M_0, \mu)$ such that \eqref{Q-2-ellipticity-condition}-\eqref{PBMO} hold, and for ${\bf F}/\mu, f/\mu \in L^2(Q_2, \mu)$, if $u\in W^{1, 2}(Q_{2}, \mu)$ is a weak solution to  \eqref{interior-MAIN}
and  
\[
\hmu(\{X\in Q_{1} : \mathcal{M}_{\mu, Q_{2}}(|\nabla u|^{2})  > \varpi^{2}\}) <\epsilon \hmu(D_{1}(Z)), \quad \forall \ Z \in \overline{Q}_1,
\]
then there exists a constant $\Upsilon = \Upsilon(n, \Lambda, M_{0})>0$ such that for any $k\in \mathbb{N}$ and $\epsilon_{1} = \Upsilon \, \epsilon$ we have that 
\[
\begin{split}
 & \hmu(\{X\in Q_{1}: \mathcal{M}_{\mu, Q_{2}}(|\nabla u|^{2}) > \varpi^{2k} \}) \\
  &\leq \sum_{i=1}^{k} \epsilon_{1}^{i} \hmu\left(\{X\in Q_{1}: \mathcal{M}_{\mu, Q_2}\left(\left|\F/\mu\right|^{2} \right) + \mathcal{M}_{\mu, Q_2}\left(\left|f/\mu\right|^{2} \right)>\delta^{2} \varpi^{2(k-i)} \}\right)\\
&\quad\quad+ \epsilon_{1}^{k}\hmu(\{X\in Q_{1}: \mathcal{M}_{\mu, Q_2}(|\nabla u|^{2}) > 1 \}).
\end{split}
\]
\end{lemma}
 \begin{proof} Lemma \ref{interior-density-est-l} follows from Proposition \ref{contra-interior}, the consequence of the Vitali's covering lemma, Lemma \ref{Vitali} below, and an iteration process. The proof is the same as that of Lemma \ref{interior-density-est-l-half} in the next section. We therefore skip it.
 \end{proof}
\begin{proof}[Proof of Theorem \ref{local-grad-estimate-interior}] With Lemma \ref{interior-density-est-l}, the proof of Theorem \ref{local-grad-estimate-interior} is standard. We refer its details to the proof of Theorem \ref{local-grad-estimate-half-cylinder} in the next section.
\end{proof}
%=================
\section{Weighted $W^{1,p}$-regularity estimates on the half cylinder $Q_1^+$} \label{half-cyl-proof}
This section gives the proof of Theorem \ref{local-grad-estimate-half-cylinder}, and so we focus our study on  degenerate elliptic problem in half-cyllinder $Q_2^+$
\begin{equation} \label{half-cyllinder-eqn}
\left\{
\begin{array}{cccl} 
\textup{div}[\mathbb{A}(X) \nabla u] & = & \textup{div}({\bf F}(X)) & \quad \text{in $Q^{+}_{2}$},\\
u &= & 0 & \quad  \text{on  $T_{2}\times (0, 2)$}, \\
\lim_{y \rightarrow 0^+ }\langle \mathbb{A}(x, y) \nabla u (x, y) - {\bf F}(x,y), e_{n + 1} \rangle &= & f(x), &\quad x\in B^{+}_2,
\end{array}
\right.
\end{equation}
where $\A : Q_2^+ \rightarrow \R^{(n+1) \times (n+1)}$ is symmetric, measurable matrix satisfying the assumptions of Theorem \ref{local-grad-estimate-half-cylinder}. We need several lemmas to prove Theorem \ref{local-grad-estimate-half-cylinder}. 

The following two lemmas are  versions of Lemma \ref{cor-varpi} and Lemma \ref{cor-bdry-varpi} for half cylinders, whose proof can be done in a similar way. 
\begin{lemma} \label{cor-varpi-half-inter} Suppose that $M_{0}>0$ and $\mu \in A_{2}(\R)$ such that $[\mu]_{A_{2}} \leq M_{0}$. 
There exists a constant $\varpi_3(n, \Lambda, M_0)~>~1$ such that for  any $\epsilon > 0 $, there exists a small constant $\delta_3(\epsilon, \Lambda, M_0, n)>0$ with the property that for every $\delta \in (0, \delta_3], \varpi \geq \varpi_3$, and every $\A \in \mathcal{B}(Q_2^+,  \Lambda, M_0, \mu)$ such that \eqref{Q-2-plus-ellipticity-condition}-\eqref{PBMO-Q-2-plus} hold,  if $u\in W^{1, 2}(Q_2^+, \mu)$ is a weak solution to 
\[
\textup{div}[\mathbb{A} \nabla u] = \textup{div}({\bf F})\quad \text{in} \quad Q_{2}^+,
\]
and if $Z= (z, z_{n+1})\in Q_1^+$, $r>0$ so that $D_{2r}(Z) \subset Q_2^+$, 
\begin{equation*}
D_{r}(Z)\cap \{X\in Q_1^+: \M_{\mu, Q_{2}^+}(|\nabla u|^{2}) \leq 1 \}\cap \{X\in Q_1^+: \M_{\mu, Q_2^+}\left(\left|\F/\mu\right|^{2}\right) \leq \delta^{2} \} \neq \emptyset, 
\end{equation*}
then 
\[
\hmu\Big (\{X\in Q_1^+: \M_{\mu, Q_{2}^+}(|\nabla u|^{2})  > \varpi^{2}\}\cap D_{r}(Z) \Big ) < \epsilon \hmu\Big(D_{r}(Z)\Big). 
\]
\end{lemma}
%=========================
\begin{lemma} \label{cor-bdry-varpi-half-bottom}
Suppose that $\Lambda>0, M_{0}>0$. 
There exists a constant $\varpi_4(n, \Lambda, M_0)> 1$ such that for  any $\epsilon > 0 $, there exists a small constant $\delta_4(\epsilon, \Lambda, M_0, n)>0$  with the property that for every $\delta \in (0, \delta_4]$, $\varpi \geq \varpi_4$, $\mu \in A_{2}(\R)$ such that $[\mu]_{A_{2}} \leq M_{0}$, and for $\A \in \mathcal{B}(Q_2^+, \Lambda, M_0, \mu)$ such that \eqref{Q-2-plus-ellipticity-condition}-\eqref{PBMO-Q-2-plus} hold, if $u\in W^{1, 2}(Q_2^+, \mu)$ is any  weak solution of \eqref{half-cyllinder-eqn}, and if $Z_0= (z, 0)\in \overline{B}_{1}^+\times\{0\}$ and some $r>0$ such that $Q_{2r}(Z_0)\subset Q_2^+$, and 
\begin{equation*} 
D_{r}(Z_0)\cap \{X\in Q_1^+: \mathcal{M}_{\mu, Q_{2}^+} (|\nabla u|^{2}) \leq 1 \}\cap \{X\in Q_1^+: \mathcal{M}_{\mu, Q_2^+}\left(\left|\F/\mu\right|^{2} \right) +\mathcal{M}_{\mu,Q_2^+}\left(\left|f/\mu\right|^{2}\right)\leq \delta^{2} \} \neq \emptyset, 
\end{equation*}
then 
\[
\hmu(\{X\in Q_1^+: \mathcal{M}_{\mu, Q_{2}^+}(|\nabla u|^{2})  > \varpi^{2}\}\cap D_{r}(Z_0) ) < \epsilon \hmu(D_{r}(Z_0)). 
\]
\end{lemma}

Our next lemma concerns the density of the level sets up to the flat side of $Q_{1}^{+}$ for the weak solutions.
\begin{lemma} \label{cor-varpi-half} Let $\Lambda>0, M_{0} \geq 1$ and $\mu \in A_2$ with $[\mu]_{A_2} \leq M_0$. There exists a constant $\varpi_5(n, \Lambda, M_0)> 1$ such that for  any $\epsilon > 0 $, there exists a small constant $\delta_5(\epsilon, \Lambda, M_0, n)>0$ with the property that for every $\delta~\in~(0, \delta_5], \varpi~\geq~\varpi_5$,   and for every $\mathbb{A} \in \mathcal{B}(Q_2^+, \Lambda, M_0, \mu)$ such that \eqref{Q-2-plus-ellipticity-condition}-\eqref{PBMO-Q-2-plus} hold, let $u\in W^{1, 2}(Q^{+}_2, \mu)$ be any  weak solution of 
\[
\left\{
\begin{array}{cccl}
\textup{div}[\mathbb{A} \nabla u] & = & \textup{div}({\bf F}) & \quad \text{in $Q^{+}_{2}$}\\
u & = & 0 & \quad \text{on $T_{2}\times (0, 2)$}, 
\end{array}
\right.
\]
and let $Z= (z', 0, z_{n+1})\in T_{1}\times (0, 1)$, and $r \in (0,1)$ so that $D^{+}_{2r}(Z) \subset Q_{2}^{+}$ and 
\begin{equation}\label{max-f-small-forany-r-int-half}
D_{r}(Z)\cap \{X\in Q^{+}_1: \mathcal{M}_{\mu,Q^{+}_{2}}(|\nabla u|^{2}) \leq 1 \}\cap \{X\in Q^{+}_1: \mathcal{M}_{\mu, Q_2^+}\left(\left|\F/\mu\right|^{2}\right) \leq \delta^{2} \} \neq \emptyset, 
\end{equation}
then 
\[
\hmu(\{X\in Q^{+}_1: \mathcal{M}_{\mu, Q^{+}_{2}}(|\nabla u|^{2})  > \varpi^{2}\}\cap D_{r}(Z) ) < \epsilon \hmu(D_{r}(Z)). 
\]
\end{lemma}
\begin{proof}
The proof is similar to the proof of Lemma \ref{cor-varpi}.  For given $\epsilon > 0$, let $\eta > 0$ to be determined and depending only on $\epsilon, \Lambda, M_0, n$.  Choose $\delta_5 = \delta(\eta, n, M_0, \Lambda)$ to be the number defined Proposition \ref{interior-approximation-proposition-halfCyl}. We with $\delta \in (0,\delta_5]$, from the assumption \eqref{max-f-small-forany-r-int-half}, we find $X_0 =(x_0, y_0) \in D_{r}(Z)\cap Q^{+}_1$ such that 
\[
\mathcal{M}_{\mu, Q^{+}_{2}}(|\nabla u|^{2})(X_0) \leq 1\quad \text{and}\quad \mathcal{M}_{\mu, Q_2^+}\left(\left|\F/\mu\right|^{2}\right)(X_0) \leq \delta^{2}. 
\]
This implies that for any $\tau > 0$, 
\begin{equation} \label{siding-max-X-zero}
\begin{split}
& \frac{1}{\hmu(D_{\tau}(X_0))}\int_{D_{\tau}(X_0) \cap Q_2^+}  |\nabla u|^{2} \mu (y) dX \leq 1 , \quad \text{and\, }\\
& \frac{1}{\hmu(D_{\tau}(X_0))}\int_{D_{\tau}(X_0) \cap Q_2^+} \left|\frac{{\bf F}(X)}{\mu(y)}\right|^{2} \mu (y) dX \leq \delta^{2}.
\end{split}
\end{equation}
Since $D_{2r}(Z) \subset D_{3r}(X_0)$, it follows from Lemma \ref{doubling} that 
\[
\frac{\hmu(D_{3r}(X_0))}{\hmu(D_{2r}(Z))} = \frac{|B_{3r}(x_0)| \mu(y_0-3r, y_0 + 3r) }{|B_{2r}(z)|\mu(z_{n+1}-2r, z_{n+1} + 2r ) } \leq M_{0} \left(\frac{3}{2}\right)^{n + 2}. 
\]
It then follows from \eqref{siding-max-X-zero} that 
\[
\frac{1}{\hmu(D_{2r}(Z))}\int_{D_{2r}(Z) \cap Q_2^+} |\nabla u|^{2} \mu (y) dX \leq  \frac{\hmu(D_{3r}(X_0))}{\hmu(D_{2r}(Z))} \frac{1}{\hmu(D_{3r}(X_0))}\int_{D_{3r}(X_0) \cap  Q_2^+} |\nabla u|^{2} \mu (y) dX \leq M_{0} \left(\frac{3}{2}\right)^{n + 2}. 
\]
Similarly, we have the inequality 
\[
\begin{split}
& \frac{1}{\hmu(D_{2r}(Z))}\int_{D_{2r}(Z) \cap  Q_2^+}\left|\frac{{\bf F}(X)}{\mu(y)}\right|^{2} \mu (y) dX \\
& \leq  \frac{\hmu(D_{3r}(X_0))}{\hmu(D_{2r}(Z))} \frac{1}{\hmu(D_{3r}(X_0))}\int_{D_{3r}(X_0) \cap  Q_2^+}\left|\frac{{\bf F}(X)}{\mu(y)}\right|^{2} \mu (y) dX \leq \delta^{2}M_{0} \left(\frac{3}{2}\right)^{n + 2}. 
\end{split}
\]
From Definition \ref{class-B}, and since $D_{3r/2}^+(Z) \subset Q_2^+$,  we notice that 
\[
\frac{1}{\hmu(D_{3r/2}(Z) )}\int_{D_{3r/2}^+(Z) }|\A(X) - \langle \A\rangle_{B_{3r/2}^+(z)}(y)|^{2}\mu^{-1}(y) dX  \leq \delta^{2}. %, \quad \wei{\A}_{B_{3r/2}^+(z)} \in \mathbb{M}_{Z, r}^+(\Lambda, M_0, \mu).
\]
By some suitable scaling, we see that all the assumption of Proposition \ref{interior-approximation-proposition-halfCyl}, are satisfied.  As a consequence there exists a $v \in W^{1,2}(D^{+}_{3r/2}(Z), \mu)$ such that 
\begin{equation} \label{conclusion-lemma-int-half}
\frac{1}{\hmu(D_{3r/2}(Z) )}\int_{D^{+}_{3r/2}(Z) } |\nabla u - \nabla v|^{2} \mu (y) dX  < \eta^{2}M_{0} \left(\frac{3}{2}\right)^{n + 2}, \quad  \|\nabla v\|_{L^{\infty}(D^{+}_{5r/4}(Z))} \leq C_{0} = C(n,M_{0}, \Lambda).
\end{equation}
Now, let $\varpi_5 = [\max\{M_{0} 9^{n + 2}, 4C^{2}_{0} \}]^{1/2}$. Then with $\varpi \geq \varpi_5$, we  claim that  
\[
\{
X\in Q^{+}_{1}: \mathcal{M}_{\mu, Q^{+}_{2}}(|\nabla u_{\kappa}|^{2}) > \varpi^{2}
\}\cap D_{r} (Z) \subset \{X\in Q^{+}_1:  \mathcal{M}_{\mu,  D^{+}_{3r/2}(Z)}(|\nabla u_{\kappa} -\nabla v|^{2})>C_{0}^{2}\}\cap D_{r}(Z). \]
To prove the claim, let $H\in D_{r}(Z)\cap Q^{+}_1$, such that
\[ 
\mathcal{M}_{\mu, Q^{+}_{3r/2}(Z)}(|\nabla u_{\kappa} -\nabla v|^{2})(H)\leq C_{0}^{2}.
\]
For $\rho < r/4$, we see that $D_{\rho}(H) \subset D_{5r/4}(Z) \subset D_{3r/2}(Z) \cap Q_2^+$, and so 
\[
\begin{split}
& \frac{1}{\hmu({D_{\rho}(H)})}\int_{D_{\rho}(H) \cap Q_2^+}|\nabla u_{\kappa}|^{2} \mu(y) dX \\
&\leq \frac{2}{\hmu({D_{\rho}(X)})} \left\{\int_{D_{\rho}(H)} \chi_{D^{+}_{3r/2}(Z)}|\nabla u_{\kappa} - \nabla v|^{2} \mu(y) dX + \int_{D_{\rho}(H)} \chi_{D^{+}_{5r/4}(Z)}|\nabla v|^{2} \mu(y) dX\right\}\\
&\leq 2 C_{0}^{2} +  2 \|\nabla v\|_{L^{\infty}(Q_{5r/4}(Z))}^{2}  \leq  4 C_{0}^{2}.
\end{split}
\]
If $\rho \geq r/4$, we use the fact that $D_{\rho}(H) \subset D_{9\rho}(X_{0})$  to estimate 
\[
\frac{1}{\hmu({D_{\rho}(H)})}\int_{D_{\rho}(H)\cap Q_2^+} |\nabla u_{\kappa}|^{2} \mu(y) dX   \leq \frac{\hmu({D_{9\rho}(X_0)})}{\hmu({D_{\rho}(H)})} \frac{1}{\hmu({D_{9\rho}(X_0)})}\int_{D_{9\rho}(X_0)\cap Q_2^+}|\nabla u_{\kappa}|^{2} \mu(y) dX \leq M_{0} 9^{n + 2}.
\]
Combining the above estimates we obtain
\[
\frac{1}{\hmu({D_{\rho}(H)})}\int_{D_{\rho}(H)} \chi_{Q^{+}_{2}} |\nabla u_{\kappa}|^{2} \mu(y) dX \leq \varpi^2, \quad \forall \rho >0.
\]
That is, $\mathcal{M}_{\mu, Q^{+}_2}(|\nabla u_{\kappa}|^{2}) (H) \leq \varpi^{2}$ as desired. 
Finally, by combining the above inclusion,  the weak $1-1$ estimates for the Hardy-Littlewood maximal operator, Lemma \ref{Hardy-Max-p-p}, and \eqref{conclusion-lemma-int-half}, we have 
\[
\begin{split}
& \hmu \Big( \Big\{ X\in Q^{+}_{1}: \mathcal{M}_{\mu, Q^{+}_{2}}(|\nabla u|^{2}) > \varpi^{2}\Big \} \cap D_{r} (Z)\Big )\\
& \leq \hmu \Big( \Big\{X\in Q^{+}_{1}: \mathcal{M}_{\mu, Q_{3r/2}}(|\nabla u_{\kappa} -\nabla v|^{2})>C_{0}^{2}\Big \} \cap D_{r} (Z) \Big)\\
& \leq \frac{C(n, M_{0})}{C_{0}^{2}}\ \hmu(D_{3r/2}(Z)) \frac{1}{\hmu(D_{3r/2}(Z))}\int_{D^{+}_{3r/2}(Z)} |\nabla u_{\kappa} - \nabla v|^{2} \mu(y)dX \\
&\leq C  \eta^2 \ \hmu(D_{3r/2}(Z)),
\end{split}
\]
where $C= C(n, M_{0}, \Lambda)$.  From the last estimate, we observe that if we choose $\eta > 0$ sufficiently small such that $C \eta^{2} < \epsilon$, Lemma \ref{cor-varpi-half} follows. 
\end{proof}
\noindent
Our next result is about the density of sets near the corner.
%======================
\begin{lemma}\label{cor-varpi-bdry-half}
Suppose that $M_{0}\geq 1, \Lambda >0$. There exists a constant $\varpi_6 = \varpi_6(n, \Lambda, M_0)> 1$ such that for  any $\epsilon > 0$, there exists a small constant $\delta_6 =\delta_6(\epsilon, \Lambda, M_0, n)>0$ with the property that for every $\delta \in (0, \delta_6]$, $\varpi \geq \varpi_6$, for every $\mu \in A_2$ with $[\mu]_{A_2} \leq M_0$, and for every $\mathbb{A} \in \mathcal{B}(Q_2^+, \Lambda, M_0, \mu)$ such that \eqref{Q-2-plus-ellipticity-condition}-\eqref{PBMO-Q-2-plus} hold, let $u\in W^{1, 2}(Q^{+}_2, \mu)$ be any  weak solution to \eqref{half-cyllinder-eqn},  if $Z= (z, 0, 0)\in T_{1}\times\{0 \} $, and $r\in (0,1)$ so that $Q_{2r}^{+}(Z) \subset Q_{2}^{+}$ and 
\begin{equation*}
D_{r}(Z)\cap  \{X\in Q^{+}_1: \mathcal{M}_{\mu, Q^{+}_{2}}(|\nabla u|^{2}) \leq 1 \}\cap \{X\in Q^{+}_1: \mathcal{M}_{\mu, Q_2^+}\left(\left|\F/\mu \right|^{2} \right) + \mathcal{M}_{\mu, Q_2^+}\left(\left|f/\mu\right|^{2} \right) \leq \delta^{2} \} \neq \emptyset, 
\end{equation*}
then 
\[
\hmu(\{X\in Q^{+}_1: \mathcal{M}_{\mu, Q^{+}_{2}} |\nabla u|^{2})  > \varpi^{2}\}\cap D_{r}(Z) ) < \epsilon \hmu(D_{r}(Z)). 
\]

\end{lemma}
\begin{proof} The proof is exactly the same as that of Lemma  \ref{cor-varpi-half}, where we use Proposition \ref{bottom-approximation-proposition-sittingCy} instead of Proposition \ref{interior-approximation-proposition-halfCyl}. 
\end{proof}
\noindent
%================
We now combine Lemma \ref{cor-varpi-half-inter}, Lemma \ref{cor-bdry-varpi-half-bottom}, Lemma \ref{cor-varpi-half}, and  Lemma \ref{cor-varpi-bdry-half} to obtain the following  important result.
%==========
\begin{proposition}\label{contra-interior-half} Let $\Lambda>0, M_{0}\geq 1$ be fixed.  There exists a constant $\varpi = \varpi(n, \Lambda, M_0)> 1$ such that for  any $\epsilon > 0 $, there exists a small constant $\delta= \delta(\epsilon, \Lambda, M_0, n)>0$ so that the following statement holds:For every $\mu \in A_2(\R)$ with $[\mu]_{A_2} \leq M_0$, and for every $\mathbb{A} \in \mathcal{B}(Q_2^+, \Lambda, M_0, \mu)$ such that \eqref{Q-2-plus-ellipticity-condition}-\eqref{PBMO-Q-2-plus} hold,  if $u\in W^{1, 2}(Q_2, \mu)$ is any  weak solution to \eqref{half-cyllinder-eqn}, and if $Z\in \overline{Q}_1^{+}$ and $0<r< 1/24$ so that 
\begin{equation}\label{uplevel-set-condition}
\mu( \{X\in Q^{+}_1: \mathcal{M}_{\mu, Q^{+}_{2}}(|\nabla u|^{2})  > \varpi^{2}\}\cap D_{r}(Z)) \geq \epsilon \mu(D_{r}(Z)), 
\end{equation}
then 
\begin{equation}\label{max-f-small-contra-half}
D_{r}(Z)\cap Q_{1}^{+}\subset  \{X\in Q^{+}_{1}: \mathcal{M}_{\mu, Q^{+}_{2}} (|\nabla u|^{2}) > 1 \}\cup \{X\in Q^{+}_1: \mathcal{M}_{\mu, Q_2^+}\left(\left|\F/\mu\right|^{2} \right) + \mathcal{M}_{\mu, Q_2^+}\left(\left|f/\mu\right|^{2}\right) >\delta^{2} \}. 
\end{equation}
\end{proposition}

\begin{proof} Let us choose $\varpi = \max\{\varpi_3, \varpi_4, \varpi_5, \varpi_6\}$, and
\[
\begin{split}
\delta  = & \min\Big \{\delta_3 (\epsilon, \Lambda, M_0, n), \ \delta_4(3^{-n-2}M_0^{-1} \epsilon,  \Lambda, M_0,n),\ \delta_5(3^{-n-2}M_0^{-1} \epsilon, \Lambda, M_0, n), \\
& \quad \quad  \delta_6( 12^{-n-2}M_0^{-1} \epsilon, \Lambda, M_0, n), \delta_6(4^{-n-2} M_0^{-1} \epsilon, \Lambda, M_0, n) \Big\},
\end{split}
\]
where $\varpi_i,$ and $\delta_i$ for $i=3, \cdots, 6$ are positive numbers obtained in Lemmas \ref{cor-varpi-half-inter}-\ref{cor-varpi-bdry-half} respectively.
We prove Proposition \ref{contra-interior-half} with this choice of $\delta$ and $\varpi$. We write $Z = (z', z_{n}, z_{n+1}) \in \overline{Q}_{1}^{+}$ and $z = (z', z_n) \in \overline{B}_1^+$.  Because $Z \in \overline{Q}_1^+$ and $r < 1/24$, $D_{2r}(Z)$ could possibly intersects with $Q_2^+$ either on its flat side, its base, or both. Hence, we only have the following four cases:
\begin{itemize}
\item[\textup{(i)}] $D_{2r}(Z)\cap \Big(T_{2}\times[0, 2] \Big) = \emptyset$, and $D_{2r}(Z)\cap \Big(B_{2}^+\times \{0\} \Big)= \emptyset$,
\item[\textup{(ii)}] $D_{2r}(Z)\cap \Big(T_{2}\times[0, 2] \Big) \neq \emptyset$, and $D_{2r}(Z)\cap \Big(B_{2}^+\times \{0\} \Big)= \emptyset$,
\item[\textup{(iii)}] $D_{2r}(Z)\cap \Big( T_{2}\times[0, 2]\Big) = \emptyset$, and $D_{2r}(Z)\cap \Big( B_{2}^+ \times \{0\} \Big)\neq \emptyset$,
\item[\textup{(iv)}] $D_{2r}(Z)\cap \Big( T_{2}\times[0, 2]\Big) \neq \emptyset$, and $D_{2r}(Z)\cap \Big( B_{2}^+\times \{0\} \Big)\neq \emptyset$.
\end{itemize}
Observe that in case (i), $D_{2r}(Z)\subset Q_{2}^{+}$, then the conclusion \eqref{max-f-small-contra-half} follows from Lemma \ref{cor-varpi-half-inter}, and our choice that $\delta <\delta_3 (\epsilon, \Lambda, M_0, n)$. To see that \eqref{max-f-small-contra-half} still holds in cases of (ii) - (iv), we will argue by contradiction.  Assume  that there is $X_0 \in D_{r}(Z)\cap Q^{+}_1$ but $X_0$ is not in the set on the right hand side of \eqref{max-f-small-contra-half}. Then, it follows that 
\begin{equation} \label{X-zero-half-cyl-def}
\mathcal{M}_{\mu, Q^{+}_{2}} (|\nabla u|^{2})(X_0) \leq 1,\quad \text{and,}\quad \mathcal{M}_{\mu, Q_2^+}\left(\left|\F/\mu\right|^{2} \right)(X_0) + \mathcal{M}_{\mu, Q_2}\left(\left|f/\mu\right|^{2} \right) (X_0) \leq \delta^{2}. 
\end{equation}
Let us first consider the case (ii) when $D_{2r}(Z)\subset Q_{2}$, and $D_{2r}(Z)\cap (T_{2}\times[0, 2]) \neq \emptyset$. It follows that $Z_{1} = (z', 0,  z^{n+1}) \in D_{2r}(Z)\cap (T_{1}\times[0, 2])$.  
It is then clear that 
 \[
X_{0}\in  D_{r}(Z)\cap Q_{1}^{+} \subset D_{3r}(Z_{1})\cap Q_{1}^{+}. 
 \]
 Now if $D_{6r}^{+} (Z_{1})\subset Q_{2}^{+}$, because $\delta \leq \delta_5(3^{-n-2} M_0^{-1}\epsilon, \Lambda, M_0, n)$, and \eqref{X-zero-half-cyl-def}, we see that all  the conditions for Lemma \ref{cor-varpi-half} are satisfied. Hence, we apply this lemma to obtain that 
 \[
 \begin{split}
 \hmu( \{x\in Q^{+}_1: \mathcal{M}_{\mu, Q^{+}_{2}}(|\nabla u|^{2})  > \varpi^{2}\}\cap D_{r}(Z))
 &\leq \hmu( \{x\in Q^{+}_1: \mathcal{M}_{\mu, Q^{+}_{2}}(|\nabla u|^{2})  > \varpi^{2}\}\cap D_{3r}(Z_{1}))\\
  &< \frac{\epsilon}{M_{0}3^{n+2}} \hmu(D_{3r}(Z_{1}))\leq \epsilon \hmu(D_{r}(Z)),
 \end{split}
 \]
 where we have used Lemma \ref{doubling}  in the last inequality. This estimate contradicts condition \eqref{uplevel-set-condition}. On the other hand, if $D_{6r}^{+} (Z_{1})\cap \Big( B_{2}\times\{0\} \Big) \neq \emptyset$, then we  consider the point $Z_{2} = (z', 0, 0)\in \overline{D_{6r}^{+}} (Z_{1})$. It follows from the assumption $r <1/24$ that 
 \begin{equation} \label{Z-2-Q-12r}
 X_{0}\in  D_{r}(Z)\cap Q_{1}^{+}\subset Q_{12r}^{+} (Z_{2}) \subset Q^{+}_{24r}(Z_{2})\subset Q_{2}^{+}. 
 \end{equation}
Now all the conditions of now Lemma \ref{cor-varpi-bdry-half} are satisfied, and $\delta < \delta_6(12^{-n-2} M_0^{-1} \epsilon , \Lambda, M_0, n)$. Then we apply Lemma \ref{cor-varpi-bdry-half}, and Lemma \ref{doubling} to obtain the estimate
\[
 \begin{split}
 \hmu( \{x\in Q^{+}_1: \mathcal{M}_{\mu, Q^{+}_{2}}(|\nabla u|^{2})  > \varpi^{2}\}\cap D_{r}(Z))
 &\leq \hmu( \{x\in Q^{+}_1: \mathcal{M}_{\mu, Q^{+}_{2}}(|\nabla u|^{2})  > \varpi^{2}\}\cap D_{12r}(Z_{2}))\\
  &<\frac{\epsilon}{M_{0}12^{n+2}} \hmu(D_{12r}(Z_{2}))\leq \epsilon \hmu(D_{r}(Z)). 
 \end{split}
 \]
 Again, this also contradicts condition \eqref{uplevel-set-condition}. Therefore,  the proof with case (ii) is completed.  
 
 Let us now consider case (iii). In this case,  we see that $Z_3  = (z, 0) \in D_{2r}(Z)\cap \Big( B_{2}\times \{0\} \Big)$, and
\[
X_{0} \in D_{r}(Z)\cap Q^{+}_{1} \subset Q_{3r}(Z_{3}). 
\] 
As before, we need to consider the cases if $Q_{6r}(Z_3) \cap (T_2 \times [0,2] )$ empty or not. If $Q_{6r}(Z_3) \cap (T_2 \times [0,2] )= \emptyset$, then $Q_{6r}(Z_3) \subset Q_2^+$. As before, since $\delta < \delta_4(3^{-n-2} M_0^{-1} \epsilon , \Lambda, M_0, n)$, we can apply Lemma \ref{cor-bdry-varpi-half-bottom} to obtain a contradiction to \eqref{uplevel-set-condition}.
On the other hand, if $Q_{6r}(Z_3) \cap (T_2 \times [0,2] ) \neq \emptyset$, then 
$Z_2 = (z', 0, 0) \in Q_{6r}(Z_3) \cap (T_2 \times [0,2] )$. Therefore, \eqref{Z-2-Q-12r} holds again, and we can apply Lemma \ref{cor-varpi-bdry-half} to obtain a contradiction.

It remains to consider the case (iv). The proof is very similar to the previous cases but much simpler. In this case,
\[
X_0 \in D_r(Z) \cap Q_1^+ \subset Q_{4r}(Z_2)^+ \subset Q_{8r}(Z_2)^+ \subset Q_2^+.
\]
Since $\delta < \delta_6(4^{-n-2} M_0^{-1}\epsilon, \Lambda, M_0, n)$, we can apply Lemma \ref{cor-varpi-bdry-half} as before to obtain the contradiction to \eqref{uplevel-set-condition}. The proof is now complete.
\end{proof}
Our next statement, which is the key in obtaining the higher gradient integrability of solution, gives the level set estimate of $\mathcal{M}_{\mu, Q_{2}}(|\nabla u|^{2})$ in terms of that of $\mathcal{M}_{\mu, Q_2^+}\left(\left|\F/\mu \right|^{2} \right)$ and $\mathcal{M}_{\mu, Q_2^+}\left(\left|f/\mu\right|^{2}\right).$ The proof relies on what is commonly called the ''growing ink spots lemma'' that has been developed by Krylov-Safonov is based on a Vitali-type covering lemma stated below. 
The proof of this lemma can be found in \cite[Lemma 3.8]{MP-1}, see also \cite[Lemma 5.4]{Byun-P}. 
\begin{lemma}\label{Vitali}
Suppose that $M_{0}>0$ and let $\mu$ be an $A_{p}$ weight for some $p>1$ such that $[\mu]_{A_{p}} \leq M_{0}$. 
Let $0< \rho_{0}<1$ be a fixed number and $F \subset E \subset Q_1$ be measurable sets for which there exists $\epsilon, \rho_{0}\in (0, 1/4)$ such that 
\begin{itemize}
\item[(i)]  $\mu(F) < \epsilon \mu(D_{1}(Z)) $ for all $Z\in \overline{Q}_{1}$, and 
\item[(ii)] for all $X\in Q_{1}$ and $\rho \in (0, \rho_{0}]$, if $\mu (F\cap D_{\rho}(X)) \geq \epsilon \mu(D_{\rho}(X))$, 
then $D_{\rho}(X)\cap Q_{1} \subset E. $
\end{itemize}  
Then there is some constant $\mathcal{K} = \mathcal{K}(n, p,M_{0}) > 0$ so that the following estimate holds 
\[
\mu(F) \leq \epsilon \mathcal{K} \,\mu(E). 
\] 
\end{lemma}
\begin{remark} \label{Vitali-Q_1+} Lemma \ref{Vitali} still holds if we replace $Q_1$ by $Q_1^+ = B_1^+ \times (0,1)$.
\end{remark}

%=======================
\begin{lemma} \label{interior-density-est-l-half}  
Suppose that $\Lambda >0, M_{0} \geq 1$, and let $\varpi$ be as in Proposition \ref{contra-interior-half}. Then, for every  $\epsilon > 0 $, there is a small $\delta= \delta(\epsilon, \Lambda, M_0, n) > 0$  with the property that for every $\mu \in A_2$ with $[\mu]_{A_2} \leq M_0$, and for every $\mathbb{A} \in \mathcal{B}(Q_2^+, \Lambda, M_0, \mu)$ such that \eqref{Q-2-plus-ellipticity-condition}-\eqref{PBMO-Q-2-plus} hold, and for ${\bf F}/\mu, f/\mu \in L^2(Q^{+}_2, \mu)$, if $u\in W^{1, 2}(Q^{+}_{2}, \mu)$ is a weak solution to \eqref{half-cyllinder-eqn}
and  if
\[
\hmu(\{X\in Q^{+}_{1} : \mathcal{M}_{\mu, Q^{+}_{2}} (|\nabla u|^{2})  > \varpi^{2}\}) <\epsilon \hmu(D_{1}(Z)), \quad \forall \ Z \in \overline{Q}^{+}_1,
\]
then there exists a constant $\mathcal{K} = \mathcal{K}(n, M_{0})>0$ such that for any $k\in \mathbb{N}$ and $\epsilon_{1} = \mathcal{K} \, \epsilon$ we have that 
\[
\begin{split}
 \hmu(\{X\in Q^{+}_{1}: \mathcal{M}_{\mu, Q^{+}_{2}}(|\nabla u|^{2}) > \varpi^{2k} \}) &\leq \sum_{i=1}^{k} \epsilon_{1}^{i} \hmu\left(\{X\in Q^{+}_{1}: \mathcal{M}_{\mu, Q_2^+}\left(\left|\F/\mu\right|^{2}\right) + \mathcal{M}_{\mu, Q_2^+}\left(\left|f/\mu\right|^{2} \right)>\delta^{2} \varpi^{2(k-i)} \}\right)\\
&\quad\quad+ \epsilon_{1}^{k}\hmu(\{X\in Q^{+}_{1}: \mathcal{M}_{\mu, Q^{+}_{2}} (|\nabla u|^{2}) > 1 \}).
\end{split}
\]
\end{lemma}
 \begin{proof} Let $\delta = \delta(\epsilon, \Lambda, M_0, n)$ be defined in Proposition \ref{contra-interior-half}. We will use induction to prove the lemma. For the case $k=1$, we are going to apply Lemma \ref{Vitali} with Remark \ref{Vitali-Q_1+}, by taking 
 \[E_{1}= \{X\in Q^{+}_{1}: \mathcal{M}_{\mu, Q^{+}_{2}}(|\nabla u|^{2}) > \varpi^{2} \}
 \] 
and 
\[
E_{2} = \left\{X\in Q^{+}_{1}: \mathcal{M}_{\mu, Q_2^+}\left(\left|\F/\mu\right|^{2} \right) + \mathcal{M}_{\mu, Q_2}\left(\left|f/\mu\right|^{2} \right)>\delta^{2}  \}\cup \{X\in  Q^{+}_1: \mathcal{M}_{\mu, Q^{+}_{2}} (|\nabla u|^{2}) > 1 \right\}. 
\]
Clearly, $E_{1} \subset E_{2} \subset Q^{+}_1$. Moreover, by the assumption,  $\hmu(E_{1}) < \epsilon  \hmu(D_{1}(Z))$, for all $Z \in \overline{Q}^{+}_1$. Also for any $Z\in \overline{Q}^{+}_{1}$ and $\rho\in (0,  1/24) $, if $\hmu (E_{1} \cap D_{\rho}(Z))\geq \epsilon \mu((D_{\rho}(Z))$, then   by Proposition \ref{contra-interior-half} we have that 
\[
D_{\rho}(Z)\cap Q^{+}_{1}\subset E_{2}.
\]
Hence,  all the conditions of  Lemma \ref{Vitali} are satisfied and hence, by Remark \ref{Vitali-Q_1+}, 
\[
\hmu(E_{2}) \leq \epsilon_{1} \hmu(E_{2}).
\]
That proves the case when $k=1$. Assume it is true for $k$. We will show the statement for $k+1$. We normalize by dividing by $\varpi$ as  $u_{\varpi} = u/\varpi$, ${\bf F}_{\varpi} = {\bf F}/\varpi$ and $f_{\varpi}$, and we see that since $\varpi > 1$ we have 
\[
\begin{split}
\hmu(\{X\in Q^{+}_{1}: \mathcal{M}_{\mu, Q^{+}_{2}}(|\nabla u_{\varpi}|^{2}) > \varpi^{2} \}) &= \hmu(\{X\in Q^{+}_{1}: \mathcal{M}_{\mu, Q^{+}_{2}}(|\nabla u|^{2}) > \varpi^{4} \})\\
& \leq \hmu(\{X\in Q^{+}_{1}: \mathcal{M}_{\mu, Q^{+}_{2}}(|\nabla u|^{2}) > \varpi^{2} \}) \leq \epsilon \hmu(D_{1}(Z)), \quad \forall Z \in \overline{Q}^{+}_1.
\end{split}
\]
Therefore, by induction assumption, it follows that 
\[
\begin{split}
 \hmu(\{X\in Q^{+}_{1}: \mathcal{M}_{\mu, Q^{+}_{2}}(|\nabla u|^{2}) > \varpi^{2(k+1)} \})
 &= \hmu(\{X\in Q^{+}_{1}: \mathcal{M}_{\mu, Q^{+}_{2}}(|\nabla u_{\varpi}|^{2}) > \varpi^{2k} \})\quad\quad\\
  &\leq \sum_{i=1}^{k} \epsilon_{1}^{i} \hmu\left(\{X\in Q^{+}_{1}: \mathcal{M}_{\mu, Q_2^+}\left(\left|\F_{\varpi}/\mu\right|^{2} \right) + \mathcal{M}_{\mu, Q_2^+}\left(\left|f_\varpi/\mu\right|^{2} \right)>\delta^{2} \varpi^{2(k-i)} \}\right)\\
&\quad \quad+ \epsilon_{1}^{k}\hmu(\{X\in Q^{+}_{1}: \mathcal{M}^{\mu}(\chi_{Q^{+}_{2}} |\nabla u_{\varpi}|^{2}) > 1 \})\\
& =  \sum_{i=1}^{k} \epsilon_{1}^{i} \hmu\left(\{X\in Q^{+}_{1}: \mathcal{M}_{\mu, Q_2^+}\left(\left|\F/\mu\right|^{2}\right) + \mathcal{M}_{\mu, Q_2^+}\left(\left|f/\mu\right|^{2} \right)>\delta^{2} \varpi^{2(k+1-i)} \}\right)\\
&\quad \quad+ \epsilon_{1}^{k}\hmu(\{X\in Q^{+}_{1}: \mathcal{M}_{\mu, Q^{+}_{2}} (|\nabla u|^{2}) > \varpi^2 \}). 
\end{split}
\]
Applying the case $k=1$ to the last term we obtain that 
\[
\begin{split}
 \hmu \Big(\Big\{X\in Q^{+}_{1}: \mathcal{M}_{\mu, Q^{+}_{2}}(|\nabla u|^{2}) > \varpi^{2(k+1)}\Big \}\Big)
 &\leq \sum_{i=1}^{k + 1} \epsilon_{1}^{i} \hmu \left(\Big \{X \in Q^{+}_{1}: \mathcal{M}_{\mu, Q_2^+}\left(\left|\F/\mu\right|^{2} \right) +  \mathcal{M}_{\mu, Q_2^+}\left( \left |f/\mu \right|^{2}\right) >\delta^{2} \varpi^{2(k+1-i)} \Big\} \right)\\
&\quad\quad+ \epsilon_{1}^{k+1}\hmu\Big( \Big\{X\in Q^{+}_{1}: \mathcal{M}_{\mu, Q^{+}_{2}} (|\nabla u|^{2}) > 1\Big \}\Big), 
\end{split}
\]
as desired. 
 \end{proof}
\noindent
\begin{proof}[Proof of Theorem \ref{local-grad-estimate-half-cylinder}] Let $\varpi> 1$ be as  given in Lemma \ref{contra-interior-half} and $\epsilon_1$ be as in 
Lemma \ref{interior-density-est-l-half}. We choose $\epsilon>0$ sufficiently small and depending on $\Lambda, n, M_0, p$ so that
\[
\varpi^p \epsilon_1 < 1/2.
\]
From this choice of $\epsilon$, let $\delta = \delta (\epsilon, \Lambda, M_0, n)$ be  defined as in Lemma \ref{interior-density-est-l-half}. It is obvious that $\delta$ depends only on $\Lambda, M_0, n$ and $p$. We now prove Theorem \ref{local-grad-estimate-half-cylinder} with this $\delta$. We first show that we can choose $N$ sufficiently large such that for $u_{N} = u/N$
\begin{equation} \label{in-ter-choice-N-half}
\hmu(\{X\in Q^{+}_{1}: \mathcal{M}_{\mu, Q^{+}_{2}} (|\nabla u_{N}|^{2}) > \varpi^{2} \})\leq \epsilon \hmu(D_{1}(Z)), \quad \forall Z \in \overline{Q}^{+}_{1}. 
\end{equation}
Our choice of such an $N$ can be made as below. Note that since $D_{1}(Z) \subset D_2$ for any $Z\in \overline{Q}^{+}_1$, by the doubling property of the $A_2$-weight, Lemma \ref{doubling}, we 
have
\[
\frac{\hmu(D_{2})}{\hmu(D_{1}(Z))} \leq M_0 \left( \frac{|D_2|}{|D_{1}(Z)|} \right)^2 =  M_0 2^{n+1}.
\]
Moreover, it follows from the weak $(1,1)$ estimate for the maximal function, Lemma \ref{Hardy-Max-p-p}, that 
\[
\hmu(\{X\in Q^{+}_{1}: \mathcal{M}_{\mu, Q^{+}_{2}} (|\nabla u_{N}|^{2}) > \varpi^{2} \}) \leq \frac{C (n, M_{0})}{N^{2} \varpi^{2}} \int_{Q^{+}_{2}}|\nabla u|^{2}\mu(y) dX.
\]
Then, by selecting  $N$ large enough that 
\[
\frac{C(n, M_{0})}{N^{2} \varpi^{2}} \int_{Q^{+}_{2}}|\nabla u|^{2}\mu(y) dX  =\epsilon\frac{\hmu(D_2)}{M_0 2^{n+1}},
\]
we obtain \eqref{in-ter-choice-N-half} as desired. Observe that by this choice and the doubling property of $\mu$, Lemma \ref{doubling}, 
\begin{equation} \label{muBzero-half}
N^{2}\hmu(Q^{+}_{1}) \leq \frac{C(n, M_{0})}{\epsilon \varpi^{2}} \|\nabla u\|^{2}_{L^{2}(Q^{+}_{2},\mu)}.  
\end{equation}
Now consider the sum 
\[
S = \sum_{k=1}^{\infty} \varpi^{pk}\hmu(\{Q^{+}_{1}: \mathcal{M}_{\mu, Q^{+}_{2}}(|\nabla u_{N}|^{2} ) (X) \geq \varpi^{2k}\}). 
\]
Applying Lemma \ref{interior-density-est-l-half} we have that 
\[
\begin{split}
S &\leq \sum_{k=1}^{\infty} \varpi^{pk} \left[\sum_{i=1}^{k } \epsilon_{1}^{i} \hmu\left(\{X\in Q^{+}_{1}: \mathcal{M}_{\mu, Q_2^+}\left(\left|\F_{N}/\mu \right|^{2} \right) + \mathcal{M}_{\mu, Q_2^+}\left(\left|f_{N}/\mu \right|^{2} \right) >\delta^{2} \varpi^{2(k-i)} \}
\right)\right]
\\
&\quad  \quad +\sum_{k=1}^{\infty} \varpi^{pk} \epsilon_{1}^{k}\hmu(\{X\in Q^{+}_{1}: \mathcal{M}_{\mu, Q^{+}_{2}}(|\nabla u_{N}|^{2} ) (x) \geq 1\}).
\end{split}. 
\]
From this and by applying the Fubini's theorem, we infer that 
\[
\begin{split}
S &\leq \sum_{j=1}^{\infty}( \varpi^{p} \epsilon_{1})^{j} \left[\sum_{k=j}^{\infty } \varpi^{p(k-j)} \hmu\left(\{X\in Q^{+}_{1}: \mathcal{M}_{\mu. Q_2+}\left(\left|\F_{N}/\mu\right|^{2} \right) + \mathcal{M}_{\mu, Q_2^+}\left(\left|f/\mu \right|^{2} \right)>\delta^{2} \varpi^{2(k-j)} \}
\right)\right]
\\
&\quad  \quad +\sum_{k=1}^{\infty} (\varpi^{p} \epsilon_{1})^{k}\mu(\{X\in Q^{+}_{1}: \mathcal{M}_{\mu, Q_2}(|\nabla u_{N}|^{2} ) (X) \geq 1\})\\
&\leq C\left(\left\| \mathcal{M}_{\mu, Q_2^+}\left( \left|\F_{N}/\mu\right|^{2}\right)\right\|^{p/2}_{L^{p/2}(Q^{+}_{1}, \mu)}  + \left\| \mathcal{M}_{\mu,Q_2^+}\left( \left|f_{N}/\mu \right|^{2}\right)\right\|^{p/2}_{L^{p/2}(Q^{+}_{1}, \mu)} + \|\nabla u_{N}\|^{2}_{L^{2}(Q^{+}_{2}, \mu)}\right) \sum_{k=1}^{\infty} (\varpi^{p} \epsilon_{1})^{k},
\end{split}
\]
where we have applied the weak $(1,1)$ estimate of the Hardy-Littlewood maximal function, Lemma \ref{Hardy-Max-p-p}. From the choice of 
$\epsilon$, we obtain that 
\[
\begin{split}
S &\leq C\left(\left\| \mathcal{M}_{\mu, Q_2^+}\left( \left|\F_{N}/\mu \right|^{2}\right)\right\|^{p/2}_{L^{p/2}(Q^{+}_{1}, \mu)}  + \left\| \mathcal{M}_{\mu, Q_2^+}\left( \left| f_{N}/\mu \right|^{2}\right)\right\|^{p/2}_{L^{p/2}(Q^{+}_{1}, \mu)} +\|\nabla u_{N}\|^{2}_{L^{2}(Q^{+}_{2}, \mu)}\right)\\
& \leq C\left(\left\|\frac{{\bf F}_{N}}{\mu}\right\|^{p}_{L^{p}(Q^{+}_{2}, \mu)}  +  \left\|\frac{f_{N}}{\mu}\right\|^{p}_{L^{p}(Q^{+}_{2}, \mu)} +\|\nabla u_{N}\|^{2}_{L^{2}(Q^{+}_{2}, \mu)}\right),
\end{split}
\]
where we have applied the strong $(p, p)$ estimate for the Hardy-Littlewood maximal  operator $\mathcal{M}_{\mu}$, Lemma \ref{Hardy-Max-p-p}.  Now using Lemma \ref{measuretheory-lp} below, we obtain 
\[
\|\nabla u_{N} \|_{L^{p}(Q^{+}_{1}, \mu)}^{p} \leq C \|\mathcal{M}_{\mu}(\chi_{Q^{+}_{2}}|\nabla u_{N}|^{2})\|^{p/2}_{L^{p/2}(Q^{+}_{1}, \mu)} \leq   C (S + \hmu(Q^{+}_{1})),
\]
and therefore multiplying by $N^{p}$ and applying \eqref{muBzero-half} we obtain that 
\[
\begin{split}
\|\nabla u \|_{L^{p}(Q^{+}_{1}, \mu)}^{p} &\leq C\left(\left\|\frac{{\bf F}}{\mu}\right\|^{p}_{L^{p}(Q^{+}_{2}, \mu)}  + \left\|\frac{f_{N}}{\mu}\right\|^{p}_{L^{p}(Q^{+}_{2}, \mu)}+ N^{p}\hmu(Q^{+}_{1})\right)\\
& \leq C\left(\left\|\frac{{\bf F}}{\mu}\right\|^{p}_{L^{p}(Q^{+}_{2}, \mu)}  + \left\|\frac{f_{N}}{\mu}\right\|^{p}_{L^{p}(Q^{+}_{2}, \mu)}+ \|\nabla u\|_{L^{2}(Q^{+}_{2},\mu)}^{p} \hmu(Q^{+}_{1})^{1 - \frac{p}{2}}\right) .
\end{split}
\]
This last estimate completes the proof. \end{proof}
We now state a standard result that is used in the above proof. 
\begin{lemma} \label{measuretheory-lp}
Assume that $g\geq 0$ is a measurable function in a bounded subset $U\subset \mathbb{R}^{n+1}$. Let $\theta>0$ and $\varpi>1$ be given constants. If $\mu$ is a weight in  $L^{1}_{loc}(\mathbb{R}^{n+1})$, then for any $1\leq p < \infty$ 
\[
g\in L^{p}(U,\mu) \Leftrightarrow S:= \sum_{j\geq 1} \varpi^{pj}\mu(\{x\in U: g(x)>\theta \varpi^{j}\}) < \infty. 
\]
Moreover, there exists a constant $C>0$ such that 
\[
C^{-1} S \leq \|g\|^{p}_{L^{p}(U,\mu)} \leq C (\mu(U) + S), 
\]
where $C$ depends only on $\theta, \varpi$ and $p$. 
\end{lemma} \noindent
%=============================

\section{Lipchitz estimates for weak solutions of homogeneous equations with degenerate coefficients} \label{Lipchits-est-section}
The primary objective of this section is to demonstrate that coefficients of the form given in \eqref{Example-class-A} belong to $\mathcal{A}(Q_{2}, \Lambda, M_{0}, \mu)$ and coefficients of the form given in \eqref{Example-class-B} belong to $\mathcal{B}(Q_{2}^{+}, \Lambda, M_{0}, \mu)$. To that end, if $\A$ is of the form given in \eqref{Example-class-A}, we will have to demonstrate that  Lipchitz estimates for weak solutions of the type stated in items (i) and (ii) of Definition \ref{class-A}  with coefficients $\langle\A\rangle_{B_{r}(x_0)}(y)$ are possible. The same will be shown for the class given in  \eqref{Example-class-B}.  
Note that for the class  of coefficients given in \eqref{Example-class-A} all $x$-averages are of the form
\begin{equation}\label{Ave-class-A}
\mathbb{A}_{0}(y) = \mu(y)\begin{bmatrix} \B_{0}(y)&0\\
0&b_{0}(y)
\end{bmatrix} 
\end{equation}
where $\begin{bmatrix} \B_{0}(y)&0\\
0&b_{0}(y)\end{bmatrix} 
$ is a measurable  $(n+1)\times (n+1)$ matrix that is uniformly elliptic in $\mathbb{R}^{n+1}$, 
while for the class  coefficients given in  \eqref{Example-class-B} all $x$-averages are of the form  
\begin{equation}\label{Ave-class-B}
\tilde{\A}_{0}(y) = \mu(y)\begin{bmatrix} \tilde{\B}_{0}&0\\
0&1 
\end{bmatrix} 
\end{equation}
where $\begin{bmatrix} \tilde{\B}_{0} &0\\
0&1 
\end{bmatrix} 
$ is an elliptic $(n+1)\times (n+1)$ constant matrix. In the next two subsections we will prove Lipchitz estimates for weak solutions of equations associated with coefficients in \eqref{Ave-class-A} and \eqref{Ave-class-B}. For future reference as well as their independent interest, we will state and prove  statements related to Lipchitz estimates for weak solutions  that are slightly general than what we need for the proof of Theorem \ref{extension-theorem}.  
\begin{lemma} \label{Ex-1} Let $\tau \in (0,1)$,  $\Lambda_0>0$, and $\mu \in A_2(\R)$.  Let  $\mathbb{A}_{0} (y)$ be a symmetric measurable $(n+1)\times (n+1)$ matrix of the form \eqref{Ave-class-A} such that 
\begin{equation*} 
\Lambda ^{-1} \mu(y) |\xi|^2 \leq \wei{\mathbb{A}_{0}(y)\xi, \xi} \leq \Lambda  \mu(y) |\xi|^2, \quad \forall \ \xi \in \mathbb{R}^{n+1}, \text{and for a. e. $y \in (0,\tau)$},
\end{equation*}
Suppose also that there is some constant $C_0 >0$ independent on $\tau$ so that  
\begin{equation} \label{growth-mu}
\int_{0}^{\tau y} \mu(s) ds \leq C_0 \tau \mu(y \tau),\quad \text{for a..e} \quad y \in (0,1).
\end{equation}
Then, for every weak solution $v \in W^{1,2}(Q_{\tau}, \mu)$ of 
\begin{equation} \label{v-eqn-example-1}
\left\{
\begin{array}{cccl}
\textup{div}[\A_0(y) \nabla v(X)] & =& 0,  &\quad    X= (x,y) \in Q_{\tau}, \\
\displaystyle{\lim_{y \rightarrow 0^+ }\wei{\A_0(y) \nabla v(x,y), e_{n+1}} } & =& 0, & \quad  x \ \in \ B_{\tau},
\end{array}
\right.
\end{equation}
and every $\theta \in (0,1)$, there is $C= C(n, \Lambda_0, C_0, [\mu]_{A_q}, \theta)$ such that
\[
\norm{\nabla v}_{L^\infty(Q_{\theta \tau})} \leq C  \left [\frac{1}{\hmu(Q_{\tau})}\int_{Q_{\tau}} |\nabla v|^2 \mu(y) dX\right]^{1/2}.
\]
\end{lemma}
\begin{proof} Observe that by the dilation
\[
v_\tau (X) := \frac{1}{\tau} v(\tau X), \quad \mu_\tau(y) := \mu(y \tau), \quad \A_{0, \tau}(y) := \A_0(y \tau)
\]
and assumption \eqref{growth-mu} 
we can assume without loss of generality that $\tau =1$. We write $M_0 = [\mu]_{A_2} = [\mu_\tau]_{A_2}$. For simplicity, we also assume that $\theta =1/2$. For each $k \in \mathbb{N}$, by using difference quotients, and induction on $k$, for example see \cite[Chapter 6.3]{Evans-book}, we can prove that $w: = \nabla_x^k v \in W^{1,2}(Q_{r}, \mu)$ with the estimate
\begin{equation} \label{nabla-x-L-2}
\int_{Q_{r}} |\nabla w|^2 \mu(y) dX \leq C(\Lambda, r, k) \int_{Q_1} |\nabla v|^2 \mu(y) dX, \quad \forall \ 0 < r <1.
\end{equation}
Then, for every $\phi \in C^\infty(\overline{Q_1})$ vanishing near $\partial Q_1 \setminus (\overline{B_1} \times \{0\})$, we can use $\nabla_x^k \phi$ as a test function for \eqref{v-eqn-example-1} to obtain
\[
\int_0^1\int_{B_1} \wei{\A_0(y) \nabla v(x,y), \nabla \nabla_x^k \phi(x,y)} dxdy =0.
\]
From \eqref{nabla-x-L-2}, and with the note that $\nabla^{l}_x \phi(\cdot,y) =0$ on $\partial B_1$ for all $y \in (0,1)$ and all $l \in \N$, we can apply the integration by parts in $x$ to find that
\[
\int_0^1\int_{B_1} \wei{\A_0(y) \nabla w(x,y), \nabla  \phi(x,y)} dxdy =0.
\]
Hence, $w$ is a weak solution of 
\begin{equation} \label{D-k-v-eqn}
\left\{
\begin{array}{cccl}
\text{div}[\A_0(y) \nabla w] & = &  0,  & \quad \text{in} \ Q_{1}, \\
\displaystyle{\lim_{y \rightarrow 0^+} \wei {\A_0(y) \nabla w (x,y), e_{n+1}}} & = & 0, & \quad \text{on} \  B_1 \times \{0\}.
\end{array} \right.
\end{equation}
It then follows from the regularity theory estimates in \cite{Fabes-1, Fabes} that 
\begin{equation} \label{gratient-x-v-estimate}
 \norm{ \nabla_x^k v}_{L^\infty(Q_{4/5})} \leq C(n, M_0) \left[ \frac{1}{\hmu(Q_1)} \int_{Q_1} |\nabla v(x,y)|^2 \mu(y) dX \right]^{1/2}, \quad k =0, 1, 2.
\end{equation}
It remains to show that $\partial_y v \in L^\infty(Q_{1/2})$. To this end, let us fix a non-negative function $\phi_0 \in C_0^\infty(B_1(0))$, with $0 \leq \phi_0 \leq 1$ and
\[
\phi_0(x) =1 \quad \text{on}\quad  B_{1/2}(0), \quad \text{and} \quad |\nabla \phi_0| \leq 2.
\]
For a fixed $x_0 \in B_{1/2}$ and $r_0 >0$ such that  $B_{r_0} (x_0) \subset B_{4/5}$, we denote
\[
\phi_{x_0, r}(x) = \phi_0((x-x_0)/r), \quad 0 < r < r_0.
\]
Note that
\begin{equation} \label{phi-x-0}
\phi_{x_0, r} =1 \quad \text{on} \quad B_{r/2}(x_0), \quad \text{and} \quad |\nabla \phi_{x_0, r}| \leq \frac{2}{r} \quad \text{in} \quad B_{r}(x_0).
\end{equation}
Also, let $\chi \in C_0^\infty((-\infty, 1))$ be non-negative. Using $\phi_{x_0, r}(x)\chi(y)$ as a test function for the equation of $v$, we infer that
\[
 \int_{Q_1} \chi(y)   \wei{\B_0(y) \nabla _{x} v, \nabla_{x}  \phi_{x_0, r}(x)} 
 = -  \int_{Q_1} \phi_{x_0, r}(x)\chi'(y) a (y) \partial_y v  dxdy 
\] 
Because $ \nabla_x^2 v \in W^{1,2}(Q_{4/5}, \mu)$, we can use the integration by parts in $x$ rewrite the above identity as
\begin{equation} \label{Ex-1-first-es}
\begin{split}
 \sum_{i, j =1}^n \int_{Q_1} \chi(y)   \phi_{x_0, r}(x)  \B_0^{i, j}(y) \partial_{x_i x_j} v dxdy =   \int_{Q_1} \phi_{x_0, r}(x)\chi'(y) a (y) \partial_y v dxdy.
\end{split}
\end{equation}
Let us now denote
\[
\begin{split}
g(y)  = -   \sum_{i, j =1}^n \frac{1}{|B_r(x_0)|} \int_{B_1}\phi_{x_0, r}(x)\B_0^{i, j}(y) \partial_{x_i x_j} v(x,y)  dx,  \quad 
h (y)   = \frac{1}{|B_r(x_0)|}\int_{B_1} \phi_{x_0, r}(x) a (y) \partial_y v(x,y) dx
\end{split}
\]
Observe that both $h, g \in L^2((0,1), \mu^{-1})$, and moreover from \eqref{Ex-1-first-es}, 
\begin{equation} \label{h-g-dy}
\int_0^1 h(y) \chi'(y) dy =  -\int_0^1  g(y) \chi(y)  dy, \quad \forall \chi \in C_0^\infty((-\infty, 1)).
\end{equation}
This particularly implies that $h \in W^{1,2}((0,1), \mu^{-1})$, and 
\[
 h'(y) = g(y), \quad \text{for a. e.} \quad y \in (0,1).
\]
Also, observe that it follows from Lemma \ref{Reverse-H} that  $W^{1,2}((0,1), \mu^{-1})\hookrightarrow W^{1, 1+\beta}(0,1)$ for some $\beta >0$ depending only on $[\mu^{-1}]_{A_2} = [\mu]_{A_2}$, and since 
$W^{1, 1+\beta}(0,1) \hookrightarrow C([0,1])$, we see that $h \in C([0,1])$. From this,  by choosing $\chi\in C_{0}^{\infty}(-\infty, 1)$ such that $\chi(0) =1$, we obtain from \eqref{h-g-dy} that
\[
h(1)\chi(1) - h(0)\chi(0) = \int_{0}^{1} (h\chi)'dy = \int_{0}^{1}h'\chi dy + \int_{0}^{1}h\chi' dy = \int_{0}^{1}g \chi dy + \int_{0}^{1}h\chi' dy = 0.
\]
As a consequence, $h(0) = 0$. Now, we can write
\begin{equation} \label{h-g-relation}
h(y)  = \int_0 ^y g( s) ds.
\end{equation}
We next estimate this integration of $g(y)$. From \eqref{gratient-x-v-estimate}, and the definition of $g$, it follows that
\[
\begin{split}
\left|\int_0^y g(s) ds \right| 
& \leq C(n, \Lambda_0)  \norm{\nabla^2_x v}_{L^\infty(Q_{4/5})}  \int_0^y \mu(s) ds \\
& \leq C(n, \Lambda_0, M_0) \left[\frac{1}{\hmu(Q_1)}\int_{Q_1} |\nabla v(x,y)|^2 \mu(y) dx dy \right]^{1/2}   \int_0^y \mu(s) ds.
\end{split}
\]
This estimate, and the assumption \eqref{growth-mu} imply
\[
|h(y)| \leq  C(n, \Lambda_0, M_0) \left[\frac{1}{\hmu(Q_1)} \int_{Q_1} |\nabla v(x,y)|^2 \mu(y) dx dy \right]^{1/2} \mu(y).
\]
Hence, from the definition of $h$, and if $x_0$ is a Lebesgue point of $a (y) \partial_y v (x, y)$, we see that for a.e. $y \in (0, 4/5)$, 
\[
\begin{split}
\lim_{r \rightarrow 0^+} |h(y)| & = \Big| a (y) \partial_y v (x_0, y) \Big| \leq  C(n, \Lambda_0, M_0) \left[\frac{1}{\hmu(Q_1)} \int_{Q_1} |\nabla v(x,y)|^2 \mu(y) dx dy \right]^{1/2} \mu(y).
\end{split}
\]
Observe that from the ellipticity condition,
\[
\begin{split}
\mu(y) |\partial_y v(x_0, y)| & \leq \Lambda |a(y) \partial_y v(x_0,y)| \\
& \leq C(n, \Lambda, M_0)  \left[\frac{1}{\hmu(Q_1)}\int_{Q_1} |\nabla v(x,y)|^2 \mu(y) dx dy \right]^{1/2}  \mu(y).
\end{split}
\]
Since $\mu \in \A_2(\R)$, we particular see that $|\{y \in (0,1): \mu(y) =0\}| =0$. Therefore,
\[
|\partial_y v(x_0, y)|\leq C(n, \Lambda, M_0)  \left[\frac{1}{\hmu(Q_1)}\int_{Q_1} |\nabla v(x,y)|^2 \mu(y) dx dy \right]^{1/2}, \quad \text{for a.e} \quad y \in (0, 4/5).
\]
This and \eqref{growth-mu} imply the following
\begin{equation} \label{gradient-y-v-est}
|\partial_y v(x_0, y)| \leq C(n, \Lambda, M_0) \left[\frac{1}{\hmu(Q_1)}\int_{Q_1} |\nabla v(x,y)|^2 \mu(y) dx dy \right]^{1/2}, \quad \text{for a. e.} \quad (x_0, y)  \in B_{1/2}\times (0, 4/5).
\end{equation}
The estimate \eqref{gratient-x-v-estimate} together with \eqref{gradient-y-v-est} complete the proof of the lemma.
\end{proof}

To state our next result, we impose additional condition for the weight $\mu$.  The weight $\mu$ could be degenerate or singular at $y =0$. However, for $y>0$, we require that there is $C_1 >0$ such that for a.e. $y \in (0,1)$ and $0 < \tau < y/2$
\begin{equation} \label{sup-inf-mu}
\begin{split}
& \inf_{s\in\Gamma_{3\tau/2}(y)} \mu(s), \ \text{and} \  \sup_{s\in \Gamma_{3\tau/2}(y)} \mu(s)\ \text{ both exist}, \quad \inf_{s\in\Gamma_{3\tau/2}(y)} \mu(s) > 0, \quad \text{and} \\
&  \sup_{y \in (0, 1), \tau < y/2} \frac{\sup_{s \in \Gamma_{3\tau/2}(y)} \mu(s)}{\inf_{s \in \Gamma_{3\tau/2}(y)} \mu(s)} \leq C_1.
\end{split}\end{equation}
where as before we use the notation that   
$\Gamma_{3\tau/2}(y) = (y -3\tau/2, y+ 3\tau/2)$. 
 
\begin{remark} \label{Ex-2-remark} If $\mu(y) = |y|^{\alpha}$ for some $\alpha \in \R$, then \eqref{sup-inf-mu} holds. Indeed, consider the case $\alpha <0$, then for each $y_0 \in (0,1)$ and each $\tau \in (0,y_0/2)$, we see that
\[
\inf_{s \in \Gamma_{3\tau/2}(y_0)} \mu (s)= (y_0 + 3\tau/2)^\alpha >0, \quad \sup_{s \in \Gamma_{3\tau/2}(y_0)} \mu(s) = (y_0 -3\tau/2)^\alpha <\infty,
\]
and
\[
\displaystyle{\frac{\sup_{s \in \Gamma_{3\tau/2}(y_0)} \mu(s)}{\inf_{s\in \Gamma_{3\tau/2}(y_0)} \mu(s)} \leq \frac{1}{7^\alpha}}, \quad \forall \ y_0 \in (0,1) \quad \tau \in (0, y_0/2).
\]
When $\alpha >0$, the proof is also similar.
\end{remark}

%============
\begin{lemma} \label{Ex-2} Let $\mathbb{A}_{0}: (0,2) \rightarrow \R^{(n+1)\times (n+1)}$ be a symmetric measurable matrix. Assume also that there are positive number $\Lambda $ and a weight function $\mu$ such that 
\[
\Lambda ^{-1} \mu(y) |\xi|^2 \leq \wei{\mathbb{A}_{0}(y)\xi, \xi} \leq \Lambda \mu(y) |\xi|^2, \quad \forall \ \xi \in \mathbb{R}^{n+1}, \text{and for a. e.  $y \in (0, 2)$}. 
\]
Moreover, assume that \eqref{sup-inf-mu} holds for the weight function $\mu$. Then for every $\theta \in (0,1)$, there is  $C =C(\Lambda, \theta, C_1, n)$ such that the following statement holds: For every $X_0 =(x_0, y_0) \in \overline{Q}_1$  and $\tau \in (0,1)$ so that $y_0 -2\tau >0$, if $v \in W^{1,2}(D_{3\tau/2}(X_0), \mu)$  is a weak solution of 
\begin{equation} \label{cyllinder-D-tau-eqn}
\textup{div}[\A_0(y) \nabla v]  = 0, \quad \text{in} \quad D_{3\tau/2}(X_0),
\end{equation}
then 
\begin{equation} \label{Ex-2-est}
\norm{\nabla v}_{L^\infty(D_{3 \theta \tau/2}(X_0))} \leq C  \left [\frac{1}{\hmu(D_{3\tau/2}(X_0))}\int_{D_{3\tau/2}(X_0)} |\nabla v(x,y)|^2 \mu(y) dxdy \right]^{1/2}.
\end{equation}
\end{lemma}
\begin{proof}  Observe that
\begin{equation} \label{lambda-1-2}
0< \lambda_0:= \inf_{s \in \Gamma_{3\tau/2}(y_0)}\mu(s) \leq \mu(y) \leq  \sup_{s \in \Gamma_{3\tau/2}(y_0)} \mu (s)< \infty, \quad 
\text{for a.e.} \ y \in \Gamma_{3\tau/2}(y_0) = (y_0 -3\tau/2, y_0 + 3\tau/2).
\end{equation}
By dividing the coefficient matrix $\A_{0}$ and the weight $\mu$ by $\lambda_0$, let us introduce the notation 
\[ 
\tilde{\A}_0 = \A_0/\lambda_0, \quad \tilde{\mu} = \mu/\lambda_0.
\]
It follows then that the new weight $\tilde{\mu}$ will satisfy the estimate 
\[
1 \leq \tilde{\mu}(y) \leq  C_0, \quad \forall y\in (y_0 -3\tau/2, y_0 + 3\tau/2). 
\]
Hence, 
\begin{equation} \label{uniform-ellip-A-til}
\Lambda^{-1} |\xi|^2  \leq \wei{\tilde{\A}_0(y) \xi, \xi} \leq C_1 \Lambda |\xi|^2, \quad \text{for a. e.} \quad y \in (y_0 -3\tau/2, y_0 + 3\tau/2).
\end{equation}
On the other hand, from \eqref{lambda-1-2}, and since $v \in W^{1,2}(D_{3\tau/2}(X_0), \mu)$, we observe $v \in W^{1,2}(D_{3\tau/2}(X_0))$. From this and  with \eqref{uniform-ellip-A-til}, we infer that $v \in W^{1,2}(D_{3\tau/2}(X_0))$  is a weak solution  of the uniformly elliptic equation
\begin{equation} \label{v-uniformly-ellip-eqn}
\text{div}[\tilde{\A}_0(y) \nabla v] = 0, \quad \text{on} \quad D_{3\tau/2}(X_0).
\end{equation}
Up to a translation, we can now applying Lemma \ref{reg-v-q} for the uniformly elliptic equation \eqref{v-uniformly-ellip-eqn}. Then, we conclude that there is $C = C(\Lambda, C_1, n, \theta)$ such that
\begin{equation} \label{q-gradient-v}
\norm{\nabla v}_{L^\infty(D_{3\theta \tau/2}(X_0))} \leq C\left\{\fint_{D_{3\tau/2}(X_0)} |\nabla v(x,y)|^2 dxdy\right\}^{1/2}.
\end{equation}
Observe that from \eqref{lambda-1-2}, we have
\[
\fint_{D_{3\tau/2}(X_0)} |\nabla v|^2 dxdy \leq 
\frac{C_1}{\hmu(D_{3\tau/2}(X_0))}\int_{D_{3\tau/2}(X_0)} |\nabla v(x,y)|^2 \mu(y) dxdy.
\]
This last estimate and \eqref{q-gradient-v} together imply
\[
\norm{\nabla v}_{L^\infty(D_{3\theta \tau/2 }(X_0))} \leq C(n,\Lambda, C_1, \theta)  \left [\frac{1}{\hmu(D_{3\tau/2}(X_0))}\int_{D_{3\tau/2}(X_0)} |\nabla v(x,y)|^2 \mu(y) dxdy \right]^{1/2}.
\]
This is the desired estimate, and the proof of the lemma is now complete.
\end{proof}

We reiterate that the last two lemmas show that the class of coefficients of the form given in \eqref{class-A} belong to $\mathcal{A}(Q_2, \Lambda, M_0, \mu)$ where $\mu(y) = |y|^{\alpha}$, for $\alpha \in (-1, 1)$. These weights are in $A_{2}(\mathbb{R})$, satisfy both \eqref{growth-mu} and  \eqref{sup-inf-mu}. 

%============
%============
\begin{lemma} \label{Ex-1-half} Let $\tau \in (0,1)$,  $\Lambda_0>0$, and $\mu \in A_2(\R)$ such that \eqref{growth-mu} holds.  Also, let  $\tilde{\A}_{0} (y)$ be a symmetric measurable $(n+1)\times (n+1)$ matrix of the form given in \eqref{Ave-class-B}
Assume also that
\begin{equation*} 
\Lambda ^{-1} \mu(y) |\xi|^2 \leq \wei{\tilde{\A}_{0}(y)\xi, \xi} \leq \Lambda  \mu(y) |\xi|^2, \quad \forall \ \xi \in \mathbb{R}^{n+1}, \text{and for a. e. $y \in (0,\tau)$},
\end{equation*}
Then, for every weak solution $v \in W^{1,2}(Q_{\tau}^+, \mu)$ of 
\begin{equation} \label{v-eqn-example-1-half}
\left\{
\begin{array}{cccl}
\textup{div}[\tilde{\A}_0(y) \nabla v(X)] & =& 0,  &\quad    X= (x,y) \in Q_{\tau}^+, \\
 v & = & 0, & \quad \text{on} \quad T_\tau \times (0, \tau), \\
\displaystyle{\lim_{y \rightarrow 0^+ }\wei{\tilde{\A}_0(y) \nabla v(x,y), e_{n+1}} } & =& 0, & \quad  x \ \in \ B_{\tau}^+,
\end{array}
\right.
\end{equation}
and every $\theta \in (0,1)$, there is $C= C(n, \Lambda_0, C_0, [\mu]_{A_q}, \theta)$ such that
\[
\norm{\nabla v}_{L^\infty(Q^+_{\theta \tau})} \leq C  \left [\frac{1}{\hmu(Q_{\tau})}\int_{Q_{\tau}^+} |\nabla v|^2 \mu(y) dX\right]^{1/2}.
\]
\end{lemma}
\begin{proof} Observe that $\tilde{\B}_0$ is a uniformly elliptic constant matrix. By rotation in $x$-variable, we can assume that $\tilde{\B}_0$ is diagonal. Then, let $\tilde{v}$ be the odd reflection with respect to $x_n$ of $v$, i.e.
\[
\tilde{v}(x', x_n ,y) = \left\{
\begin{array}{ll}
v(x', x_n,y), & \quad \text{for} \quad X= (x', x_n, y) \in Q_\tau, x_n >0, \\
-v(x', -x_n,y), & \quad \text{for} \quad X= (x', x_n, y) \in Q_\tau, x_n <0.
\end{array} \right.
\] 
Standard calculations will show that $\tilde{v}$ is  a weak solution of
\begin{equation} \label{v-eqn-example-1-half-reflect}
\left\{
\begin{array}{cccl}
\textup{div}[\tilde{\A}_0(y) \nabla \tilde{v} (X)] & =& 0,  &\quad    X= (x,y) \in Q_{\tau},\\
\displaystyle{\lim_{y \rightarrow 0^+ }\wei{\tilde{\A}_0(y) \nabla \tilde{v}(x,y), e_{n+1}} } & =& 0, & \quad  x \ \in \ B_{\tau},
\end{array}
\right.
\end{equation}
Lemma \ref{Ex-1-half} then follows from Lemma \ref{Ex-1}.
\end{proof}
Our next lemma is the Lipchitz regularity estimate of weak solutions in half cylinders.
\begin{lemma} \label{Ex-2-half} 
For a given a weight function $\mu$ that satisfies \eqref{sup-inf-mu}, let $\tilde{\A}_{0}: (0,2) \rightarrow \R^{(n+1)\times (n+1)}$ be a symmetric measurable matrix of the form given in \eqref{Ave-class-B}. 
Assume also that there exists $\Lambda > 0 $ such that 
\[
\Lambda ^{-1} \mu(y) |\xi|^2 \leq \wei{\tilde{\A}_{0}(y)\xi, \xi} \leq \Lambda \mu(y) |\xi|^2, \quad \forall \ \xi \in \mathbb{R}^{n+1}, \text{and for a. e.  $y \in (0, 2)$}. 
\] 
Then for every $\theta \in (0,1)$, there is  $C =C(\Lambda, \theta, C_1, n)$ such that the following statement holds: For every $x_0 = (x_0', 0) \in \overline{T}_1$, $X_0 =(x_0, y_0) \in \overline{T}_1\times (0,1)$  and $\tau \in (0,1)$ so that $y_0 -2\tau >0$, if $v \in W^{1,2}(D_{3\tau/2}^+(X_0), \mu)$  is a weak solution of 
\begin{equation} \label{cyllinder-D-tau-eqn-half}
\left\{
\begin{array}{cccl}
\textup{div}[\tilde{\A}_0(y) \nabla v]  & = & 0,  &\quad \text{in} \quad D_{3\tau/2}^+(X_0),\\
 v & = & 0, & \quad \text{on} \quad T_{3\tau/2}(x_0) \times (y_0 -3\tau/2, y_0 + 3\tau/2),
\end{array} \right.
\end{equation}
then 
\begin{equation} \label{Ex-2-est-half}
\norm{\nabla v}_{L^\infty(D_{3 \theta \tau/2}^+(X_0))} \leq C  \left [\frac{1}{\hmu(D_{3\tau/2}(X_0)}\int_{D_{3\tau/2}^+(X_0)} |\nabla v(x,y)|^2 \mu(y) dxdy \right]^{1/2}.
\end{equation}
\end{lemma}
\begin{proof} The proof is exactly the same as that of Lemma \ref{Ex-2}, using Lemma \ref{reg-v-q-half-cyll} instead. We therefore skip it.
\end{proof}
\begin{remark} It is essential for our application that the constants $C$ in the estimates \eqref{Ex-2-est} and \eqref{Ex-2-est-half} of Lemma \ref{Ex-2} and Lemma \ref{Ex-2-half} are independent on the location $y_0$ and the radius $\tau$. Moreover, we emphasize that we do not require $\mu \in A_2(\R)$ in these two lemmas. Lemma \ref{Ex-2} and Lemma \ref{Ex-2-half} therefore hold with $\mu(y) = |y|^\alpha$ with $\alpha \in \R$, and the constants $C$ in  \eqref{Ex-2-est} in \eqref{Ex-2-est-half} depends only on $\Lambda, \theta, \alpha$ and $n$ for this case.
\end{remark}
%=========
\section{Proof of Theorem \ref{extension-theorem} and Theorem \ref{fractional-theorem}} \label{proof-extension-problem-sec}

This section shows that Theorem \ref{extension-theorem} follows from Theorem \ref{local-grad-estimate-interior} and Theorem \ref{local-grad-estimate-half-cylinder}. 
\begin{proof}[Proof of Theorem \ref{extension-theorem}] The proof is standard, using partition of unity, and flattening of the boundary $\partial \Omega$. First of all, observe that with $\mu(y) = |y|^{\alpha}$ and $\alpha \in (-1,1)$, then $\mu \in A_2(\R)$ with $M_0: = [\mu]_{A_2} = C(\alpha)$.  Let $\delta_0 = \delta(\Lambda, M_0, n, p)$, where $\delta(\Lambda, M_0, n, p)$ is the number defined in Theorem \ref{local-grad-estimate-interior}. Similarly, let $\hat{\delta}_0 = \delta((n+1)\Lambda, M_0, n, p)$ be the number defined in Theorem \ref{local-grad-estimate-half-cylinder}. Then, take 
\begin{equation} \label{choice-delta-extension}
\delta = \min\{\delta_0, \hat{\delta}_0\}/(2C_0)
\end{equation}
where $C_0 = C_0(\Lambda, M_0, n) >1$ will be determined. We prove Theorem \ref{extension-problem} with this choice of $\delta$.  To this end, for any number $x_0 \in \overline{\Omega}$, we consider the two cases.\\
\ \\
{\bf Case I:} If $x_0 \in \Omega$, and $r  \in (0,r_0)$, where $r_{0}>0$ is from \eqref{B-BMO-mu},  such that $B_{2r}(x_0) \subset \Omega$. Since $u \in W^{1,2}(B_{2r}(x_0) \times (0, 2), \mu)$ is a weak solution of
\[
\left\{
\begin{array}{cccl}
\textup{div}[\A(X) \nabla u(X)] & = & \textup{div}[\F], & \quad \text{in} \quad B_{2r}(x_0) \times (0,2), \\
\lim_{y \rightarrow 0^+} \wei{\A(X) \nabla u - \F(X), e_{n+1}} & =& f(x), & \quad \text{on} \quad B_{2r}(x_0).
\end{array} \right.
\]
By a suitable dilation in $x$-variable, and a translation, Lemma \ref{Ex-1}, and Lemma \ref{Ex-2}, we can see that $\A \in \mathcal{A}(B_{2r}(x_0)\times (0,2),  \Lambda, M_0, \mu)$.  Therefore, we can obtain from Theorem \ref{local-grad-estimate-interior} that
\begin{equation} \label{extensional-problem-est-1}
\begin{split} 
& \norm{\nabla u}_{L^p(B_{r}(x_0) \times (0,1), \mu)}  \leq C\left[\hmu(B_{2r}(x_0) \times (0,1))^{\frac{1}{p} -\frac{1}{2}} \norm{\nabla u}_{L^2(\Omega_2, \mu)}   + \norm{\F/\mu}_{L^p(\Omega_2, \mu)} + \norm{ f/\mu}_{L^p(\Omega_2, \mu)} \right].
\end{split}
\end{equation}
\ \\
{\bf Case II:} If $x_0 \in \partial \Omega$. Since $\partial \Omega \in C^1$, after translating, and rotating we can assume that $x_0 = 0 \in \R^n$,  and for all $r \in (0,r_0)$, 
\[
B_r \cap \Omega = \{x = (x', x_n) \in B_r: x_n > \gamma(x')\}
\]
for some $C^1$ function $\gamma$ mapping some open set in $\R^{n-1}$ containing the origin of  $\R^{n-1}$,  
and 
\[
\gamma(0) =0, \quad \nabla_{x'} \gamma(0) =0.
\]
By taking $r$ sufficiently small, which can be made uniformly by the compactness of $\Omega$, and from the continuity of $\nabla \gamma$, we can assume also that
\begin{equation} \label{gamma-gradient}
\sup_{x'\in  } |\nabla_{x'} \gamma'(x')| \leq \delta/(2C_0).
\end{equation}
Then, let us define 
\[ 
\Phi(x', x_n) = (x', x_n - \gamma(x')), \quad  \Psi(z', z_n) = \Phi^{-1}(z', z_n) = (z', z_n + \gamma(z')).
\]
Then, observe that
\begin{equation} \label{gradient-Phi-Psi}
\norm{\nabla \Phi}_{L^\infty} \leq n+1, \quad \norm{\nabla \Psi}_{L^\infty} \leq n+1.
\end{equation}
Then, choose $\rho \in (0, r/\sqrt{n+1})$ and sufficiently small, we can see that $B_{2\rho}^+ \subset \Phi(\Omega \cap B_r)$. We then define
\[
\hat{u}(Z) = u(\Psi(z), y), \quad \hat{\F}(Z) = \F(\Psi(z), y), \quad \hat{f}(z) = f(\Psi(z)), \quad Z = (z, y) \in B_{2\rho}^{+ } \times (0,2),
\]
and
\[
\hat{\A}(Z) = \mu(y) \left(\begin{matrix} \hat{\B}(z) & 0 \\ 0  &1 \end{matrix}\right), \quad \hat{\B}(z) = \nabla \Phi(\Psi(z)) \B(\Psi(z)) \nabla \Phi(\Psi(z))^*, \quad Z = (z, y) \in B_{2\rho}^+ \times (0,2).
\]
We observe that the above transformation affects the matrix only the $x$-variable and so $\hat{u} \in W^{1,2}(B_{2\rho}^+ \times (0,2), \mu)$ is a weak solution of
\begin{equation} \label{u-hat-eqn}
\left\{
\begin{array}{cccl}
\textup{div}[\hat{\A}(Z) \nabla \hat{u}(Z)] & = & \textup{div}[\hat{\F}(Z)], & \quad \text{in} \quad B_{2\rho}^+ \times (0,2), \\
\hat{u} & = & 0, & \quad \text{on} \quad T_{2\rho} \times (0,2), \\
\lim_{y \rightarrow 0^+} \wei{\hat{\A}(Z) \nabla \hat{u}(Z) - \hat{\F}(Z), e_{n+1}} & = &0, & \quad \text{a.e}  \quad z \in B_{2\rho}^+.
\end{array} \right.
\end{equation}
By some standard calculation, see for example \cite[p. 2163-2164]{LTT}, we can check and see that there is 
$C_0 = C_0(\Lambda, M_0, n)$ such that
\[
\begin{split}
& \sup_{\tau \in (0, \rho)} \sup_{(Z_0 = (z_0,y_0) \in \overline{B}_\rho^+ [0,1]} \frac{1}{\hmu(D_\tau(Z_0))} \int_{D_\tau(Z_0)} |\hat{\A}(z,y) - \wei{\hat{\A}}_{B_{\tau}(z_0) \cap B_\rho^+} |^2 \mu^{-1}(y) dzdy  \\
& \leq C_0 \left[ [\B]_{\textup{BMO}}(\Omega) + \sup_{x'} |\nabla \gamma(x')|^2 \right] < \hat{\delta}_0,
\end{split}
\]
where we have used our choice of $\delta$ in \eqref{choice-delta-extension}, and \eqref{gamma-gradient} in the above estimate.  Moreover, it follows from \eqref{gradient-Phi-Psi} that
\[
[(n+1)\Lambda]^{-1}|\xi|^2 \mu(y) \leq \wei{\hat{\A}(Z)\xi, \xi} \leq (n+1)\Lambda|\xi|^2 \mu(y), \quad
\text{for a.e.} \quad Z= (z,y) \in B_{2\rho}^+ \times (0,2), \quad \forall \ \xi \in \R^{n+1}.
\]
From the last two estimates, Lemma \ref{Ex-1}, Lemmas \ref{Ex-2}-\ref{Ex-2-half}, and Remark \ref{Ex-2-remark}, we conclude that $\hat{\A} \in \mathcal{B}(B_{2\rho}^+ \times (0,2), (n+1)\Lambda, M_0, \mu)$. From this, we can use Theorem \ref{local-grad-estimate-half-cylinder} to obtain
\[
\begin{split}
&\norm{\nabla \hat{ u}}_{L^p(B_\rho^+ \times (0,2), \mu)} \\
& \leq C\left[\hmu(B_{2\rho}^+ \times (0,2))^{\frac{1}{p} -\frac{1}{2}} \norm{\nabla \hat{u}}_{L^2(B_{2\rho}^+ \times (0,2), \mu)}   + \norm{\hat{\F}/\mu}_{L^p(B_{2\rho}^+ \times (0,2), \mu)} + \norm{\hat{f}/\mu}_{L^p(B_{2\rho}^+\times (0,2), \mu)}  \right].
\end{split}
\]
Changing back to $X = (x,y)$ variable and using the fact that  $|\nabla u(\Psi(z), y)| \leq \sqrt{n+1} |\nabla \hat{u}(z, y)|$, we then obtain
\[
\begin{split}
&\norm{\nabla u}_{L^p(\Psi(B_\rho^+) \times (0,2), \mu)} \\
& \leq C\left[\hmu(B_{2\rho} \times (0,2))^{\frac{1}{p} -\frac{1}{2}} \norm{\nabla u}_{L^2(\Omega_2, \mu)}   + \norm{\F/\mu}_{L^p(\Omega_2, \mu)}  + \norm{f/\mu}_{L^p(\Omega_2, \mu)}  \right].
\end{split}
\]
Let $\rho' = \rho/\sqrt{n+1}$. From \eqref{gradient-Phi-Psi}, it follows that  $B_{\rho'} \cap \Omega \subset \Psi(B_\rho^+)$. Therefore, 
\begin{equation} \label{boundary-est-gradient-extension}
\begin{split}
&\norm{\nabla u}_{L^p((B_{\rho'} \cap \Omega)\times (0,2), \mu)} \leq C\left[\hmu(B_{2\rho} \times (0,2))^{\frac{1}{p} -\frac{1}{2}} \norm{\nabla u}_{L^2(\Omega_2, \mu)}   + \norm{\F/\mu}_{L^p(\Omega_2, \mu)}  + \norm{f/\mu}_{L^p(\Omega_2, \mu)}  \right].
\end{split}
\end{equation}
The rest of the proof is now standard. We cover $\overline{\Omega}$ with finite number of interior balls $\{B_{r_i}(x_i)\}_{i=1,\cdots K_1}$ and boundary balls $\{B_{\rho_k}(z_k)\}_{k=1, \cdots, K_2}$, where $B_{2r_i}(x_i) \subset \Omega$, $z_k \in \partial \Omega$, $\rho_k$ is chosen as $\rho'$ above, and $K_1, K_2 \in \N$ are some numbers. Then, we use
\eqref{extensional-problem-est-1} for the balls $B_{r_i}(x_i)$, and use \eqref{boundary-est-gradient-extension} for the balls $B_{\rho_i}(z_i)$. Adding all estimates together, we obtain
\[
\begin{split}
& \norm{\nabla u}_{L^p(\Omega_1, \mu)}  \leq C\left[\hmu(B \times (0,2))^{\frac{1}{p} -\frac{1}{2}} \norm{\nabla \hat{u}}_{L^2(\Omega_2, \mu)}   + \norm{\hat{\F}/\mu}_{L^p(\Omega_2, \mu)} + \norm{\hat{f}/\mu}_{L^p(\Omega_R, \mu)}  \right],
\end{split}
\]
with some ball $B \subset \R^n$ large enough so that 
\[
\displaystyle{\cup_{i=1}^{K_1} B_{2r_i}(x_i) \cup_{i=1}^{K_2} B_{2\rho_i\sqrt{n+1}}(z_i)} \subset B.
\]
The proof is complete.
\end{proof}
\begin{proof}[Proof of Theorem \ref{fractional-theorem}]
As it has been demonstrated in \cite{Caffa-Sil, Caffa-Stinga, Capella, T-Stinga2010} that for $f\in L^{2}(\Omega)$,  if $u$ is a solution to the fractional elliptic equation $L^{s}u = f$, for $0<s<1$ in the sense defined in Section 2, then $u = Tr|_{\Omega}(U) = U(x, 0)$, where $U = U(x, y): \Omega\times (0, \infty)\to  \mathbb{R}$ solves the degenerate equation 
\[
\tag{EP} \left\{
\begin{array}{cccl}
\textrm{div}(y^{1-2s} \tilde{\A}(x)\nabla U) & = & 0,&\quad \Omega\times (0, \infty)\\
U &= &0,&\quad \partial \Omega\times (0, \infty)\\
\lim_{y\to 0^{+}}\langle y^{1-2s} \tilde{\A}(x)\nabla U, {\bf e}_{n+1}\rangle  &= & f,&\quad \Omega\times\{0\}
\end{array}
\right., \quad\quad  \text{where}\,\, \tilde{\A}(x) =  \left( \begin{matrix} \mathbb{B}(x)&0\\
0&1
\end{matrix} \right). 
\]
We may now directly apply Theorem \ref{extension-theorem} to (EP) to conclude that for $p\geq 2, $ if $\mathbb{B}$ satisfies \eqref{B-BMO-mu}, and $f\in L^{p}(\Omega)$, then $\nabla U \in L^{p}(\Omega\times (0, 1), y^{1-2s}dX)$ with the estimate 
\[
\int_{\Omega}\int_{0}^{1} |\nabla U|^{p} y^{1-2s}dy dx \leq C \left[\left( \int_{\Omega\times (0, 2)}|\nabla U|^{2} y^{1-2s}dy dx \right)^{p/2} + \int_{\Omega} \int_{0}
^{2} |f(x)|^{p} y^{(1-2s)(1-p)}dy dx\right].\]
Next we will estimate the terms in the right hand side of the above estimate. To that end, using the test function $\chi(y)^{2} U(x, y)$, where $\chi\in C_{c}^{\infty}(-\infty, 3)$,  $0\leq \chi\leq 1$ and $\chi(0) =1$,   we obtain from (EP) that for any $s\in (0, 1)$
\[
\begin{split}
\int_{\Omega} \int_{0}^{3}\langle\tilde{\A}(x)\nabla U, \nabla U\rangle \chi(y)^{2}  y^{1-2s}dy dx &=  - 2\int_{\Omega}\int_{0}^{3} U(x, y) \langle \tilde{\A}(x)\nabla U, {\bf e}_{n+1} \rangle  \chi(y)\chi'(y)  y^{1-2s}dy dx  +  \int_{\Omega}  U(x, 0) f(x)dx\\
&\leq \Lambda/2 \int_{\Omega}\int_{0}^{3} |\nabla U|^{2}\chi(y)^{2} y^{1-2s}dy dx + C(\Lambda) \int_{\Omega}\int_{0}^{3} |U|^{2} \chi'(y)^{2} y^{1-2s}dy dx \\
&+ \epsilon^{2}\int_{\Omega}  |U|^{2}(x, 0)dx + C/\epsilon^{2}\int_{\Omega} |f(x)|^{2}dx.
\end{split}
\]
where we have applied Young's inequality. 
From the above and the trace lemma, Lemma \ref{trace-zero}, it follows that  by choosing $\epsilon$ small, 
\[
\begin{split}
\int_{\Omega}\int_{0}^{3} |\nabla U|^{2}  \chi(y)^{2}  y^{1-2s}dy dx& \leq C(\Lambda) \left(\int_{\Omega}\int_{0}^{3} |U|^{2}  y^{1-2s}dy dx  + \int_{\Omega} |f(x)|^{2}dx\right)\\
&\leq C(\Lambda) \left(\int_{\Omega}|u(x)|^{2} dx  + \int_{\Omega} |f(x)|^{2}dx\right) \leq  C \int_{\Omega} |f(x)|^{2}dx,
\end{split}
\]
where the last inequality follows from \cite[Theorem 2.5]{Caffa-Stinga}. 
On the other hand, the quantity 
\[
 \int_{\Omega} \int_{0} ^{2} |f(x)|^{p} y^{(1-2s)(1-p)}dy dx = \int_{0}^{3} y^{(1-2s)(1-p)}\int_{\Omega}|f|^{p}(x)dx \leq C \int_{\Omega}  |f(x)|^{p}  dx\]
 provided $(1-2s)(1-p) +1 >0$, obtaining the  estimate  
\[
\int_{\Omega}\int_{0}^{2}|\nabla U|^{p} y^{1-2s}dy dx  \leq  C \|f\|_{L^{p}(\Omega)}^{p}. 
\]
Since $U$ vanishes on $\partial \Omega\times (0, \infty)$, then from Poincar\'e's inequality we have that 
\[
\|U\|_{W^{1, p}(\Omega\times (0, 2), y^{1-2s}dX)} \leq C \|f\|_{L^{p}}.
\] 
We now apply the characterization of traces of weighted Sobolev spaces presented in \cite[Theorem 2.8]{Nekvinda} (see also \cite{Capella,Cabre})  to conclude that $u = Tr|_{\Omega}(U)$ is in the fractional Sobolev space $W^{\alpha, p}(\Omega)$
with the estimate 
\[
\|u\|_{W^{\alpha, p}(\Omega)} \leq C \|U\|_{W^{1, p}(\Omega\times (0, 2), y^{1-2s}dX)} \leq C \|f\|_{L^{p}}
\]
where $\alpha = 1- \frac{2-2s}{p}$.  
\end{proof}
%=====================

\ \\ 
\noindent \textbf{Acknowledgement.} T. Mengesha's research is supported by NSF grants DMS-1312809 and DMS-1615726.  T. Phan's research is supported by the Simons Foundation, grant \#~354889. The authors would like to thank their colleague Abner J. Salgado for fruitful discussion on spectral fractional elliptic equations. 

%==================
\appendix
\section{Lipschitz estimates for uniformly elliptic equations}
We recall some fundamental results on Lipschitz estimates for uniformly elliptic equations. 
\begin{lemma} \label{reg-v-q} Let $\A_0: (-1,1) \rightarrow \R^{(n+1)\times (n+1)}$ be a symmetric, measurable matrix satisfying the uniform ellipticity condition: There is $\Lambda_0 >0$ such that
\[
\Lambda_0^{-1} |\xi|^2 \leq \wei{\A_0(y) \xi, \xi} \leq \Lambda_0|\xi|^2, \quad \text{for all} \quad 
\xi \in \R^{n+1}, \quad \text{for a.e.} \quad y \in (-1,1).
\]
Then, for every weak solution $v \in W^{1,2}(D_1)$ of
\begin{equation} \label{uniform-elliptic-eqn}
\textup{div}[\A_0(y) \nabla v(x,y)] = 0, \quad \text{for} \quad X= (x,y) \in D_1,
\end{equation}
and for every $\theta \in (0,1)$, there is $C = C(\Lambda_0, \theta, n) >0$ such that
\[
\norm{\nabla v}_{L^\infty(D_{\theta})} \leq C \left\{\fint_{D_1} |\nabla v|^2 dxdy \right\}^{1/2}.
\]
\end{lemma}
\noindent
\begin{proof} This lemma is known, see for example \cite{Byun-P, Kim-Krylov}. For the proof, one can find it in \cite[Lemma 3.3]{Byun-P}.
\end{proof}
\begin{lemma} \label{reg-v-q-half-cyll} Let $\A_0: (-1,1) \rightarrow \R^{(n+1)\times (n+1)}$ be a symmetric, measurable matrix 
of the form \[
\A_0(y) = a_0(y) \left( \begin{matrix} \tilde{\B}_0 & 0 \\ 0 & 1 \end{matrix} \right).
\]
where $\tilde{\B}_0$ is a symmetric  $n\times n$ constant matrix and  $a_0$ is a measurable function defined on $(-1,1)$. Assume also that there is $\Lambda_0 >0$ such that
\[
\Lambda_0^{-1} |\xi|^2 \leq \wei{\A_0(y) \xi, \xi} \leq \Lambda_0|\xi|^2, \quad \text{for all} \quad 
\xi \in \R^{n+1}, \quad \text{for a.e.} \quad y \in (-1,1).
\]
Then, for every weak solution $v \in W^{1,2}(D_1^+)$ of
\begin{equation} \label{uniform-elliptic-eqn-half}
\left\{
\begin{array}{cccl}
\textup{div}[\A_0(y) \nabla v(X)] & = & 0, & \quad \text{for} \quad X= (x,y) \in D_1^+, \\
v & = & 0, & \quad \text{on} \quad T_1 \times (-1,1),
\end{array} \right.
\end{equation}
and for every $\theta \in (0,1)$, there is $C = C(\Lambda_0, \theta, n) >0$ such that
\[
\norm{\nabla v}_{L^\infty(D_{\theta}^+)} \leq C \left\{\fint_{D_1^+} |\nabla v|^2 dxdy \right\}^{1/2}.
\]
\end{lemma}
\begin{proof}  By rotation in $x$ variable, we can assume that $\tilde{\B}_0$ is diagonal. Then, let $\tilde{v}$ be the odd reflection of $v$ with respect the $x_n$, i.e.
\[
\tilde{v}(x', x_n ,y) = \left\{
\begin{array}{ll}
v(x', x_n, y), & \quad X = (x', x_n, y) \in D_1, \quad x_n >0, \\
-v(x', -x_n, y), & \quad X = (x', x_n, y) \in D_1, \quad x_n <0.
\end{array} \right.
\]
It then follows that $\tilde{v}$ is a weak solution of
\[
\textup{div}[\A_0(y) \nabla \tilde{v}(X)]  =  0,  \quad \text{for} \quad X= (x,y) \in D_1.
\]
The desired estimate the follows from Lemma \ref{reg-v-q}.
\end{proof}

\end{document}